\documentclass[12pt]{article}
\usepackage{layout,yfonts, graphicx, amsmath, amsthm, amssymb, euscript, enumerate, enumitem, bbold, bbm, mathrsfs,easyReview}

\usepackage{caption}
\usepackage{subcaption}

\usepackage[numbers]{natbib} %cambia el estilo citar la bibliografia
\numberwithin{equation}{section}
\usepackage{color}
\definecolor{webgreen}{rgb}{0,.5,0}
\definecolor{webbrown}{rgb}{.8,0,0}
\definecolor{emphcolor}{rgb}{0.95,0.95,0.95}

\usepackage{hyperref}
\hypersetup{%
		colorlinks=true,
		linkcolor=webbrown,
		filecolor=webbrown,
		citecolor=webgreen,
		breaklinks=true}
\ifpdf \hypersetup{pdftex,
		bookmarksopen=true,
		bookmarksnumbered=true
	} \else \hypersetup{dvips} \fi

\usepackage{blindtext} %Para enumerar
\usepackage{scrextend}
\addtokomafont{labelinglabel}{\sffamily}

\usepackage[margin=1in]{geometry}
\theoremstyle{plain}
\newtheorem{teor}{Theorem}[section]
\newtheorem{defin}[teor]{Definition}
\newtheorem{prop}[teor]{Proposition}
\newtheorem{lema}[teor]{Lemma}

\theoremstyle{remark}  
\newtheorem{rem}[teor]{Remark}

\newtheoremstyle{hyp}{}{}{\itshape}{}{}{}{3pt}{}
\theoremstyle{hyp}

\DeclareMathOperator*{\infess}{inf\, ess}
\newcommand{\cn}{\mathbb{n}}
\newcommand{\hf}{\hat{f}}
\newcommand{\w}{\mathpzc{w}}
\newcommand{\vs}{\mathpzc{v}}
\newcommand{\sol}{\mathpzc{u}}
\DeclareMathAlphabet{\mathpzc}{OT1}{pzc}{m}{it}

\DeclareMathOperator{\deri}{D}
\DeclareMathOperator{\dist}{dist}

\DeclareMathOperator{\Lp}{L}
\DeclareMathOperator{\trans}{T}
\DeclareMathOperator{\comp}{c} 
\DeclareMathOperator{\loc}{loc}
\DeclareMathOperator{\hol}{C}
\DeclareMathOperator*{\argmin}{arg\, min}
\DeclareMathOperator{\tr}{tr}
\DeclareMathOperator{\sob}{W}

\DeclareMathOperator{\expo}{e}
\DeclareMathOperator{\sop}{supp}

\newcommand{\E}{\mathbbm{E}}
\newcommand{\tmt}{\mathpzc{t}}
\newcommand{\tms}{\mathpzc{s}}
\newcommand{\set}{\mathcal{O}}
\newcommand{\dif}{\mathcal{L}}

\newcommand{\R}{\mathbbm{R}}
\newcommand{\N}{\mathbbm{N}}
\newcommand{\uno}{\mathbbm{1}}
\newcommand{\der}{\mathrm{d}}
\newcommand{\F}{\mathbbm{F}}
\newcommand{\Pro}{\mathbbm{P}}
\newcommand{\BE}{\begin{equation}}
	
	\newcommand{\EE}{\end{equation}}
\newcommand {\BA}{\begin{align}}
	\newcommand{\EA}{\end{align}}
\newcommand{\eqdef}{\raisebox{0.4pt}{\ensuremath{:}}\hspace*{-1mm}=}
\newcommand{\defeq}{=\hspace*{-1mm}\raisebox{0.4pt}{\ensuremath{:}}}

\title{\Large{\bf  {A mixed singular/switching  control problem with terminal cost for   modulated  diffusion processes}}\footnote{\textbf{Funding:}  {This} study has been funded by the Russian Academic Excellence Project `5-100'.}}
\author{\large{\bf Mark Kelbert}\\ 
	\large{\bf Harold A. Moreno-Franco}\footnote{Corresponding author: hmoreno@hse.ru}\\
	\small{\it Department of Statistics and Data Analysis}\\ 
	\small{\it Laboratory of Stochastic Analysis and its Applications}\\ 
	\small{\it National Research University Higher School of Economics, Moscow, Russian Federation}}%\\
\date{}
\begin{document}	
%\listoftodos %contenido de las notas	
%\layout
%\fontfamily{ptm}\selectfont
\maketitle
\vspace{-0.5cm}
\begin{abstract}
	\noindent {In this paper, we study the regularity of the value function associated with a stochastic control problem where two controls  act simultaneously on  a modulated multidimensional diffusion process. The first is a switching control modelling a random clock. Every time the random clock rings, the  generator matrix is replaced by another, resulting in a different dynamic for the finite state Markov chain  of the  modulated diffusion process. The second  is a  singular stochastic control that is  executed on the  process  within each regime.}    
	
\end{abstract}

\section{{Introduction and main results}}

The goal of this paper is to study the regularity of the value function that is associated with a mixed singular/switching stochastic control  problem for a  modulated multidimensional diffusion in a bounded domain. Within a regime $\ell\in \mathbb{M}\eqdef\{1,2,\dots,m\}$, a singular stochastic control is executed on a multidimensional diffusion which is  modulated by a finite state Markov chain with generator matrix $Q_{\ell}\eqdef(q_{\ell}(\iota,\kappa))_{\iota,\kappa\in\mathbb{I}}$, where $\mathbb{I}\eqdef\{1,2,\dots,n\}$. {Here, the criterion  is to minimize the expected costs that the singular and switching controls produce every time that they act on the modulated diffusion process, subject to  a penalization that is produced  at the first moment that the controlled process is outside  the bounded set; for more details about it, see the subsection below.}

{The control problem presented in this work can be applied, for example, in the area of finance if we assume that the cash reserve process of a firm is governed by a modulated one-dimensional process until a ruin time. Considering a fixed family of transition matrices $\mathcal{Q}=\{Q_{\ell}\}_{\ell\in\mathbb{M}}$ and an increasing sequence of stopping times $\{\tau_{i}\}_{i\geq0}$, and according to the data observed at time $\tau_{i}$, the firm has the option of changing the transition matrix $Q_{\ell_{i-1}}$ by $Q_{\ell_{i}}$,   with a   cost $\vartheta_{\ell_{i-1},\ell_{i}}$, in such a way that the Markov chain associated with the modulated process can model the times and the states that the reserve process would be well-defined on the interval $[\tau_{\ell_{i}},\tau_{\ell_{i-1}})$. Let us define this switching control by $\varsigma=(\tau_{i}, {\ell_{i}})_{i\geq0}$. Within each regime $\ell_{i}$, the expenses of the firm, which are paid out from the reserve process, are given by a non-decreasing and right-continuous process $\zeta$. Then, under a minimization criterion, the firm wishes to find a strategy $(\varsigma^{*},\zeta^{*})$ that reduces the expected costs that the company must assume.}

{As far as we know, the existing stochastic control literature has not yet considered the problems described above, and they could be a research line of interest for both the stochastic control theory and its applications.    }

%For this purpose, we must first guarantee the classic solution $u^{\varepsilon,\delta}=(u^{\varepsilon,\delta}_{\ell,\iota})_{(\ell,\iota)\in\mathbb{M}\times\mathbb{I}}$ to a non-partial differential system (NPDE); see \eqref{NPD.1}. Then, defining $u$ as the limit of $u^{\varepsilon,\delta}$, when $(\varepsilon,\delta)$ goes to $(0,0)$,   we will verify that $u$ is the unique strong solution to the corresponding Hamilton-Jacobi-Bellman (HJB) equation of the value function mentioned above; see \eqref{esd5}. Finally by probabilistic methods, we will prove that $u$ an agrees

\subsection{Model formulation}
Let $W=\{W_{\tmt}:\tmt\geq0\}$ be a $k$-dimensional standard Brownian motion and let $I^{(\ell)}=\{I^{(\ell)}_{\tmt}:\tmt\geq0\}$, with $\ell\in\mathbb{M}$, be  a continuous-time Markov chain with finite state space $\mathbb{I}$ and generator matrix $Q_{\ell}=(q_{\ell}(\iota,\kappa))_{\iota,\kappa\in\mathbb{I}}$, i.e.,
\begin{equation}\label{q1}
	\Pro[I^{(\ell)}_{\tmt+\Delta \tmt}=\kappa|I^{(\ell)}_{\tmt}=\iota,  I^{(\ell)}_{\tms},\tms\leq\tmt]=
	\begin{cases}
		q_{\ell}(\iota,\kappa)\Delta \tmt +o(\Delta\tmt),&\text{if}\ \kappa\neq \iota,\\
		1+q_{\ell}(\iota,\kappa)\Delta\tmt+o(\Delta \tmt)&\text{if}\ \kappa=\iota.
	\end{cases}
\end{equation}
The entries of the generator matrix $Q_{\ell}$ satisfy
\begin{equation*}%\label{q2}
	\begin{split}
		&q_{\ell}(\iota,\kappa)\geq0\ \text{for}\ \iota,\kappa\in\mathbb{I},\ \text{with}\ \kappa\neq\iota,\\ 
		&q_{\ell}(\iota,\iota)=-\sum_{\kappa\in\mathbb{I}\setminus\{\iota\}}q_{\ell}(\iota,\kappa)\ \text{for}\ \iota\in\mathbb{I}.
	\end{split}
\end{equation*} 
We assume that $W , I^{(1)},\dots,I^{(m)}$ are independent and are   defined on a complete probability space  $(\Omega,\mathcal{F},\Pro)$. Let $\F=\{\mathcal{F}_{\tmt}\}_{\tmt\geq0}$ be the filtration  generated by  $W$ and $\{I^{(\ell)}\}_{\ell\in\mathbb{M}}$. 

We consider the triple $(X^{\xi,\varsigma},J^{\varsigma},I)$ as a  stochastic controlled process  that   evolves as:
\begin{equation}\label{es1}
	\begin{split}
		X^{\xi,\varsigma}_{\tmt}&=X^{\xi,\varsigma}_{\tilde{\tau}_{i}{-}} -\displaystyle\int_{\tilde{\tau}_{i}}^{\tmt}b(X^{\xi,\varsigma}_{\tms},I^{(\ell_{i})}_{\tms})\der \tms+\displaystyle\int_{\tilde{\tau}_{i}}^{\tmt}\sigma(X^{\xi,\varsigma}_{\tms},I^{(\ell_{i})}_{\tms})\der W_{\tms}-\int_{\tilde{\tau}_{i}}^{\tmt}\cn_{\tms}\der\zeta_{\tms},\\
		I_{\tmt}&=I^{(\ell_{i})}_{\tmt}\ \text{and}\  J^{\varsigma}_{\tmt}= \ell_{i}\quad\text{for}\ \tmt\in[{\tilde\tau}_{i},\tilde{\tau}_{i+1})\ \text{ and }\ i\geq0,
	\end{split}
\end{equation} 
where $X^{\xi,\varsigma}_{0-}=x_{0}\in\overline{\set}\subset\R^{d}$, $J^{\varsigma}_{0-}=\ell_{0}\in\mathbb{M}$, $I_{0}=I^{(\ell_{0})}_{0}=\iota_{0}\in\mathbb{I}$,   $\tau\eqdef\{\tmt>0:(X^{\xi,\varsigma}_{\tmt},I_{\tmt})\notin\set\times\mathbb{I}\}$, and  $\tilde{\tau}_{i}\eqdef\tau_{i}\wedge\tau$.  The parameters $b_{\iota}\eqdef b(\cdot,\iota) :\overline{\set}\longrightarrow\R^{d}$ and $\sigma_{\iota}\eqdef\sigma(\cdot,\iota) :\overline{\set}\longrightarrow\R^{d}\times\R^{k}$, with $\iota\in\mathbb{M}$ fixed,  satisfy appropriate conditions to ensure the well-definiteness of the stochastic differential equation (SDE) \eqref{es1}; {see Assumption \eqref{h2}}.   

The control process $(\xi,\varsigma)$ is in $\mathcal{U}\times\mathcal{S}$ where the singular control $\xi=(\mathbb{n},\zeta)$ belongs to the class $\mathcal{U}$ of admissible controls that satisfy
\begin{equation}\label{cont.1}
	\begin{cases}
		(\mathbb{n}_{\tmt},\zeta_{\tmt})\in\R^{d}\times\R_+,\ t\geq0,\ \text{such that}\ X^{\xi,\varsigma}_{\tmt}\in\set\,\ \tmt\in[0,\tau),\\
		(\mathbb{n},\zeta)\ \text{is adapted to the filtration}\  \F,\\
		\zeta_{0-}=0\ \text{and}\ \zeta_{\tmt}\ \text{is non-decreasing and is right continuous} \\
		\text{with left hand limits,}\ \tmt\geq0,\ \text{and }\ |\mathbb{n}_{\tmt}|=1\ {\der\zeta_{\tmt}\text{-a.s.},\ \tmt\geq0} ,
	\end{cases}
\end{equation}
and the switching   control process   $\varsigma\eqdef(\tau_{i},\ell_{i})_{i\geq0}$ belongs to the class $\mathcal{S}$ of switching regime sequences that satisfy
\begin{equation}\label{cont.2}
	\begin{cases}
		\varsigma\ \text{is a sequence of $\F$-stopping times and regimes in ${\mathbb{M}}$, i.e,}\\ 
		\text{$\varsigma =(\tau_{i},\ell_{i})_{i\geq0}$ is such that $0=\tau_{0}\leq\tau_{1}<\tau_{2}<\cdots$, $\tau_{i}\uparrow\infty$ as $i\uparrow \infty$ $\Pro$-a.s.,}\\
		\text{ and for each $i\geq0$, {$\ell_{i}$ is $\mathcal{F}_{\tau_{i}}$-measurable valued in $\mathbb{M}$.}}
	\end{cases}
\end{equation}

Given  the initial state  $( x_{0},\ell_{0},\iota_{0})\in\overline{\set}\times\mathbb{M}\times\mathbb{I}$  and the  control   $(\xi,\varsigma)\in\mathcal{U}\times\mathcal{S}$,  the \textit{functional cost} of the controlled process $(X^{\xi,\varsigma},J^{\varsigma},I)$  is defined by 
\begin{multline}\label{esd1.1.1}
	V_{\xi,\varsigma}(x_{0},\ell_{0},\iota_{0})\eqdef\E_{ x_{0},\ell_{0},\iota_{0}}\biggr[\int_{0}^{\tau}\expo^{-r(\mathpzc{t})}[ h (X^{\xi,\varsigma}_{\mathpzc{t}},I_{\tmt})\der \mathpzc{t}+ g   (X^{\xi,\varsigma}_{\mathpzc{t}-},I_{\mathpzc{t}})\circ\der\zeta_{\mathpzc{t}}]\biggr]\\
	+\sum_{i\geq0}\E_{ x_{0},\ell_{0},\iota_{0}}\big[\expo^{-r(\tau_{i+1})}\vartheta_{\ell_{i},\ell_{i+1}}\uno_{\{\tau_{i+1}<\tau\}}\big]+\E_{ x_{0},\ell_{0},\iota_{0}}[\expo^{-r(\tau)}f(X^{\xi,\varsigma}_{\tau},I_{\tau})\uno_{\{\tau<\infty\}}],
\end{multline}
where  $\E_{ x_{0},\ell_{0},\iota_{0}}$ is the expected value associated with $\Pro_{ x_{0},\ell_{0},\iota_{0}}$,    the probability law of $(X^{\xi,\varsigma},J^{\varsigma},I)$ when it starts at $( x_{0},\ell_{0},\iota_{0})$,   $r(\mathpzc{t})=\int_0^{\mathpzc{t}} c(X^{\xi,\varsigma}_{\tms},I_{\tms})d\tms$ {represents the accumulated interest rate at time $\tmt$,} and
\begin{align*}
	&\int_{0}^{\tmt}\expo^{-r( \tms)} g (X^{\xi,\varsigma}_{ \tms-},I_{\tms})\circ\der\zeta_{ \tms}\eqdef\int_{0}^{\tmt}\expo^{-r(\tms)} g   (X^{\xi,\varsigma}_{ \tms},I_{\tms})\der\zeta_{ \tms}^{\comp}\notag\\
	&\quad\qquad\qquad\qquad\qquad\qquad\quad+\sum_{ 0\leq\tms\leq\tmt}\expo^{-r(\tms)}\Delta\zeta_{ \tms}\int_{0}^{1} g    (X^{\xi,\varsigma}_{ \tms-}-\lambda \mathbb{n}_{ \tms}\Delta\zeta_{ \tms},I_{\tms})\der\lambda, %\label{eq0.1.1}
\end{align*}
with $\zeta^{\comp}$ {denoting} the continuous part of $\zeta$. {We can appreciate in \eqref{esd1.1.1} that the cost for switching regimes are represented by $\vartheta_{\ell,\ell'}$, and the terminal cost is given by $f(X^{\xi,\varsigma}_{\tau},I_{\tau})\uno_{\{\tau<\infty\}}$. Additionally, at time $\tmt $, we have  the costs $g (X^{\xi,\varsigma}_{\tmt},I_{\tmt})\circ\der\zeta_{\tmt}$ when the singular control $\xi$ is executed, and $h(X^{\xi,\varsigma}_{\tmt},I_{\tmt})$ if not.}

Under the assumption that there is  no {\it loop of zero cost} (see Eq \eqref{l1}), 
one of   {the} main  goals of this paper  is to verify that the value function
\begin{equation}\label{vf1}
	V_{\ell_{0}}( x_{0},\iota_{0})\eqdef\inf_{\xi,\varsigma}V_{\xi,\varsigma}( x_{0},\ell_{0},\iota_{0}),\ \text{for}\ ( x_{0},{\ell_{0}},\iota_{0})\in\overline{\set}\times \mathbb{M}\times\mathbb{I},
\end{equation}
is in $\hol^{0}(\overline{\set})\cap\sob^{2,\infty}_{\loc}(\set)$; see Theorem \ref{M1}. 

{The novelties  of this work, in contrast with the existing literature (see, i.e., \cite{KM2019,KM2022} and  references therein), are: 
	\begin{enumerate}
		\item[(i)]  Every time that there is a switching, the generator matrix is replaced by another,  resulting in a different dynamic for the finite state Markov chain  of the modulated multidimensional diffusion process.
		\item[(ii)] We add  a terminal cost in the value function. 
\end{enumerate}}

{The issues mentioned above are reflected in the corresponding Hamilton-Jacobi-Bellman (HJB) equation with gradient constraint (see \eqref{esd5}) of the value function $V$ in the following way: 
	\begin{enumerate}	
		\item[(i)] The solution of this HJB equation is a matrix function $u=(u_{\ell,\iota})_{(\ell,\iota)\in\mathbb{M}\times\mathbb{I}}$  where the row $
		\ell$ represents the matrix transition $Q_{\ell}$ with which the states $\iota$ (the columns) should be interacting with each other. These types of problems can be  found  in the literature only when the regime set $\mathbb{M}$ is a singleton set, i.e., optimal stochastic  control problems with Markov switching; see, i.e., \cite{FR2020, FY2018, JP2012}.
		\item[(ii)] The terminal cost is considered in the HJB equation as a boundary condition. The solution $u$ to the HJB equation is constructed as a limit of a sequence of functions $\{u^{\varepsilon,\delta}\}_{(\varepsilon,\delta)\in(0,1)^{2}}$ when $(\varepsilon,\delta)$ goes to $(0,0)$. The  entries of this sequence are classical solutions  to a   non-linear partial differential system (NPDS), which inherits  the same boundary condition (see \eqref{NPD.1}). So,  we must first guarantee the existence and uniqueness of $\{u^{\varepsilon,\delta}\}_{(\varepsilon,\delta)\in(0,1)^{2}}$, whose entries are in $C^{4}(\overline{\set})$, and then verify  for each $(\ell,\iota)\in\mathbb{M}\times\mathbb{I}$, that  $\{u^{\varepsilon,\delta}_{\ell,\iota}\}_{(\varepsilon,\delta)\in(0,1)^{2}}$  is bounded,   uniformly in $(\varepsilon,\delta)$, with respect to the norms $\|\cdot\|_{\hol^{0}(\overline{\set})}$, $\|\cdot\|_{\hol^{1}_{\loc}({\set})}$ and $\|\cdot\|_{\hol^{2}_{\loc}({\set})}$, in such a way that $u$ is well defined. Previous similar studies to ours; see  \cite{KM2019,KM2022} and  references therein; have shown that the sequences of functions related to their HJB equations are uniformly bounded with respect to the norms $\|\cdot\|_{\hol^{1}(\overline{\set})}$ and $\|\cdot\|_{\hol^{2}_{\loc}({\set})}$ due to their null boundary condition  {and  the non-negativeness of these sequences of functions. In fact, under the assumption that there exists a sub-solution to our HJB equation, the uniformly bounded property mentioned above is verified.} Existence and uniqueness of the solutions to the HJB equations with gradient constraint and a non-null boundary condition almost everywhere, have been studied by few authors; see, i.e.,  \cite{Hynd3}. 
\end{enumerate}}

%Since the solution $u$ to the HJB equation is taken as a limit of functions $u^{\varepsilon,\delta}$, when $(\varepsilon,\delta)$ goes to $(0,0)$, which are the unique classical solutions  to another HJB equations with the same boundary condition; see Equation \eqref{p13.0.1.0}, it is first checked  that $\{u^{\varepsilon}\}_{\varepsilon\in(0,1)}$  is uniformly bounded with respect $\varepsilon$ in the norms $\|\cdot\|_{\hol^{0}(\overline{\set})}$, $\|\cdot\|_{\hol^{1}_{\loc}(\set)}$ and $\|\cdot\|_{\sob^{2,\infty}_{\loc}({\set})}$. In contrast with similar problems that our, where the authors took the boundary condition identically equal to zero, they obtained their corresponding sequence of functions $\{u^{\varepsilon}\}_{\varepsilon\in(0,1)}$  uniformly bounded with respect $\varepsilon$ in the norms $\|\cdot\|_{\hol^{0,1}(\overline{\set})}$, $\|\cdot\|_{\hol^{1}_{\loc}(\set)}$ and $\|\cdot\|_{\sob^{2,\infty}_{\loc}({\set})}$; see, i.e., \cite{KM2019,KM2022} and  references therein.

\subsection{{Assumptions and main results}} 

In order to see that the value function $V_{\ell}(\cdot,\iota)$, defined in \eqref{vf1}, belongs to  $\hol^{0}(\overline{\set})\cap\sob^{2,\infty}_{\loc}(\set)$ for each $(\ell,\iota)\in\mathbb{M}\times\mathbb{I}$, let us first give  necessary conditions to guarantee the existence and uniqueness of the solution $u_{\ell}(\cdot,\iota)$ to \eqref{esd5} on the same space.

\subsubsection*{Assumptions}
{\it
	\begin{enumerate}[label=(H\arabic*),ref=H\arabic*]
		\item\label{h0}  The domain set $\set$ is an open and bounded set such that its boundary $\partial\set$ is of class $\hol^{4,\alpha}$, with $\alpha\in(0,1)$ fixed.
		\item  \label{h3} The switching costs sequence  $\{\vartheta_{\ell,\ell'}\}_{\ell,\kappa\in\mathbb{I}}$ is such that  $\vartheta_{\ell,\ell'}\geq 0$ and satisfies
		\begin{equation}\label{eq1}
			\vartheta_{\ell_{1},\ell_{3}}\leq\vartheta_{\ell_{1},\ell_{2}}+\vartheta_{\ell_{2},\ell_{3}},\ \text{for}\ \ell_{3}\neq\ell_{1}, \ell_{2}, 
		\end{equation}
		which means that it is cheaper to switch directly from regime $\ell_{1}$ to $\ell_{3}$ than  using the intermediate regime $\ell_{2}$.  Additionally, we assume that there is  no {\it loop of zero cost}, i.e.,
		\begin{equation}\label{l1}
			\begin{split}
				&\text{no family of regimes}\  \{\ell_{0},\ell_{1}, \dots,\ell_{n},\ell_{0}\}\\ 
				&\text{such that}\  \vartheta_{\ell_{0},\ell_{1}}=\vartheta_{\ell_{1},\ell_{2}}=\cdots=\vartheta_{\ell_{n},\ell_{0}}=0.
			\end{split}
		\end{equation}

	\end{enumerate}
	Let $\iota$ be in $\mathbb{I}$. Then:
	\begin{enumerate}[label=(H\arabic*),ref=H\arabic*]
		\setcounter{enumi}{2}	
		\item\label{a4}The {real valued functions} $f_{\iota}\eqdef f(\cdot,\iota)$, $h_{\iota}\eqdef h(\cdot,\iota)$ and $g_{\iota}\eqdef g(\cdot,\iota)$ belong to $\hol^{2,\alpha}(\overline{\set})$, are non-negative, and $\|f _{\iota}\|_{\hol^{2,\alpha}(\overline{\set})}$, $\|h_{\iota} \|_{\hol^{2,\alpha}(\overline{\set})}$, $\|g_{\iota} \|_{\hol^{2,\alpha}(\overline{\set})},$ are bounded by some finite positive constant $\Lambda$. 
		\item\label{h2} Let $\mathcal{S}(d)$ be the set of $d\times d$ symmetric matrices. The coefficients of the differential part of $ \dif_{\ell,\iota}$ (see \eqref{eq2}), $a_{\iota} \eqdef a(\cdot,\iota) :\overline{\set}\longrightarrow\mathcal{S}(d)$, $b_{\iota} \eqdef b(\cdot,\iota) =(b_{1}(\cdot,\iota) ,\dots,b_{d}(\cdot,\iota)) :\overline{\set}\longrightarrow\R^{d}$ and $c_{\iota}\eqdef c(\cdot,\iota) :\overline{\set}\longrightarrow\R$, are such that  $a_{\iota,i,j} ,b_{\iota,i} ,c_{\iota}\in\hol^{2,\alpha}(\overline{\set})$, $c_{\iota}>0$ on $\set$ and $\|a_{\iota,i,j} \|_{\hol^{2,\alpha}(\overline{\set})},\|b_{\iota,i} \|_{\hol^{2,\alpha}(\overline{\set})},\|c_{\iota} \|_{\hol^{2,\alpha}(\overline{\set})}$ are bounded by some finite positive constant $\Lambda$. We assume that there exist   a real number $\theta>0$  such that  
		\begin{equation}\label{H2}
			%{\Theta|\zeta|^{2}\geq}
			\langle a_{\iota}(x) \zeta,\zeta\rangle\geq \theta|\zeta|^{2},\ \text{for all}\  x\in\overline{\set},\ \zeta\in \R^{d}.
		\end{equation} 
	\end{enumerate}
}
{Taking into account} \eqref{l1} and  a heuristic derivation from dynamic programming principle, the HJB equation corresponding to the value function $V_{\ell,\iota}\eqdef V_{\ell}(\cdot,\iota)$ is given by
\begin{equation}\label{esd5}
	\begin{split}
		\max\bigg\{[c_{\iota}-\dif_{\ell,\iota}]u_{\ell,\iota}-  h_{\iota},
		|\deri^{1}u_{\ell,\iota}|- g_{\iota},u_{\ell,\iota}-\mathcal{M}_{\ell,\iota}u   \bigg\}&= 0,\ \text{on}\ \set,\\
		\text{s.t.}\ u_{\ell,\iota}&=f_{\iota},\ \text{in}\ \partial\set,
	\end{split}
\end{equation}
where for each $(\ell,\iota)\in\mathbb{M}\times\mathbb{I}$, $u_{\ell,\iota}= u_{\ell}(\cdot,\iota):\overline{\set}\longrightarrow\R $  and 
\begin{align}
	\mathcal{M}_{\ell,\iota}u(x)&\eqdef\min_{\ell'\in \mathbb{M}\setminus\{\ell\}}\{u_{\ell',\iota}(x)+\vartheta_{\ell,\ell'}\},\label{p6.0}\\
	\dif_{\ell,\iota} u_{\ell,\iota}(x)&\eqdef\tr[a_{\iota}(x) \deri^{2}u_{\ell,\iota}(x)]-\langle b_{\iota}(x) ,\deri^{1}u_{\ell,\iota}(x) \rangle\notag \\
	&\quad+\sum_{\kappa\in\mathbb{I}\setminus\{\iota\}}q_{\ell}(\iota,\kappa)[u_{\ell,\kappa}(x)-u_{\ell,\iota}(x)],\label{eq2}
\end{align}
with $a_{\iota}=(a_{\iota,i,j})_{d\times d}$ is such that $a_{\iota,i,j}\eqdef\frac{1}{2}[\sigma_{\iota}\sigma^{\trans}_{\iota}]_{i,j}$. Here $|\cdot|$, $\langle\cdot,\cdot\rangle$ and $\tr[\,\cdot\,]$  {represent} the Euclidean norm, the inner product, and  the  {matrix trace}, respectively. The operator $\deri^{k}u_{\ell,\iota}(x)$, with $k\geq1$ an integer number, represents the $k$-th differential operator of $u_{\ell,\iota}(x)$ with respect to $x$. 

Under assumptions \eqref{h0}--\eqref{h2},  we have the following proposition.
\begin{prop}\label{M1}
	The HJB equation \eqref{esd5} has a unique  non-negative strong solution  {(in the almost everywhere sense)} $u=(u_{\ell,\iota})_{\mathbb{M}\times\mathbb{I}}$ where  $ u_{\ell,\iota} \in\hol^{0}(\overline{\set})\cap\sob^{2,\infty}_{\loc}(\set)$ for each $(\ell,\iota)\in\mathbb{M}\times\mathbb{I}$.
\end{prop}

 {In such a way that we can verify   the agreement of the value function $V$ and the solution $u$ to  \eqref{esd5} in $\overline{\set}$, we need to assume the following statements.}
{\it\begin{enumerate}[label=(H\arabic*),ref=H\arabic*]
		\setcounter{enumi}{4}
		\item\label{h5} The domain set $\set$ is an open, convex and bounded set such that its boundary $\partial\set$ is of class $\hol^{4,\alpha}$, with $\alpha\in(0,1)$ fixed.
		\item\label{h6}  {There exists a matrix function $\underline{u}=(\underline{u}_{\ell,\iota})_{(\ell,\iota)\in\mathbb{M}\times\mathbb{I}}$ such that $\underline{u}_{\ell,\iota}\in\hol^{1}(\overline{\set})\cap\hol^{2}(\set)$ and satisfies
			\begin{equation*}%\label{esd5.1}
				\begin{split}
					\max\bigg\{[c_{\iota}-\dif_{\ell,\iota}]\underline{u}_{\ell,\iota}-  h_{\iota},
					|\deri^{1}\underline{u}_{\ell,\iota}|- g_{\iota},\underline{u}_{\ell,\iota}-\mathcal{M}_{\ell,\iota}\underline{u}   \bigg\}&\leq -\bar{r},\ \text{on}\ \set,\\
					\text{s.t.}\ \underline{u}_{\ell,\iota}&=f_{\iota},\ \text{in}\ \partial\set,
				\end{split}
			\end{equation*}
			for some $\bar{r}>0$.}
\end{enumerate} }
Under assumptions \eqref{h3}--\eqref{h6},  the main {goal}  obtained in this document is as follows.
\begin{teor}\label{verf2}
	Let $V$ be the value function given by \eqref{vf1}. Then  {$V_{\ell_{0}}(\cdot,\iota_{0})\in\hol^{1}(\overline{\set})\cap\sob^{2,\infty}_{\loc}(\set)$} and $V_{\ell_{0}}( x_{0},\iota_{0})=u_{\ell_{0}}( x_{0},\iota_{0})$ for $( x_{0},\iota_{0})\in \overline{\set}\times\mathbb{I}$ and $ \ell_{0}\in\mathbb{M}$. 
\end{teor}

{In order to verify the results above, first we need to guarantee the existence and uniqueness of the classical solution  $u^{\varepsilon,\delta}=(u^{\varepsilon,\delta}_{\ell,\iota})_{(\ell,\iota)\in\mathbb{M}\times\mathbb{I}}$ to the following NPDS
	\begin{equation}\label{NPD.1}
		\begin{split}
			[c_{\iota}-\dif_{\ell,\iota}]u_{\ell,\iota}^{\varepsilon,\delta}+\psi_{\varepsilon}(
			|\deri^{1}u_{\ell,\iota}^{\varepsilon,\delta}|^2- g^{2}_{\iota})
			+\sum_{\ell'\in \mathbb{M}\backslash \{\ell\}}\psi_{\delta}(u_{\ell,\iota}^{\varepsilon,\delta}-u_{\ell',\iota}^{\varepsilon,\delta}-\vartheta_{\ell,\ell'})&=   h_{\iota} \ \text{on}\ \set,\\
			\text{s.t.}\ u_{\ell,\iota}^{\varepsilon,\delta}&=f_{\iota},\ \text{in}\ \partial{\set},
		\end{split}
	\end{equation}
	where $\psi_\varepsilon$   is defined by $\psi_\varepsilon(t)=\varphi(t/\varepsilon)$ with $\varepsilon\in (0,1)$, and $\varphi:\R\to \R$ is in $C^\infty(\R)$ is such that 
	\begin{equation}
		\begin{split}\label{p12.1}
			\varphi(t)=0,\quad t\leq 0,\quad 
			\varphi(t)>0,\quad t>0,\\
			\varphi(t)=t-1,\quad t\geq 2, \quad 
			\varphi^{\prime}(t)\geq 0,\quad \varphi^{\prime\prime}(t)\geq 0.
		\end{split}
	\end{equation}
	Then, as an intermediate step,  it will be proven that  $u^{\varepsilon}$ defined as limit of $u^{\varepsilon,\delta}$, when $\delta\downarrow0$, is the unique strong solution to the following HJB equation
	\begin{equation}\label{pc1}
		\begin{split}
			\max\big\{[c_{\iota}-\dif_{\ell,\iota}]u_{\ell,\iota}^{\varepsilon}
			+\psi_{\varepsilon}(|\deri^{1}u^{\varepsilon}_{\ell,\iota}|^{2}- g^{2}_{\iota})
			-  h_{\iota},
			u_{\ell,\iota}^{\varepsilon}-\mathcal{M}_{\ell,\iota}u^{\varepsilon}\big\}&= 0,\ \text{on}\ \set,\\
			\text{s.t.}\ u_{\ell,\iota}^{\varepsilon}&=f_{\iota},\ \text{in}\ \partial{\set}.
		\end{split}
	\end{equation} 
	The reason for doing that is because,  {under assumptions \eqref{h0}--\eqref{h2} and \eqref{h6}}, $u^{\varepsilon}$ coincides with the value function $V^{\varepsilon}$, which will be defined later on (see \eqref{vw1}), of an $\varepsilon$-penalized
	absolutely continuous/switching ($\varepsilon$-PACS) control problem; see Section \ref{Pp1}. Although the solution $u$ to the HJB equation \eqref{esd5} can be constructed directly as a limit of $u^{\varepsilon,\delta}$, when $(\varepsilon,\delta)$ goes to $(0,0)$, we required first to analyse the properties of the optimal stochastic control associated with the $\varepsilon$-PACS control problem mentioned above, in such a way that we can corroborate the equivalence between $u$ and $V$ in $\overline{\set}$.}

{We would like to mention that the NPDS \eqref{NPD.1}, named in the  PDE theory as a {\it non-linear elliptic cooperative system},  is a problem of interest itself because  we can find literature related to this  problem only when the regime set $\mathbb{M}$ is a singleton set;  see, i.e.,  \cite{BS2004,NSS2020,S1995}.}

Under assumptions \eqref{h0}, \eqref{a4} and \eqref{h2},   the following result is obtained.
\begin{prop}\label{princ1.0}
	Let $\varepsilon,\delta\in(0,1)$ be fixed. There exists a unique non-negative solution $u^{\varepsilon,\delta}=(u^{\varepsilon,\delta}_{\ell,\iota})_{(\ell,\iota)\in\mathbb{M}\times\mathbb{I}}$ to the NPDS  \eqref{NPD.1} where  $u^{\varepsilon,\delta}_{\ell,\iota}\in\hol^{4,\alpha}(\overline{\set})$ for each $(\ell,\iota)\in\mathbb{M}\times\mathbb{I}$.
\end{prop}

Under assumptions \eqref{h0}--\eqref{h2},   the following result is obtained.
\begin{prop}\label{princ1.1}
	For each $\varepsilon\in(0,1)$ fixed, there exists a unique non-negative strong solution $u^{\varepsilon}=(u^{\varepsilon}_{\ell,\iota})_{\mathbb{M}\times\mathbb{I}}$ to the HJB equation  \eqref{pc1} where  $u^{\varepsilon}_{\ell,\iota}\in\hol^{0}(\overline{\set})\cap\sob^{2,\infty}_{\loc}(\set)$ for each $(\ell,\iota)\in\mathbb{M}\times\mathbb{I}$.
\end{prop} 

{The rest of this document is organized as follows: in Section \ref{NPDSs}, using \eqref{h0}, \eqref{a4} and \eqref{h2}, and by a fixed point argument, the existence and uniqueness of the solution $u^{\varepsilon,\delta}$ to the NPDS \eqref{NPD.1}, with $(\varepsilon,\delta)\in(0,1)^{2}$ fixed, is proven.  Then, in Section   \ref{prop1},   some estimations for $u^{\varepsilon,\delta}$ are given. For that aim, we first study the classical solution to a linear elliptic cooperative system; see Lemma \ref{dm1}. Afterwards,  using Proposition \ref{princ1.0} and Lemmas \ref{R1} and \ref{Lb1}, Arzel\`a-Ascoli  compactness criterion and the reflexivity of $\Lp^{p}_{\loc}(\set)$;  see \cite[Section C.8, p. 718]{evans2} and \cite[Thm. 2.46, p. 49]{adams} respectively, we discuss the existence, regularity and uniqueness of the solutions  $u$ and $u^{\varepsilon}$  to \eqref{esd5} and \eqref{pc1}, respectively. Later, in Section \ref{Pp1},  we introduce the   $\varepsilon$-PACS  control problem and its  verification lemma  is presented.  Afterwards, we give the proof of Theorem \ref{M1}. To finalize this section, let us say that the notations and definitions of the function spaces that are used in this paper are standard and the reader can find them in \cite{adams,Cs1,evans2, garroni,gilb}.}

\section{Existence and uniqueness of the solution to the NPDS \eqref{NPD.1}}\label{NPDSs}

Let  $\mathcal{C}^{k}_{m, n}$, $\mathcal{C}^{k,\alpha}_{m, n}$ be the sets of {$(m\times n)$-matrix functions} given by $(\hol^{k}(\overline{\set}))^{m\times n}$, $(\hol^{k,\alpha}(\overline{\set}))^{m\times n}$, respectively, with $k\in\N$ and $\alpha\in(0,1)$. Defining $\|\w\|_{\mathcal{C}^{k}_{m, n}}=\max_{(\ell,\iota)\in\mathbb{M}\times\mathbb{I}}\{\|\w_{\ell,\iota}\|_{\hol^{k}(\overline{\set})}\}$ for each $\w=(\w_{\ell,\iota})_{(\ell,\iota)\in\mathbb{M}\times\mathbb{I}}\in\mathcal{C}^{k}_{m,n}$, it can be verified that $\|\cdot\|_{\mathcal{C}^{k}_{m, n}}$, $\|\cdot\|_{\mathcal{C}^{k,\alpha}_{m, n}}$ are norms on $\mathcal{C}^{k}_{m,n}$, $\mathcal{C}^{k,\alpha}_{m,n}$, respectively, and  $(\mathcal{C}^{k}_{m,n},\|\cdot\|_{\mathcal{C}^{k}_{m, n}})$ $(\mathcal{C}^{k,\alpha}_{m,n},\|\cdot\|_{\mathcal{C}^{k,\alpha}_{m, n}})$ are  Banach spaces.

{Since the arguments to  guarantee the existence of the solution to the NPDS \eqref{NPD.1} are based  on Schaefer's fixed point theorem, we will provide the necessary results to obtain the conditions of this  theorem (see, i.e., \cite[Thm. 4 p. 539]{evans2}).}

Let us define  the operators $\widetilde{\dif}_{\iota}$ and $\Xi_{\ell,\iota}$ as follows
\begin{align*}
	\widetilde{\dif}_{\iota}\w_{\ell,\iota}&=\tr[a_{\iota} \deri^{2}\w_{\ell,\iota} ]-\langle b_{\iota} ,\deri^{1}\w_{\ell,\iota} \rangle,\\%\label{npd2}\\
	\Xi_{\ell,\iota}\w&=\sum_{\ell'\in\mathbb{M}\setminus\{\ell\}}\psi_{\delta}(\w_{\ell,\iota}-\w_{\ell',\iota} -\vartheta_{\ell,\ell'})+\sum_{\kappa\in\mathbb{I}\setminus\{\iota\}}q_{\ell}(\iota,\kappa)[\w_{\ell,\iota}-\w_{\ell,\kappa}]+\psi_{\varepsilon}(|\deri^{1} \w_{\ell,\iota}|^{2}- g_{\iota}^{2}).\notag%\label{npd3}
\end{align*}
We observe then for each $\w\in\mathcal{C}^{1,\alpha}_{m,n}$ fixed, there exists a unique solution $\sol=(\sol_{\ell,\iota})_{(\ell,\iota)\in\mathbb{M}\times\mathbb{I}}\in\mathcal{C}^{2,\alpha}_{m, n}$ to the following linear partial differential system (LPDS) 
\begin{equation}\label{LPD.1u} 
	\begin{split}
		[c_{\iota}-\widetilde{\dif}_{\iota}]  \sol_{\ell,\iota}=h_{\iota}- \Xi_{\ell,\iota}\w ,\ \text{in}\ \set,&\\
		\text{s.t.}\  
		\sol_{\ell,\iota}=f_{\iota},\ \text{on}\ \partial\set , &
	\end{split}
	\quad\text{for}\ (\ell,\iota)\in\mathbb{M}\times\mathbb{I},
\end{equation}
since \eqref{h0}, \eqref{a4} and \eqref{h2} hold and $\Xi_{\ell,\iota}\w\in\hol^{0,\alpha}(\overline{\set})$ (see Theorem 6.14 of \cite{gilb}). Additionally, due to Theorem 1.2.10 of \cite{garroni}, the following inequality can be checked 
\begin{equation}\label{in3}
	\|\sol_{\ell,\iota}\|_{\hol^{2,\alpha}(\overline{\set})}\leq C_{2}\bigg[1 +\dfrac{1}{\delta}+\dfrac{1}{\varepsilon}+\bigg[1+\dfrac{1}{\delta}\bigg]\|\w\|_{\mathcal{C}^{0,\alpha}_{m,n}}+\dfrac{1}{\varepsilon}\|\w_{\ell,\iota}\|^{2}_{\hol^{1,\alpha}(\overline{\set})}\bigg]\quad\text{for}\ (\ell,\iota)\in\mathbb{M}\times\mathbb{I},
\end{equation}
for some  $C_{2}=C_{2}(d,\Lambda,\alpha,\theta)$. Defining the mapping
$$\overline{T}:(\mathcal{C}^{1,\alpha}_{m,n},\|\cdot\|_{\mathcal{C}^{1,\alpha}_{m,n}})\longrightarrow(\mathcal{C}^{1,\alpha}_{m,n},\|\cdot\|_{\mathcal{C}^{1,\alpha}_{m,n}})$$ 
as $\overline{T}[\w]=\sol$ for each $\w\in\mathcal{C}^{1,\alpha}_{m,n}$, where  $\sol\in\mathcal{C}^{2,\alpha }_{m,n}\subset\mathcal{C}^{1,\alpha}_{m,n}$ is the unique solution to the LDPS \eqref{LPD.1u}, we get  that, by \eqref{in3} and by Arzel\`a-Ascoli's  compactness criterion; see \cite[Section C.8, p. 718]{evans2}, $\overline{T}$  maps bounded sets in $\mathcal{C}^{1,\alpha}_{m,n}$ into bounded sets in $\mathcal{C}^{2,\alpha}_{m,n}$ which are precompact in $\mathcal{C}^{1,\alpha}_{m,n}$. From here and by the uniqueness of the solution to the LPDS  \eqref{LPD.1u}, it can be verified  that $\overline{T}$ is a continuous and compact mapping from $\mathcal{C}^{1,\alpha}_{m,n}$ into itself. 

Now, we only need to verify that the set
$$\mathcal{A}_{1}\eqdef\{\w\in\mathcal{C}^{1,\alpha}_{m,n}:\w=\varrho\overline{ T}[\w],\ \text{for some }\ \varrho\in[0,1]\}$$ 
is bounded uniformly  on the norm $\|\cdot\|_{\mathcal{C}^{1,\alpha}_{m,n}}$. {Notice that the LPDS associated with $\varrho=0$ is
	\begin{equation}\label{eq_1}
		\begin{split}
			[{c}_{\iota}-\widetilde{\dif}_{\iota}]  \w_{\ell,\iota}&= 0,\ \text{in}\ \set,\\
			\text{s.t.}\  
			\w_{\ell,\iota}&=0,\  \text{on}\ \partial\set ,
		\end{split} \quad \text{for}\ (\ell,\iota)\in\mathbb{M}\times\mathbb{I}.
	\end{equation}  
	which solution is immediately $\w\equiv\overline{0}\in\mathcal{C}^{1,\alpha}_{m, n}$,  with $\overline{0}$ the null matrix function}. 

\begin{lema}\label{B1}
	If $\w\in\mathcal{C}^{1,\alpha}_{m, n}$ is such that $\overline{T}[\w]=\frac{1}{\varrho}\w=(\frac{1}{\varrho}\w_{\ell,\iota})_{(\ell,\iota)\in\mathbb{M}\times\mathbb{I}}$ for some $\varrho\in(0,1]$,  then there exists a constant $C_{1}>0$ independent of $\varrho$ and $\w$ such that
	\begin{equation}\label{e0}
		\|{\w}_{\ell,\iota}\|_{\hol^{1,\alpha}(\overline{\set})}\leq C_{1}\bigg[1+\dfrac{1}{\varepsilon}+\dfrac{1}{\delta}[1+\|{\w}\|_{\mathcal{C}^{0}_{m,n}}]\bigg]\quad\text{for}\ (\ell,\iota)\in\mathbb{M}\times\mathbb{I}.
	\end{equation}
\end{lema}

\begin{proof}
	Observe that $\w\in\mathcal{C}^{2,\alpha}_{m, n}$  and
	\begin{equation}\label{LPDw.1}
		\begin{split}
			[{c}_{\iota}-\widetilde{\dif}_{\iota}]  \w_{\ell,\iota}= \varrho[h_{\iota}-\Xi_{\ell,\iota} \w],\ \text{in}\ \set,&\\
			\text{s.t.}\  
			\w_{\ell,\iota}=\varrho f_{\iota},\  \text{on}\ \partial\set ,&
		\end{split} \quad \text{for}\ (\ell,\iota)\in\mathbb{M}\times\mathbb{I}.
	\end{equation}  
	Defining $\bar{\w}=(\bar{\w}_{\ell,\iota})_{(\ell,\iota)\in\mathbb{M}\times\mathbb{I}}$ as $\bar{\w}_{\ell,\iota}=\w_{\ell,\iota}-\varrho f_{\iota}$, \eqref{LPDw.1} can be rewritten in the following way
	\begin{equation*}%\label{LPDbw.1}
		\begin{split}
			[\Gamma_{\iota}+1]  \bar{\w}_{\ell,\iota}=\mu[1+|\deri^{1}\bar{\w}_{\ell,\iota}|^{2}],\ \text{in}\ \set,&\\
			\text{s.t.}\  
			\bar{\w}_{\ell,\iota}=0,\  \text{on}\ \partial\set ,&
		\end{split} \quad \text{for}\ (\ell,\iota)\in\mathbb{M}\times\mathbb{I},
	\end{equation*}  
	where $\Gamma_{\iota}\w_{\ell,\iota}\eqdef-\tr[a_{\iota}\deri^{2}\w_{\ell,\iota}]$ and
	\begin{align*}
		\mu&\eqdef\dfrac{1}{1+|\deri^{1}\bar{\w}_{\ell,\iota}|^{2}}\bigg[[1-c_{\iota}]\bar{\w}_{\ell,\iota}-\langle b_{\iota},\deri^{1}\bar{\w}_{\ell,\iota}\rangle+\varrho h_{\iota}\notag\\
		&\quad-\varrho[c_{\iota}-\widetilde{\mathcal{L}}_{\iota}]f_{\iota}-\varrho\sum_{\ell'\in\mathbb{M}\setminus\{\ell\}}\psi_{\delta}(\bar{\w}_{\ell,\iota}-\bar{\w}_{\ell',\iota} -\vartheta_{\ell,\ell'})\notag\\
		&\quad-\sum_{\kappa\in\mathbb{I}\setminus\{\iota\}}q_{\ell}(\iota,\kappa)[\bar{\w}_{\ell,\iota}+\varrho f_{\iota}-[\bar{\w}_{\ell,\kappa}+\varrho f_{\kappa}]]-\psi_{\varepsilon}(|\deri^{1} [\bar{\w}_{\ell,\iota}+\varrho f_{\iota}]|^{2}- g_{\iota}^{2})\bigg].
	\end{align*}
	Applying \cite[Thm. 4.12, p.85]{adams} and \cite[Lemma 4]{AC1978} (see also \cite[Lemma 2.4]{taira1997}), we get that 
	\begin{equation}\label{e1}
		\|\bar{\w}_{\ell,\iota}\|_{\hol^{1,\alpha}(\overline{\set})}\leq K_{1,1}\|\mu\|_{\Lp^{\infty}(\set)},
	\end{equation}
	for some constant $K_{1,1}>0$ independent of $\varrho$ and $\bar{\w}$. Meanwhile
	\begin{align}\label{e2}
		|\mu|&\leq [1+c_{\iota}]|\bar{\w}_{\ell,\iota}|+|[c_{\iota}-\widetilde{\mathcal{L}}_{\iota}]f_{\iota}|+ h_{\iota}\notag\\
		&\quad+\sum_{\kappa\in\mathbb{I}\setminus\{\iota\}}q_{\ell}(\iota,\kappa)[|\bar{\w}_{\ell,\iota}|+f_{\iota}+|\bar{\w}_{\ell,\kappa}|+ f_{\kappa}]+\dfrac{1}{\delta}\sum_{\ell'\in\mathbb{M}\setminus\{\ell\}}[|\bar{\w}_{\ell,\iota}|+|\bar{\w}_{\ell',\iota}|+\vartheta_{\ell,\ell'}]\notag\\
		&\quad+|b_{\iota}|\dfrac{ |\deri^{1}\bar{\w}_{\ell,\iota}|}{1+|\deri^{1}\bar{\w}_{\ell,\iota}|^{2}}+\dfrac{1}{\varepsilon}\bigg[ \dfrac{2|\deri^{1}\bar{\w}_{\ell,\iota}|^2}{1+|\deri^{1}\bar{\w}_{\ell,\iota}|^{2}}+ 2|\deri^{1}f_{\iota}|^{2}+ g_{\iota}^{2}\bigg]\notag\\
		&\leq K_{1,2}\bigg[1+\dfrac{1}{\varepsilon}+\dfrac{1}{\delta}[1+\|{\w}\|_{\mathcal{C}^{0}_{m,n}}]\bigg]\quad \text{on}\ \overline{\set},
	\end{align}
	for some constant $K_{1,2}>0$ independent of $\varrho$ and $\bar{\w}$. By \eqref{e1}, \eqref{e2} and taking into account that $\|\w_{\ell,\iota}\|_{\hol^{1,\alpha}(\overline{\set})}\leq\|\bar{\w}_{\ell,\iota}\|_{\hol^{1,\alpha}(\overline{\set})}+\|f_{\iota}\|_{\hol^{1,\alpha}(\overline{\set})}$ we see that \eqref{e0} is true.
\end{proof}

In view of \eqref{e0}, to see $\mathcal{A}_{1}$ is bounded uniformly  with respect to the norm $\|\cdot\|_{\mathcal{C}^{1,\alpha}_{m,n}}$,  it is sufficient to check that $w$ is uniformly bounded with respect to the norm  	$\|\cdot\|_{\mathcal{C}^{0}_{m,n}}$.
\begin{lema}\label{B12}
	If $\w\in\mathcal{C}^{1,\alpha}_{m, n}$ is such that $\overline{T}[\w]=\frac{1}{\varrho}\w=(\frac{1}{\varrho}\w_{\ell,\iota})_{(\ell,\iota)\in\mathbb{M}\times\mathbb{I}}$ for some $\varrho\in(0,1]$,  then 
	\begin{equation}\label{w1}
		0\leq\w_{\ell,\iota}(x)\leq{ \Lambda\max_{(x',\kappa)\in\overline{\set}\times\mathbb{I}}\bigg\{1,\dfrac{1}{c_{\kappa}(x')}\bigg\}}\ \text{for}\ x\in\overline{\set}\ \text{and}\ (\ell,\iota)\in\mathbb{I}\times\mathbb{I}.  
	\end{equation}
\end{lema}
\begin{proof}
	Let $(x_{\circ},\ell_{\circ},\iota_{\circ})\in\overline{\set}\times\mathbb{M}\times\mathbb{I}$ be such that 
	$$\w_{\ell_{\circ},\iota_{\circ}}(x_{\circ})=\min_{(x,\ell,\iota)\in\overline{\set}\times\mathbb{M}\times\mathbb{I}}\w_{\ell,\iota}(x).$$
	If $x_{\circ}\in\partial\set$, it follows easily that $\w_{\ell,\iota}(x)\geq \w_{\ell_{\circ},\iota_{\circ}}(x_{\circ})=\varrho f_{\iota_{\circ}}(x_{\circ})\geq0$ for  $(x,\ell,\iota)\in\overline{\set}\times\mathbb{M}\times\mathbb{I}$. Suppose that $x_{\circ}\in\set$. Then,
	\begin{equation*}
		\begin{split}
			&\deri^{1}\w_{\ell_{\circ},\iota_{\circ}}(x_{\circ})=0,\  0\leq\tr[a_{\iota_{\circ}}(x_{\circ})\deri^{2}\w_{\ell_{\circ},\iota_{\circ}}(x_{\circ})],\\
			&\w_{\ell_{\circ},\iota_{\circ}}(x_{\circ})-\w_{\ell_{\circ},\kappa}(x_{\circ} )\leq 0\ \text{for}\ \kappa\in\mathbb{I}\setminus\{\iota_{\circ}\},\\ &\w_{\ell_{\circ},\iota_{\circ}}(x_{\circ})-\w_{\ell',\iota_{\circ}}(x_{\circ} )\leq 0\ \text{for}\ \ell'\in\mathbb{M}\setminus\{\ell_{\circ}\}.
		\end{split} 
	\end{equation*}
	From here and using \eqref{a4}  and \eqref{LPDw.1}, it gives 
	\begin{align*}
		0&\leq\tr[a_{\iota_{\circ}}(x_{\circ})\deri^{2}\w_{\ell_{\circ},\iota_{\circ}}(x_{\circ})]\notag\\
		&= c_{\iota_{\circ}}(x_{\circ})\w_{\ell_{\circ},\iota_{\circ}}(x_{\circ})-\varrho h_{\iota_{\circ}}(x_{\circ})+\varrho\Xi_{\ell_{\circ},\iota_{\circ}}\w(x_{\circ})\leq c_{\iota_{\circ}}(x_{\circ})\w_{\ell_{\circ},\iota_{\circ}}(x_{\circ}).
	\end{align*}
	Since $c_{\iota_{\circ}}>0$ on $\overline{\set}$, it follows that $\w_{\ell_{\circ},\iota_{\circ}}(x_{\circ})\geq0$. Therefore $\w_{\ell,\iota}(x)\geq \w_{\ell_{\circ},\iota_{\circ}}(x_{\circ})\geq0$, for all $x\in\overline{\set}$ and $(\ell,\iota)\in\mathbb{M}\times\mathbb{I}$. Supposing now  $(x^{\circ},\ell^{\circ},\iota^{\circ})\in\overline{\set}\times\mathbb{M}\times\mathbb{I}$   such that 
	$$\w_{\ell^{\circ},\iota^{\circ}}(x^{\circ})=\max_{(x,\ell,\iota)\in\overline{\set}\times\mathbb{M}\times\mathbb{I}}\w_{\ell,\iota}(x),$$
	taking into account that 
	\begin{equation*}
		\begin{split}
			&\deri^{1}\w_{\ell^{\circ},\iota^{\circ}}(x^{\circ})=0,\ \tr[a_{\iota^{\circ}}(x^{\circ})\deri^{2}\w_{\ell^{\circ},\iota^{\circ}}(x^{\circ})]\leq 0,\\
			&\w_{\ell^{\circ},\iota^{\circ}}(x^{\circ})-\w_{\ell^{\circ},\kappa}(x^{\circ} )\geq 0\ \text{for}\ \kappa\in\mathbb{I}\setminus\{\iota^{\circ}\},\\
			&\w_{\ell^{\circ},\iota^{\circ}}(x^{\circ})-\w_{\ell',\iota^{\circ}}(x^{\circ} )\geq 0\ \text{for}\ \ell'\in\mathbb{M}\setminus\{\ell^{\circ}\}.
		\end{split} 
	\end{equation*}
	and arguing in a similar way as before, the reader can easily see  that \eqref{w1} is true. With this remark, we conclude the proof.
\end{proof}
From now on, for simplicity of notation, we shall replace $u^{\varepsilon,\delta}$ by $u$ in the proofs of the results that we will share below.
\begin{proof}[Proof of Proposition \ref{princ1.0}. Existence.] By \eqref{h0}, \eqref{a4}, \eqref{h2}, \eqref{e0}  and \eqref{w1}, it follows that $\mathcal{A}_{1}$, is bounded uniformly {with respect} to the norm $\|\cdot\|_{\mathcal{C}^{1,\alpha}_{m, n}}$. From here and since the mapping $\overline{T}$ is continuous and compact, by Schaefer's fixed point theorem, it yields that there exists a fixed point $u=(u_{\ell,\iota})_{(\ell,\iota)\in\mathbb{M}\times\mathbb{I}}\in\mathcal{C}^{1,\alpha}_{m, n}$ to the problem $T[u]=u$.   {In addition}, we have $u=T[u]\in\mathcal{C}^{2,\alpha}_{m, n}$. By Theorem 9.19 of \cite{gilb}, we conclude  that $u\in\mathcal{C}^{3,\alpha}_{m,n}$, since \eqref{h0}, \eqref{a4} and \eqref{h2} hold and ${\Xi}_{\ell,\iota} u\in\hol^{1,\alpha}(\overline{\set})$. Again, repeating the same argument as before, we obtain  that $u\in\mathcal{C}^{4,\alpha}_{m,n}$. Finally, The non-negativeness of $u_{\ell,\iota}$ can be verified using similar arguments seen in the proof of Lemma \ref{B12}.
\end{proof}

 {Let us first  introduce  the concepts of sub-solution and super-solution for the  NPDS \eqref{NPD.1}. Then, we shall give a principle of comparison; see Lemma \eqref{supsuper}. Finally, from  here, we shall see that the  uniqueness of the solution to the NPDS \eqref{NPD.1} is an immediately consequence of this lemma.
	\begin{defin}\label{princmax1}
		\begin{enumerate}[label=(\arabic*),ref=H\arabic*]
			\item A matrix function $\varphi=(\varphi_{\ell,\iota})$ in $\mathcal{C}^{2}_{m,n}({\set})\cap\mathcal{C}^{0}_{m,n}(\overline{\set})$ is a sub-solution of \eqref{NPD.1} if
			\begin{equation*}%\label{NPD.1sub}
				\begin{split}
					[c_{\iota}-\dif_{\ell,\iota}]\varphi_{\ell,\iota}+\psi_{\varepsilon}(
					|\deri^{1}\varphi_{\ell,\iota}|^2- g^{2}_{\iota})
					+\sum_{\ell'\in \mathbb{M}\backslash \{\ell\}}\psi_{\delta}(\varphi_{\ell,\iota}-\varphi_{\ell',\iota}-\vartheta_{\ell,\ell'})&\leq   h_{\iota} \ \text{on}\ \set,\\
					\text{s.t.}\ \varphi_{\ell,\iota}&=f_{\iota},\ \text{in}\ \partial{\set},
				\end{split}
			\end{equation*}
			\item  A matrix function $\eta=(\eta_{\ell,\iota})$ in $\mathcal{C}^{2}_{m,n}({\set})\cap\mathcal{C}^{0}_{m,n}(\overline{\set})$ is is a super-solution of \eqref{NPD.1} if
			\begin{equation*}%\label{NPD.1sub}
				\begin{split}
					[c_{\iota}-\dif_{\ell,\iota}]\eta_{\ell,\iota}+\psi_{\varepsilon}(
					|\deri^{1}\eta_{\ell,\iota}|^2- g^{2}_{\iota})
					+\sum_{\ell'\in \mathbb{M}\backslash \{\ell\}}\psi_{\delta}(\eta_{\ell,\iota}-\eta_{\ell',\iota}-\vartheta_{\ell,\ell'})&\geq   h_{\iota} \ \text{on}\ \set,\\
					\text{s.t.}\ \eta_{\ell,\iota}&=f_{\iota},\ \text{in}\ \partial{\set},
				\end{split}
			\end{equation*}
		\end{enumerate}
	\end{defin}
	%An immediate consequence of  {this}  definition is the following result, which is used to prove Lemma \ref{Lb1}. We denote $\varphi\leq\eta$ if $\varphi_{\ell,\iota}\leq\eta_{\ell,\iota}$ for all $(\ell,\iota)\in\mathbb{M}\times\mathbb{I}$.
	\begin{lema}\label{supsuper}
		Let $\varphi,\eta\in\mathcal{C}^{2}_{m,n}({\set})\cap\mathcal{C}^{0}_{m,n}(\overline{\set})$. If  $\varphi$ and $\eta$ are a sub-solution and a super-solution of \eqref{NPD.1}, respectively, then $\varphi\leq\eta$  on $\overline{\set}$.
\end{lema}}

 {\begin{proof}
		From Definition \ref{princmax1}, we get
		\begin{align}\label{sup1}
			[c_{\iota}-\dif_{\ell,\iota}][\varphi-\eta]_{\ell,\iota}+\psi_{\varepsilon}(
			|\deri^{1}\varphi_{\ell,\iota}|^2- g^{2}_{\iota})-\psi_{\varepsilon}(
			|\deri^{1}\eta_{\ell,\iota}|^2- g^{2}_{\iota})
			&\notag\\
			+\sum_{\ell'\in \mathbb{M}\backslash \{\ell\}}[\psi_{\delta}(\varphi_{\ell,\iota}-\varphi_{\ell',\iota}-\vartheta_{\ell,\ell'})-\psi_{\delta}(\eta_{\ell,\iota}-\eta_{\ell',\iota}-\vartheta_{\ell,\ell'})]&\leq 0\ \text{on}\ \set,
		\end{align}
		with $[\varphi-\eta]_{\ell,\iota}=0,\ \text{in}\ \partial{\set}$. Let  $(x_{\circ},\ell_{\circ},\iota_{\circ})\in\overline{\set}\times\mathbb{M}\times\mathbb{I}$ satisfying
		\begin{align*}
			&[\varphi-\eta]_{\ell_{\circ},\iota_{\circ}}(x_{\circ})=\max_{(x,\ell,\iota)\in\overline{\set}\times\mathbb{M}\times\mathbb{I}}\{ [\varphi-\eta]_{\ell,\iota}(x)\}.
		\end{align*} 
		If $x_{\circ}\in\partial{\set}$, trivially, we have $\varphi\leq\eta$ on $\overline{\set}$. Suppose that $x_{\circ}\in \set$. This means that at $x_{\circ}$
		\begin{equation}\label{sup1.1}
			\begin{split}
				&\deri^{1}[\varphi-\eta]_{\ell_{\circ},\iota_{\circ}}=0,\ \tr[a_{\iota_{\circ}}\deri^{2}[\varphi-\eta]_{\ell_{\circ},\iota_{\circ}}]\leq 0,\\
				&[\varphi-\eta]_{\ell_{\circ},\kappa}-[\varphi-\eta]_{{\ell}_{\circ},\iota_{\circ}}\leq 0 ,\ \text{for}\ \kappa\in\mathbb{I}\setminus\{\iota_{\circ}\},\\
				&\eta_{\ell_{\circ},\iota_{\circ}}-\eta_{\ell',\iota_{\circ}}\leq \varphi_{{\ell}_{\circ},\iota_{\circ}}-\varphi_{{\ell}',\iota_{\circ}} ,\ \text{for}\ \ell'\in\mathbb{M}\setminus\{\ell_{\circ}\}. 
			\end{split}
		\end{equation}
		Since $\psi_{\delta}$ is increasing, we get that $\psi_{\delta}(\eta_{\ell_{\circ},\iota_{\circ}}-\eta_{\ell',\iota_{\circ}}-\vartheta_{\ell_{\circ},\ell'})\leq \psi_{\delta}(\varphi_{{\ell}_{\circ},\iota_{\circ}}-\varphi_{{\ell}',\iota_{\circ}}-\vartheta_{\ell_{\circ},\ell'}) $ at $x_{\circ}$ for each $\ell'\in\mathbb{M}\setminus\{\ell_{\circ}\}$. From here and by \eqref{sup1} and \eqref{sup1.1}, it implies
		\begin{align}
			0&\geq	\tr[a_{\iota_{\circ}}\deri^{2}[\varphi-\eta]_{\ell_{\circ},\iota_{\circ}}]\notag\\
			&\geq c_{\iota_{\circ}}[\varphi-\eta]_{\ell_{\circ},\iota_{\circ}}-\sum_{\kappa\in\mathbb{I}\setminus\{\iota_{0}\}}q_{\ell_{\circ}}(\iota_{\circ},\kappa)[[\varphi-\eta]_{\ell_{\circ},\kappa}-[\varphi-\eta]_{{\ell}_{\circ},\iota_{\circ}}]\notag\\
			&\quad+\sum_{\ell'\in \mathbb{M}\backslash \{\ell_{\circ}\}}[\psi_{\delta}(\varphi_{\ell_{\circ},\iota_{\circ}}-\varphi_{\ell',\iota_{\circ}}-\vartheta_{\ell_{\circ},\ell'})-\psi_{\delta}(\eta_{\ell_{\circ},\iota_{\circ}}-\eta_{\ell',\iota_{\circ}}-\vartheta_{\ell_{\circ},\ell'})]\notag\\
			&\geq c_{\iota_{\circ}}[\varphi-\eta]_{\ell_{\circ},\iota_{\circ}}\quad\text{at}\ x_{\circ}.
		\end{align}
		Then, $[\varphi-\eta]_{\ell,\iota}(x)\leq [\varphi-\eta]_{\ell_{\circ},\iota_{\circ}}(x_{\circ})\leq0$ for $(x,\ell,\iota)\in\overline{\set}\times\mathbb{M}\times\mathbb{I}$, because of $c_{\iota_{\circ}}>0$ on $\overline\set$. Therefore, $\varphi\leq\eta$ for all $x\in\overline{\set}$.
	\end{proof}
	\begin{proof}[Proof of Proposition \ref{princ1.0}. Uniqueness]  The uniqueness of the solution $u$ to the NPDS \eqref{NPD.1} is obtained by contradiction. Assume that there are two solutions $u,v\in\mathcal{C}^{4,\alpha}_{m,n}$ to the NPDS \eqref{NPD.1}.  Taking $u$ and $v$ as a sub-solution and super-solution, respectively, from Lemma \ref{supsuper}, it follows immediately  that $u_{\ell,\iota}(x)-v_{\ell,\iota}(x)\leq u_{\ell_{\circ},\iota_{\circ}}(x_{\circ})-v_{\ell_{\circ},\iota_{\circ}}(x_{\circ})\leq0$ for $(x,\ell,\iota)\in\overline{\set}\times\mathbb{M}\times\mathbb{I}$. Taking now $ v$ and  $u$  as a sub-solution and super-solution, respectively, we obtain that $v_{\ell,\iota} -u_{\ell,\iota}\leq0$ on $\overline{\set}$ for $(\ell,\iota)\in\mathbb{M}\times\mathbb{I}$.  Therefore $u =v $ and from here we conclude that the NPDS \eqref{NPD.1} has a unique solution $u$, whose components belong to $\hol^{4,\alpha}(\overline{\set})$.
\end{proof}}

\section{Existence and uniqueness of the solutions to the HJB equations  \eqref{esd5} and \eqref{pc1}}\label{prop1}
To study the existence and uniqueness of the solutions $u$ and $u^{\varepsilon}$ to the variational inequalities \eqref{esd5} and \eqref{pc1}, respectively, we will proceed in the same way as in \cite{KM2022}, i.e.,  we will first  verify  that the sequence $\{u^{\varepsilon,\delta}\}_{(\varepsilon,\delta)\in(0,1)^{2}}$ is bounded,   uniformly in $(\varepsilon,\delta)$, with respect to the norms $\|\cdot\|_{\hol^{0}(\overline{\set})}$, $\|\cdot\|_{\hol^{1}_{\loc}({\set})}$ and $\|\cdot\|_{\hol^{2}_{\loc}({\set})}$; see Lemmas \ref{Lb10}--\ref{Lb1}; and then, for each $\varepsilon\in(0,1)$ fixed, $u^{\varepsilon}$ will be taken as limit of $u^{\varepsilon,\delta}$ when $\delta$ goes to zero. Finally, providing that $\{u^{\varepsilon}\}_{\varepsilon\in(0,1)}$ is well defined and is bounded uniformly with respect to the norms   $\|\cdot\|_{\hol^{0}(\overline{\set})}$, $\|\cdot\|_{\hol^{1}_{\loc}({\set})}$ and $\|\cdot\|_{\sob^{2}_{\loc}({\set})}$, we will see   $u$   as a limit of $u^{\varepsilon}$, when $\varepsilon$ goes to zero.

%For each $(\varepsilon,\delta)\in(0,1)^{2}$, let us consider $u^{\varepsilon,\delta}\in\mathcal{C}^{4,\alpha}_{m,n}$ the unique solution to \eqref{NPD.1}. 
\begin{lema}\label{Lb10}
	There exists $\overline{u}\in\mathcal{C}^{4,\alpha}_{m,n}$ independent of $\varepsilon$ and $\delta$ such that for each $(\ell,\iota)\in\mathbb{M}\times\mathbb{I}$,
	\begin{align}
		&0\leq u_{\ell,\iota}^{\varepsilon,\delta}\leq \overline{u}_{\ell,\iota}\quad \text{on}\ \overline{\set}.\label{ap1}
	\end{align}
	 {Additionally, if \eqref{h6} holds, then
		\begin{align}
			&\underline{u}_{\ell,\iota}\leq u_{\ell,\iota}^{\varepsilon,\delta}\leq \overline{u}_{\ell,\iota}\quad \text{on}\ \overline{\set},\label{ap1.0}\\
			&|\deri^{1}u^{\varepsilon,\delta}_{\ell,\iota}(x)|\leq C_{2}\quad \text{for}\ x\in\partial\set, \label{ap1.1}
		\end{align}
		for some positive constant $C_{2}$ indenpendent of $\varepsilon$ and $\delta$.}
\end{lema}
To prove Lemma \ref{Lb10}, we require first to see  the existence and uniqueness of the classical solution $\overline{u}=(\overline{u}_{\ell,\iota})_{(\ell,\iota)\in\mathbb{M}\times\mathbb{I}}$ to the problem
\begin{equation}\label{D1}
	\begin{split}
		[c_{\iota}-\dif_{\ell,\iota}]\overline{u}_{\ell,\iota}=   {h}_{\iota},\ \text{on}\ \set,\quad\text{s.t.}\ \overline{u}_{\ell,\iota}=f_{\iota},\ \text{in}\ \partial{\set}.
	\end{split}
\end{equation}

\begin{lema}\label{dm1}
	If \eqref{h0}, \eqref{a4} and \eqref{h2} hold, there exists a unique non-negative solution $\overline{u}$ to the Dirichlet problem \eqref{D1} such that $\overline{u}_{\ell,\iota}  \in\hol^{4,\alpha}(\overline{\set})$ for each $(\ell,\iota)\in\mathbb{M}\times\mathbb{I}$.
\end{lema}
{Though due to the results of Sweers \cite{S1992}, under weak assumptions in $h_{\iota}$ and the parameters of $\mathcal{L}_{\ell,\iota}$,  that guaranteed  the existence of the non-negative solution $v$ to \eqref{D1} in  a strong sense, it is possible to verify that, by fixed point arguments, \eqref{D1} has a classical solution under the assumptions imposed above; see \eqref{h0}, \eqref{a4} and \eqref{h2}. The reader can find the proof of Lemma \ref{dm1} in the appendix; see Subsection \ref{proof1}, since this is similar to the proof of  Proposition \ref{princ1.1}.}

\begin{proof}[Proof of Lemma \ref{Lb10}]
	 {By Lemma \ref{dm1}, let us  consider $\overline{u}\in\mathcal{C}^{4,\alpha}_{m,n}$  as the unique non-negative solution to the  Dirichlet problem \eqref{D1}, which is a super-solution to the NPDS \eqref{NPD.1}. Then, from Lemma \ref{supsuper}, we get that $u_{\ell,\iota}\leq\overline{u}_{\ell,\iota}$ for $(\ell,\iota)\in\mathbb{M}\times\mathbb{I}$.  Let us assume \eqref{h6} holds. It implies that $\underline{u}$ is a sub-solution to the NPDS \eqref{NPD.1}. Therefore $\underline{u}_{\ell,\iota}\leq {u}_{\ell,\iota}$ for $(\ell,\iota)\in\mathbb{M}\times\mathbb{I}$.  Now, we take a point $x$ in $\partial \set$ and a unit vector $\mathbb{n}_{x}$ outside of $\set$ such that it is not tangent to $\set$. Defining $\mathbb{v}=-\mathbb{n}_{x}$, we have that 
		\begin{align*}	
			\partial_{\mathbb{v}}u_{\ell,\iota}(x)=\lim_{\varrho\rightarrow0}\frac{u_{\ell,\iota}(x+\varrho\mathbb{v})-f_{\iota}(x)}{\varrho}\geq\lim_{\varrho\rightarrow0}\frac{\underline{u}_{\ell,\iota}(x+\varrho\mathbb{v})-f_{\iota}(x)}{\varrho}=\partial_{\mathbb{v}}\underline{u}_{\ell,\iota}(x).
		\end{align*}
		In the same way it gives $\partial_{\mathbb{v}}u_{\ell,\iota}(x)\leq \partial_{\mathbb{v}}\overline{u}_{\ell,\iota}(x)$. Since, $\mathbb{v}=-\mathbb{n}_{x}$ and $u_{\ell,\iota},\underline{u}_{\ell,\iota},\overline{u}_{\ell,\iota}\in\hol^{1}(\overline{\set})$, it yields 
		\begin{equation}\label{n1}
			\langle\mathbb{n}_{x},\deri^{1}\overline{u}(x)\rangle\leq\langle\mathbb{n}_{x},\deri^{1}u(x)\rangle\leq \langle\mathbb{n}_{x},\deri^{1}\underline{u}(x)\rangle.
		\end{equation}
		If the vector $\mathbb{n}_{x}$ is inside of $\set$, proceeding as before, we have
		\begin{equation}\label{n2}
			\langle\mathbb{n}_{x},\deri^{1}\underline{u}(x)\rangle\leq\langle\mathbb{n}_{x},\deri^{1}u(x)\rangle\leq \langle\mathbb{n}_{x},\deri^{1}\overline{u}(x)\rangle
		\end{equation}
		Therefore, from \eqref{n1} and \eqref{n2}, it gives  $$|\deri^{1}u_{\ell,\iota}(x)|\leq\max\{|\deri^{1}\underline{u}_{\ell,\iota}(x)|,|\deri^{1}\overline{u}_{\ell,\iota}(x)|\}\defeq C_{2}.$$ }
\end{proof}

\begin{rem}\label{R1}
	From now on, we  consider cut-off functions $\varpi\in\hol^{\infty}_{\comp}( \set)$    satisfying $0\leq\varpi\leq1$, $\varpi=1$ on the open ball $B_{\beta r}\subset B_{\beta'r}\subset\set$ and $\varpi=0$ on $\set  \setminus B_{\beta' r}$, with $r>0$, $\beta'= \frac{\beta+1}{2}$ and $\beta\in(0,1]$.  It is  {also} assumed that $\|\varpi\|_{\hol^{2}(\overline{B_{\beta r}})}\leq K_{2}$  where $K_{2}>0$ is a constant independent of $\varepsilon$ and $\delta$.
\end{rem}
\begin{lema}\label{Lb1}
	There exist  positive constants $C_{3},C_{4}$ independent of $\varepsilon,\delta$ such that for each $x\in\overline{\set}$
	\begin{align}
		&\varpi(x)|\deri^{1}u_{\ell,\iota}^{\varepsilon,\delta}(x)|\leq  C_{3},\label{ap2}\\
		&\varpi(x)|\deri^{2}u_{\ell,\iota}^{\varepsilon,\delta}(x)|\leq C_{4}.\label{ap3}
	\end{align}
	 {Additionally, if \eqref{h6} holds, then there exists a constant $C_{5}$ independent of $\varepsilon,\delta$ such that for each $x\in\overline{\set}$
		\begin{align}
			&|\deri^{1}u_{\ell,\iota}^{\varepsilon,\delta}(x)|\leq C_{5}.\label{ap2.2}
	\end{align}}
\end{lema}
The proof of the lemma above is a slight modification of the proof  of the similar conclusions in Lemma 3.1 of \cite{KM2022}. So,  we shall sketch some of those ideas in the appendix and remit to the reader to this paper for more details.

\subsection{ {Sketch of the proofs of Propositions \ref{M1} and \ref{princ1.1} }}
Let  $\varepsilon\in(0,1)$ and $(\ell,\iota)\in\mathbb{M}\times\mathbb{I}$ be  fixed.  {Under  assumptions \eqref{h0}--\eqref{h2}}, by \eqref{ap1}, \eqref{ap2} and \eqref{ap3}, using Arzel\`a-Ascoli compactness criterion (see \cite[p. 718]{evans2}) and that for each $p\in(1,\infty)$, $(\Lp^{p}(B_{\beta r}),\|\cdot\|_{\Lp^{p}(B_{r})})$, with $B_{r}\subset\set$, is a reflexive space (see \cite[Thm. 2.46, p. 49]{adams}), we get that there exist  a sub-sequence $\{u^{\varepsilon,\delta_{\hat{n}}}_{\ell,\iota}\}_{\hat{n}\geq1}$ of $\{u^{\varepsilon,\delta}_{\ell,\iota}\}_{\delta\in(0,1)}$  and a function $w^{\varepsilon}_{\ell,\iota}$ in $\sob^{2,\infty}_{\loc}(\set)$ such that
\begin{equation}
	\begin{split} \label{conv1.0}
		&u^{\varepsilon,\delta_{\hat{n}}}_{\ell,\iota}\underset{\delta_{\hat{n}}\rightarrow0}{\longrightarrow}w^{\varepsilon}_{\ell,\iota}\ \text{in}\hol^{1}_{\loc}(\set),\\  &\partial_{ij}u^{\varepsilon,\delta_{\hat{n}}}_{\ell,\iota}\underset{\delta_{\hat{n}}\rightarrow0}{\longrightarrow}\partial_{ij}w^{\varepsilon}_{\ell,\iota}, \text{weakly}\ \Lp^{p}_{\loc}(\set),\ \text{for each}\ p\in(1,\infty).
	\end{split}
\end{equation}
Taking
\begin{equation}\label{sol1}
	u^{\varepsilon}_{\ell,\iota}(x)\eqdef
	\begin{cases}
		w^{\varepsilon}_{\ell,\iota}(x)&\quad\text{if}\ x\in\set,\\
		f_{\iota}(x)&\quad\text{if}\ x\in\partial\set,
	\end{cases}
\end{equation}
and considering that $0\leq u^{\varepsilon,\delta_{\hat{n}}}_{\ell,\iota}\leq v_{\ell,\iota}$ on $\set$ and $ u^{\varepsilon,\delta_{\hat{n}}}_{\ell,\iota}=v_{\ell,\iota}=f_{\iota}$ in $\partial\set$, it implies that 
\begin{equation}\label{econv1.0}
	0\leq u^{\varepsilon}_{\ell,\iota}\leq v_{\ell,\iota}\  \text{on}\ \set\  \text{and}\   u^{\varepsilon}_{\ell,\iota}=f_{\iota}\  \text{in}\  \partial\set. 
\end{equation}
Therefore $u^{\varepsilon}_{\ell,\iota}\in \hol^{0}(\overline{\set})\cap\sob^{2,\infty}_{\loc}(\set)$.  Now, using \eqref{ap1}, \eqref{ap2}, \eqref{ap3}  and   \eqref{conv1.0},  the following inequalities hold  for each $(\ell,\iota)\in\mathbb{M}\times\mathbb{I}$,
\begin{equation}\label{ineqe1.0}
	\begin{split}
		&\varpi(x)|\deri^{1}u^{\varepsilon}_{\ell,\iota}(x)|\leq C_{3} \ \text{for}\ x\in\overline{\set},\\ &\|\deri^{2}u^{\varepsilon}_{\ell,\iota}\|_{\Lp^{p}(B_{\beta r})}\leq C_{4} \ \text{for each}\ p\in(1,\infty)\ \text{and}\ B_{\beta r}\subset\set  .
	\end{split}
\end{equation}	
Then, from \eqref{econv1.0}, \eqref{ineqe1.0} and by the same criterion in \eqref{conv1.0}, we have that there exist  a sub-sequence $\{u^{\varepsilon_{n}}_{\ell,\iota}\}_{n\geq1}$ of $\{u^{\varepsilon}_{\ell,\iota}\}_{\varepsilon\in(0,1)}$  and a $w_{\ell,\iota}$ in $\sob^{2,\infty}_{\loc}(\set)$ such that
\begin{equation*}%\label{sol2}
	\begin{split} 
		&u^{\varepsilon_{\hat{n}}}_{\ell,\iota}\underset{\varepsilon_{\hat{n}}\rightarrow0}{\longrightarrow}w_{\ell,\iota}\ \text{in}\hol^{1}_{\loc}(\set),\\ &\text{$\partial_{ij}u^{\varepsilon_{n}}_{\ell,\iota}\underset{\varepsilon_{\hat{n}}\rightarrow0}{\longrightarrow}\partial_{ij}w_{\ell,\iota}$,  weakly $ \Lp^{p}_{\loc}(\set)$, for each $p\in(1,\infty)$.}
	\end{split}
\end{equation*}
Define 
\begin{equation}\label{sol.0}
	u_{\ell,\iota}(x)=
	\begin{cases}
		w_{\ell,\iota}(x)&\quad\text{if}\ x\in\set,\\
		f_{\iota}(x)&\quad\text{if}\ x\in\partial\set.
	\end{cases}
\end{equation}
Since \eqref{econv1.0} holds, we get that
\begin{equation*}
	0\leq u_{\ell,\iota}\leq v_{\ell,\iota}\  \text{on}\ \set\  \text{and}\   u_{\ell,\iota}=f_{\iota}\  \text{in}\  \partial\set. 
\end{equation*}
Therefore $u^{\varepsilon}_{\ell,\iota}\in \hol^{0}(\overline{\set})\cap\sob^{2,\infty}_{\loc}(\set)$.

To finalize this  section, let us remark that $u^{\varepsilon}$ and $u$, taken as in \eqref{sol1} and \eqref{sol.0} are the unique solutions to \eqref{pc1} and  \eqref{esd5}, respectively.  {Since the proofs of the previous asseverations are a slight modification of the proofs of Theorem 2.1 and Proposition 2.1 of \cite{KM2022},   we omit them and   remit  the reader to this paper for more details; See \ref{proof4} and \ref{proof5}.} 

 {Additionally,   assuming  \eqref{h6} holds. by \eqref{ap1.0}, \eqref{ap1.1} and \eqref{ap2.2},   it follows that the entries to the unique solution $u^{\varepsilon}=(u^{\varepsilon}_{\ell,\iota})_{(\ell,\iota)\in\mathbb{M}\times\mathbb{I}}$ to \eqref{pc1} belong to $\hol^{1}(\overline{\set})\cap\sob^{2,\infty}_{\loc}(\set)$ and
	\begin{equation}
		\begin{split} \label{conv1}
			&u^{\varepsilon,\delta_{\hat{n}}}_{\ell,\iota}\underset{\delta_{\hat{n}}\rightarrow0}{\longrightarrow}u^{\varepsilon}_{\ell,\iota}\ \text{in}\  \hol^{1}(\overline{\set})\\ &\partial_{ij}u^{\varepsilon,\delta_{\hat{n}}}_{\ell,\iota}\underset{\delta_{\hat{n}}\rightarrow0}{\longrightarrow}\partial_{ij}u^{\varepsilon}_{\ell,\iota},\   \text{weakly}\  \Lp^{p}_{\loc}(\set),\ \text{for each}\ p\in(1,\infty).
		\end{split}
	\end{equation}
	Meanwhile  by \eqref{ap1.0} \eqref{ap2.2} and \eqref{conv1},  the following inequalities hold  for each $(\ell,\iota)\in\mathbb{M}\times\mathbb{I}$,
	\begin{equation}\label{ineqe1}
		\underline{u}_{\ell,\iota}\leq u^{\varepsilon}_{\ell,\iota}\leq \overline{u}_{\ell,\iota}\quad\text{and}\quad |\deri^{1}u^{\varepsilon}_{\ell,\iota}|\leq C_{5},\quad\text{in}\ \overline{\set}.
	\end{equation}	
	Besides, for each $B_{\beta r}\subset\set$, from \eqref{ap3}, it follows that there exists a positive constant $C_{6}$ independent of $\varepsilon$ such that 
	\begin{equation}\label{ineqe2}
		||\deri^{2}u^{\varepsilon}_{\ell,\iota}||_{\Lp^{p}(B_{\beta r})}\leq C_{6},\quad \text{for each}\ p\in(1,\infty). 
	\end{equation}
	Then, from \eqref{ineqe1} and \eqref{ineqe2}, we get that the entries to the unique solution $u=(u_{\ell,\iota})_{(\ell,\iota)\in\mathbb{M}\times\mathbb{I}}$ to \eqref{esd5}  belong to   $\hol^{1}(\overline{\set})\cap\sob^{2,\infty}_{\loc}(\set)$ such that
	\begin{equation}
		\begin{split} \label{econv1}
			&\text{$u^{\varepsilon_{n}}_{\ell,\iota}\underset{\varepsilon_{n}\rightarrow0}{\longrightarrow}u_{\ell,\iota}$ in $\hol^{1}(\overline{\set})$,}\\ &\text{$\partial_{ij}u^{\varepsilon_{n}}_{\ell,\iota}\underset{\varepsilon_{n}\rightarrow0}{\longrightarrow}\partial_{ij}u_{\ell,\iota}$,  weakly $ \Lp^{p}_{\loc}(\set)$, for each $p\in(1,\infty)$.}
		\end{split}
	\end{equation}
	We summarise the seen above in the following proposition.
	\begin{prop}
		If \eqref{h0}--\eqref{h2} and \eqref{h6} hold, then $u^{\varepsilon}=(u^{\varepsilon}_{\ell,\iota})_{(\ell,\iota)\in\mathbb{M}\times\mathbb{I}}$, $u=(u_{\ell,\iota})_{(\ell,\iota)\in\mathbb{M}\times\mathbb{I}} $, defined in \eqref{sol1} and \eqref{sol.0}, are the unique solutions to the HJB equations \eqref{pc1} and \eqref{esd5}, respectively, whose entries belong to the space $\hol^{1}(\overline{\set})\cap\sob^{2,\infty}_{\loc}(\set)$.
	\end{prop}
}

\section{$\varepsilon$-PACS control problem and proof of Theorem   \ref{verf2}}\label{Pp1}

In this section, we shall  verify that the value functions $V$  given in \eqref{vf1} agrees with the solution $u$ to the HJB equation \eqref{esd5}  {on $\overline{\set}$  under assumptions \eqref{h3} and {\eqref{a4}--\eqref{h5}}.  To prove it, firstly, we study an  {$\varepsilon$-PACS} control problem that is closely related to the value function problem seen previously.   

The penalized control set $\mathcal{U}^{\varepsilon}$ is defined by 
\begin{gather}\label{p1}
	\mathcal{U}^{\varepsilon}
	=\{\xi=(\mathbb{n},\zeta)\in \mathcal{U}:\zeta_t\text{ is absolutely continuous, }0\leq \dot{\zeta}_t\leq 2C/\varepsilon\}
\end{gather}
where $\varepsilon\in(0,1)$ is fixed and $C$ is a positive constant independent of $\varepsilon$. Let $( x_{0}, \ell_{0},\iota_{0})\in \overline{\mathcal{O}}\times \mathbb{M} \times \mathbb{I}$ be fixed. The controlled process $(X^{\xi,\varsigma},J^{\varsigma},I)$ evolves as 
\begin{gather}\label{esd1}
	\begin{split}
		&X^{\xi,\varsigma}_{\tmt}=X^{\xi,\varsigma}_{\tilde{\tau}_{i}} -\displaystyle\int_{\tilde{\tau}_{i}}^{\tmt}
		[b(X^{\xi,\varsigma}_{\tms},I^{(\ell_{i})}_{\tms})+\mathbb{n}_{\tms}\dot{\zeta}_{\tms}]\der \tms+\displaystyle\int_{\tilde{\tau}_{i}}^{\tmt}\sigma(X^{\xi,\varsigma}_{\tms},I^{(\ell_{i})}_{\tms})\der W_{\tms},\\
		&I_{\tmt}=I^{(\ell_{i})}_{\tmt}\ \text{and}\ J^{\varsigma}_{\tmt}=\ell_{i}\quad\text{for}\ \tmt\in[{\tilde\tau}_{i},\tilde{\tau}_{i+1})\ \text{ and }\ i\geq0,
	\end{split}
\end{gather}
where $\tilde{\tau}_i=\tau_i\wedge \tau$ and $\tau$ is the first exit time of $X^{\xi,\varsigma}$ from the set $\mathcal{O}$.  Defining the Legendre transform of $H^{\varepsilon}_{\iota}(\gamma,x)\eqdef H^{\varepsilon}(\gamma,x,\iota)=\psi_\varepsilon\left(|\gamma|^2-g^2_{\iota}(x)\right)$, with $(x,\iota)\in \overline{\mathcal{O}}\times \mathbb{I}$, as 
\begin{gather*}
	l^{\varepsilon}_{\iota}(y,x)\eqdef l^{\varepsilon}(y,x,\iota)
	\eqdef\sup_{\gamma \in \R^d}
	\{\left<\gamma,y\right>-H^{\varepsilon}_{\iota}(\gamma,x)\},\quad \text{for } y\in \R^d
\end{gather*}
the corresponding penalized functional cost for $(\xi,\varsigma)\in\mathcal{U}^{\varepsilon}\times\mathcal{S}$ is defined as
\begin{align}\label{pen1}
	\mathcal{V}_{\xi,\varsigma}( x_{0},\ell_{0},\iota_{0})&\eqdef\E_{ x_{0},\ell_{0},\iota_{0}}\biggr[\int_{[0,\tau]}\expo^{-r(\mathpzc{t})}[ h (X^{\xi,\varsigma}_{\mathpzc{t}},I_{\tmt})
	+l^\varepsilon(\dot{\zeta}_{\tmt}\mathbb{n}_{\tmt},X_{\tmt}^{\varepsilon,\varsigma},I_{\tmt})]
	\der \mathpzc{t}]\biggr]\notag\\
	&\quad+\sum_{i\geq0}\E_{ x_{0},\ell_{0},\iota_{0}}\big[\expo^{-r(\tau_{i+1})}\vartheta_{\ell_{i},\ell_{i+1}}\uno_{\{\tau_{i+1}<\tau\}}\big]\notag\\
	&\quad+\E_{ x_{0},\ell_{0},\iota_{0}}[\expo^{-r(\tau)}f(X^{\xi,\varsigma}_{\tau},I_{\tau})\uno_{\{\tau<\infty\}}],
\end{align}
and the value function is then given by 
\begin{equation}\label{vw1}
	V^{\varepsilon}_{\ell_{0}}( x_{0},\iota_{0})\eqdef \inf_{\xi,\varsigma} \mathcal{V}_{\xi,\varsigma}( x_{0},\ell_{0},\iota_{0}),
\end{equation}
where its HJB equation {takes the form}
\begin{equation}\label{p13.0.1.0}
	\begin{split}
		\max\bigg\{[c_{\iota}-\dif_{\ell,\iota}]u_{\ell,\iota}^{\varepsilon}+\sup_{y\in \R^d}
		\{\langle \deri^{1}u^{\varepsilon}_{\ell,\iota},y\rangle
		-l^{\varepsilon}_{\iota}(y,\cdot)\}
		-  h_{\iota},u_{\ell,\iota}^{\varepsilon}-\mathcal{M}_{\ell,\iota}u^{\varepsilon} \bigg\}&= 0,\ \text{on}\ \set,\\
		\text{s.t.}\ u_{\ell,\iota}^{\varepsilon}&=f_{\iota},\ \text{in}\ \partial{\set},
	\end{split}
\end{equation}
where $\mathcal{M}_{\ell,\iota}$ and  $\dif_{\ell,\iota}$ are as in  \eqref{p6.0}. Observe that \eqref{p13.0.1.0} can be rewritten as \eqref{pc1} because of 
$H^{\varepsilon}_{\iota}(\gamma,x)=\sup_{y\in\R^{d}}\{\langle \gamma,y\rangle-l^{\varepsilon}_{\iota}(y,x)\}$.
%The equality above is true, due to $H^{\varepsilon}_{\ell}(\cdot,x)\in\hol^{2}(\R^{d})$..

To facilitate the notation of this section, let us denote $\dif_{\ell,\iota}\hat{f}_{\ell,\iota}$ and $\mathcal{M}_{\ell,\iota}\hat{f}$ by $\dif_{\ell,\iota}\hat{f}_{\ell} (\cdot,\iota)$ and $\mathcal{M}_{\ell}\hat{f} (\cdot,\iota)$, respectively.

 {Let $(X^{\xi,\varsigma},J^{\varsigma},I)$ evolve as \eqref{es1}, with $(\xi,\varsigma)\in\mathcal{U}\times\mathcal{S}$ and initial state $( x_{0},\ell_{0},\iota_{0})\in {\set}\times\mathbb{M}\times\mathbb{I}$. Recall that $(\xi,\varsigma)$ satisfies \eqref{cont.1} and \eqref{cont.2}.}

Let us start showing a general result which will be helpful for the purposes of the section.
\begin{lema}\label{Ito1}
	Let $(X^{\xi,\varsigma},J^{\varsigma},I)$ evolve as \eqref{es1}, with $(\xi,\varsigma)\in\mathcal{U}\times\mathcal{S}$ and initial state $( x_{0},\ell_{0},\iota_{0})\in {\set}\times\mathbb{M}\times\mathbb{I}$. Let $\hat{f}=(\hat{f}_{1},\dots,\hat{f}_{m})$ be a sequence of real valued function such that $\hat{f}_{\ell}(\cdot,\iota)\in\hol^{2}(\overline{\set})$ for $(\ell,\iota)\in\mathbb{M}\times\mathbb{I}$. Take $\hat{\tau}_{0}^{q}\eqdef0$ and    $\hat{\tau}_{i}^{q}\eqdef\tilde{\tau}_{i}\wedge\inf\{ \tmt>\tau_{i-1}:X_{\tmt}\notin\set_{q}\}$, with $\tilde{\tau}_{i}={\tau}_{i}\wedge\tau$,  $i\geq1$, $\set_{q}\eqdef\{x\in\set:\dist(x,\partial\set )>1/q\}$ and $q$ a positive integer  {large enough} such that $X_{0-}=x_{0}\in\set_{q}$. Then
	\begin{align}\label{m6}
		&\E_{x_{0},\iota_{0},\ell_{0}}[\expo^{-r(\hat{\tau}^{q}_{i})}\hat{f}_{\ell_{i}}(X^{\xi,\varsigma}_{\hat{\tau}^{q}_{i}},I_{\hat{\tau}^{q}_{i}})\uno_{\{\tau_{i}<\tau\}}]\notag\\
		&=\E_{x_{0},\iota_{0},\ell_{0}}\bigg[\bigg\{\expo^{-r(\hat{\tau}^{q}_{i+1})}\hat{f}_{\ell_{i}}(X^{\xi,\varsigma}_{\hat{\tau}^{q}_{i+1}},I_{\hat{\tau}^{q}_{i+1}})-\sum_{\hat{\tau}^{q}_{i}<\tms\leq\hat{\tau}^{q}_{i+1}}\expo^{-r(\tms)}\mathcal{J}[X^{\xi,\varsigma}_{\tms},I_{\tms}, \hat{f}_{\ell_{i}}]\notag\\
		&\quad+\int_{\hat{\tau}^{q}_{i}+}^{\hat{\tau}^{q}_{i+1}}\expo^{-r(\tms)}[[c(X^{\xi,\varsigma}_{\tms},I_{\tms})\hat{f}_{\ell_{i}}(X^{\xi,\varsigma}_{\tms},I_{\tms})\notag\\
		&\quad\quad-\dif_{\ell_{i},I_{\tms}}\hat{f}_{\ell_{i}}(X^{\xi,\varsigma}_{\tms},I_{\tms})]\der\tms+\langle\deri^{1} \hat{f}_{\ell_{i}}(X^{\xi,\varsigma}_{\tms},I_{\tms}),\mathbb{n}_{\tms}\rangle\der\zeta^{\comp}_{\tms}]\bigg\}\uno_{\{\tau_{i}<\tau\}}\bigg],
	\end{align}
where $\displaystyle\int_{a+}^{b}$ defines the integral operator on the interval $[a,b)$, $\xi^{\comp}$ is the continuous part of the process $\xi$, and 
 \begin{align}\label{m3}
	\mathcal{J}[X_{\tms},I_{\tms}, \hat{f}_{\ell_{i}}]\eqdef
	\hf_{\ell_{i}}(X_{\tms},I_{\tms})-\hf_{\ell_{i}}(X_{\tms-},I_{\tms})
	=\hf_{\ell_{i}}(X_{\tms-}-\mathbb{n}_{\tms}\Delta \zeta_{\tms},I_{\tms})-\hf_{\ell_{i}}(X_{\tms-},I_{\tms}).
\end{align}
\end{lema}	
From now on, for simplicity of notation, we replace $X^{\xi,\varsigma}$ by $X$ in the  proofs of the results. 	
\begin{proof}
	   For each $i\geq0$, we assign $\rho^{\ell_{i}}_{0}\eqdef\hat{\tau}^{q}_{i}\leq\rho^{\ell_{i}}_{1}<\rho^{\ell_{i}}_{2}<\dots < \rho^{\ell_{i}}_{j-1}\leq\hat{\tau}^{q}_{i+1}\defeq\rho^{\ell_{i}}_{j} $, for some $j\geq 0$, as all the possible random times where the process $I$ has a jump on the interval time $[\hat{\tau}^{q}_{i},\hat{\tau}^{q}_{i+1}]$, i.e. $I_{\tmt}=\iota^{\ell_{i}}_{j'}$ if $\tmt\in[\rho^{\ell_{i}}_{j'},\rho^{\ell_{i}}_{j'+1})$, for $j'\in\{0,1,\dots,j-1\}$. Using integration by parts and It\^o's formula in $\expo^{-r(\tmt)}\hat{f}_{\ell_{i}}(X_{\tmt},\iota^{\ell_{i}}_{j'})$ on $ [\rho^{\ell_{i}}_{j'},\rho^{\ell_{i}}_{j'+1}]$ for $j'\in\{0,1,\dots,j-1\}$; see \cite[Theorem 33]{pro},  we get that 
	\begin{align}\label{m1}
	&\expo^{-r(\rho^{\ell_{i}}_{j'})}\hat{f}_{\ell_{i}}(X_{\rho^{\ell_{i}}_{j'}},\iota^{\ell_{i}}_{j'})-\expo^{-r(\rho^{\ell_{i}}_{j'+1})}\hat{f}_{\ell_{i}}(X_{\rho^{\ell_{i}}_{j'+1}},\iota^{\ell_{i}}_{j'+1})\notag\\
	&\quad=\int_{\rho^{\ell_{i}}_{j'}+}^{\rho^{\ell_{i}}_{j'+1}}\expo^{-r(\tms)}[[c(X_{\tms},\iota^{\ell_{i}}_{j'})\hat{f}_{\ell_{i}}(X_{\tms},\iota^{\ell_{i}}_{j'})-\widetilde{\dif}_{\iota^{\ell_{i}}_{j'}}\hat{f}_{\ell_{i}}(X_{\tms},\iota^{\ell_{i}}_{j'})]\der\tms+\langle\deri^{1} \hat{f}_{\ell_{i}}(X_{\tms},\iota^{\ell_{i}}_{j'}),\mathbb{n}_{\tms}\rangle\der\zeta^{\comp}_{\tms}]\notag\\
	&\quad\quad-\int_{\rho^{\ell_{i}}_{j'}+}^{\rho^{\ell_{i}}_{j'+1}}\expo^{-r(\tms)}\langle\deri^{1}\hat{f}_{\ell_{i}}(X_{\tms},\iota^{\ell_{i}}_{j'}),\sigma(X_{\tms},\iota^{\ell_{i}}_{j'})\der W_{\tms}\rangle-\sum_{\rho^{\ell_{i}}_{j'}<\tms<\rho^{\ell_{i}}_{j'+1}}\expo^{-r(\tms)}\mathcal{J}[X_{\tms},\iota^{\ell_{i}}_{j'}, \hat{f}_{\ell_{i}}]\notag\\
	&\quad\quad-\expo^{-r(\rho^{\ell_{i}}_{j'+1})}[\hf_{\ell_{i}}(X_{{\rho^{\ell_{i}}_{j'+1}}},\iota^{\ell_{i}}_{j'+1})-\hf_{\ell_{i}}(X_{\rho^{\ell_{i}}_{j'+1}-},\iota^{\ell_{i}}_{j'+1})]\notag\\
	&\quad\quad-\expo^{-r(\rho^{\ell_{i}}_{j'+1})}[\hf_{\ell_{i}}(X_{{\rho^{\ell_{i}}_{j'+1}}-},\iota^{\ell_{i}}_{j'+1})-\hf_{\ell_{i}}(X_{\rho^{\ell_{i}}_{j'+1}-},\iota^{\ell_{i}}_{j'})],
\end{align}
%and for $j'=j-1$,
%\begin{align}\label{m2}
%	&\expo^{-r(\rho^{\ell_{i}}_{j-1})}\hat{f}_{\ell_{i}}(X_{\rho^{\ell_{i}}_{j-1}},\iota^{\ell_{i}}_{j-1})-\expo^{-r(\rho^{\ell_{i}}_{j})}\hat{f}_{\ell_{i}}(X_{\rho^{\ell_{i}}_{j}},\iota^{\ell_{i}}_{j})\notag\\
%	&\quad=\int_{\rho^{\ell_{i}}_{j-1}+}^{\rho^{\ell_{i}}_{j}}\expo^{-r(\tms)}[[c(X_{\tms},\iota^{\ell_{i}}_{j})\hat{f}_{\ell_{i}}(X_{\tms},\iota^{\ell_{i}}_{j})-\widetilde{\dif}_{\iota^{\ell_{i}}_{j}}\hat{f}_{\ell_{i}}(X_{\tms},\iota^{\ell_{i}}_{j})]\der\tms+\langle\deri^{1} \hat{f}_{\ell_{i}}(X_{\tms},\iota^{\ell_{i}}_{j}),\mathbb{n}_{\tms}\rangle\der\zeta^{\comp}_{\tms}]\notag\\
%	&\quad\quad-\int_{\rho^{\ell_{i}}_{j-1}+}^{\rho^{\ell_{i}}_{j}}\expo^{-r(\tms)}\langle\deri^{1}\hat{f}_{\ell_{i}}(X_{\tms},\iota^{\ell_{i}}_{j}),\sigma_{\ell_{i}}(X_{\tms},\iota^{\ell_{i}}_{j})\der W_{\tms}\rangle-\sum_{\rho^{\ell_{i}}_{j-1}<\tms\leq\rho^{\ell_{i}}_{j}}\expo^{-r(\tms)}\mathcal{J}[X_{\tms},\iota^{\ell_{i}}_{j-1}, \hat{f}_{\ell_{i}}]\notag\\
%	&\quad\quad-\expo^{-r(\rho^{\ell_{i}}_{j})}[\hf_{\ell_{i}}(X_{{\rho^{\ell_{i}}_{j}}},\iota^{\ell_{i}}_{j})-\hf_{\ell_{i}}(X_{\rho^{\ell_{i}}_{j}},\iota^{\ell_{i}}_{j-1})],
%\end{align}
where   $\widetilde{\dif}_{\iota}\hf_{\ell}(\cdot,\iota)\eqdef \tr[a(\cdot,\iota) \deri^{2}\hf_{\ell}(\cdot,\iota)]-\langle b(\cdot,\iota),\deri^{1}\hf_{\ell}(\cdot,\iota)\rangle$.
 Taking into account \eqref{m1} it can be verified that   
 \begin{align}\label{m4}
 &\expo^{-r(\hat{\tau}^{q}_{i})}\hat{f}_{\ell_{i}}(X_{\hat{\tau}^{q}_{i}},I_{\hat{\tau}^{q}_{i}})-\expo^{-r(\hat{\tau}^{q}_{i+1})}\hat{f}_{\ell_{i}}(X_{\hat{\tau}^{q}_{i+1}},I_{\hat{\tau}^{q}_{i+1}})\notag\\
 &\quad=\int_{\hat{\tau}^{q}_{i}+}^{\hat{\tau}^{q}_{i+1}}\expo^{-r(\tms)}[[c(X_{\tms},I_{\tms})\hat{f}_{\ell_{i}}(X_{\tms},I_{\tms})-\widetilde{\dif}_{I_{\tms}}\hat{f}_{\ell_{i}}(X_{\tms},I_{\tms})]\der\tms+\langle\deri^{1} \hat{f}_{\ell_{i}}(X_{\tms},I_{\tms}),\mathbb{n}_{\tms}\rangle\der\zeta^{\comp}_{\tms}]\notag\\
 &\quad\quad-\int_{\hat{\tau}^{q}_{i}+}^{\hat{\tau}^{q}_{i+1}}\expo^{-r(\tms)}\langle\deri^{1}\hat{f}_{\ell_{i}}(X_{\tms},I_{\tms}),\sigma(X_{\tms},I_{\tms})\der W_{\tms}\rangle\notag\\
 &\quad\quad-\sum_{\hat{\tau}^{q}_{i}<\tms\leq\hat{\tau}^{q}_{i+1}}\expo^{-r(\tms)}\{\mathcal{J}[X_{\tms},I_{\tms}, \hat{f}_{\ell_{i}}]+\hf_{\ell_{i}}(X_{\tms-},I_{\tms})-\hf_{\ell_{i}}(X_{\tms-},I_{\tms-})\}.
 \end{align}
Let us consider $\Delta^{\ell_{i}}_{\iota,\kappa}$, with $\iota\neq\kappa$, as the consecutive, with respect to the lexicographic ordering on $\mathbb{I}\times\mathbb{I}$, left-closed, right-open intervals of the real line, which have length $q_{\ell_{i}}(\iota,\kappa)$. Defining $\bar{h}_{\ell_{i}}:\mathbb{I}\times\R\longrightarrow\R$ as
\begin{equation*}
	\bar{h}_{\ell_{i}}(\iota,z)=\sum_{\kappa\in\mathbb{I}\setminus\{\iota\}}(\kappa-\iota)\uno_{\{z\in\Delta^{\ell_{i}}_{\iota,\kappa}\}},
\end{equation*}
we have that \eqref{q1} is equivalent to 
\begin{equation*}
\der I^{(\ell_{i})}_{\tmt}=\int_{\R}\bar{h}_{\ell_{i}}(I^{(\ell_{i})}_{\tmt-},z)N(\der \tmt,\der z),
\end{equation*}
where $N(\der\tmt,\der z)$ is a Poisson random measure with intensity $\der \tmt\times \nu(\der z)$ independent of $W$, and $\nu$ is the Lebesgue measure on $\R$; for more details see, e.g. \cite{YZ2010}. From here and recalling that $I$ is governed by $Q_{\ell_{i}}$  on $ (\hat{\tau}^{q}_{i},\hat{\tau}^{q}_{i+1}]$, we have the next equivalent expression for \eqref{m4},  
\begin{align}\label{m5}
	&\expo^{-r(\hat{\tau}^{q}_{i})}\hat{f}_{\ell_{i}}(X_{\hat{\tau}^{q}_{i}},I_{\hat{\tau}^{q}_{i}})-\expo^{-r(\hat{\tau}^{q}_{i+1})}\hat{f}_{\ell_{i}}(X_{\hat{\tau}^{q}_{i+1}},I_{\hat{\tau}^{q}_{i+1}})\notag\\
	&\quad=\int_{\hat{\tau}^{q}_{i}+}^{\hat{\tau}^{q}_{i+1}}\expo^{-r(\tms)}[[c(X_{\tms},I_{\tms})\hat{f}_{\ell_{i}}(X_{\tms},I_{\tms})-\dif_{\ell_{i},I_{\tms}}\hat{f}_{\ell_{i}}(X_{\tms},I_{\tms})]\der\tms+\langle\deri^{1} \hat{f}_{\ell_{i}}(X_{\tms},I_{\tms}),\mathbb{n}_{\tms}\rangle\der\zeta^{\comp}_{\tms}]\notag\\
	&\quad\quad-\widetilde{\mathcal{M}}[\hat{\tau}^{q}_{i},\hat{\tau}^{q}_{i+1};X,I,\hf_{\ell_{i}}]-\sum_{\hat{\tau}^{q}_{i}<\tms\leq\hat{\tau}^{q}_{i+1}}\expo^{-r(\tms)}\mathcal{J}[X_{\tms},I_{\tms}, \hat{f}_{\ell_{i}}].
\end{align}
where the process 
\begin{multline*}%\label{mar1}
	\widetilde{\mathcal{M}}[\hat{\tau}^{q}_{i},\tmt\wedge\hat{\tau}^{q}_{i+1};X,I,\hf_{\ell_{i}}]\eqdef\int_{\hat{\tau}^{q}_{i}+}^{\tmt\wedge\hat{\tau}^{q}_{i+1}}\expo^{-r(\tms)}\langle\deri^{1} \hf_{\ell_{i}}(X_{\tms},I_{\tms}),\sigma(X_{\tms},I_{\tms})\rangle\der  W_{\tms}\\
	+\int_{\hat{\tau}^{q}_{i}+}^{\tmt\wedge\hat{\tau}^{q}_{i+1}}\int_{\R}\expo^{-r(\tms)}[\hf_{\ell_{i}}(X_{\tms-},I_{\hat{\tau}^{q}_{i}}+\bar{h}_{\ell_{i}}(I_{\tms-},z))-\hf_{\ell_{i}}(X_{\tms-},I_{\tms-})][N(\der \tms,\der z)-\der\tms\times\nu(\der z))]
\end{multline*}
is a square-integrable martingale. Therefore, multiplying by $\uno_{\{\tau_{i}<\tau\}}$ and taking expected value in both sides of \eqref{m5}, we get \eqref{m6}.
\end{proof}

\subsection{Verification Lemma for  {$\varepsilon$-PACS} control problem}
Let   $ (X^{\xi,\varsigma},J^{\varsigma},I)$ evolve as \eqref{esd1}, with $(\xi,\varsigma)\in\mathcal{U}^{\varepsilon}\times\mathcal{S}$ and initial state $( x_{0},\ell_{0},\iota_{0})\in \overline{\set}\times\mathbb{M}\times\mathbb{I}$. Under assumptions \eqref{h3}--\eqref{h2}  {and \eqref{h6}} , Lemmas \ref{lv1} and \ref{convexu1.0} shall be proven.

\begin{lema}[Verification Lemma for  {$\varepsilon$-PACS} control problem. First part]\label{lv1} Let $\varepsilon\in(0,1)$ be fixed. Then  $u^{\varepsilon}_{\ell_{0}}( x_{0},\iota_{0})\leq {V}^{\varepsilon}_{\ell_{0}}( x_{0},\iota_{0})$ for each $( x_{0},\ell_{0},\iota_{0})\in\overline{\set}\times\mathbb{M}\times\mathbb{I}$.
\end{lema}

\begin{proof}
	Take  $u^{\varepsilon,\delta_{\hat{n}}}=(u^{\varepsilon,\delta_{\hat{n}}}_{1},\dots,u^{\varepsilon,\delta_{\hat{n}}}_{m})$  satisfying \eqref{conv1} which is the unique solutions to the NPDS \eqref{NPD.1}, when $\delta=\delta_{\hat{n}}$. By Proposition \ref{princ1.0}, it is known that $u^{\varepsilon,\delta_{\hat{n}}}_{\ell}(\cdot,\iota)\in\hol^{4,\alpha}(\overline{\set})$ for $(\ell,\iota)\in\mathbb{M}\times \mathbb{I}$. Then, considering $\{\hat{\tau}^{q}_{i}\}_{i\geq0}$ as in Lemma \ref{Ito1}, we get that \eqref{m4} is true when $\hat{f}=u^{\varepsilon,\delta_{\hat{n}}}$.  Notice that $\zeta^{\comp}=\dot{\zeta}$ and $\Delta\zeta=0$, due to $\xi\in\mathcal{U}^{\varepsilon}$.  Then, $\mathcal{J}[X_{\tms},I_{\tms}, u^{\varepsilon,\delta_{\hat{n}}}_{\ell_{i}}]=0$ for $\tms\in(\hat{\tau}^{q}_{i},\hat{\tau}^{q}_{i+1}]$.  	On the other hand, by \eqref{NPD.1} and since $\psi_{\cdot}\geq0$, it is known that $c(x,\iota)u^{\varepsilon,\delta_{\hat{n}}}_{\ell}(x,\iota)-\dif_{\ell,\iota}u^{\varepsilon,\delta_{\hat{n}}}_{\ell}(x,\iota)\leq h(x,\iota)-\psi_{\varepsilon}(|\deri^{1}u^{\varepsilon,\delta_{\hat{n}}}_{\ell}(x,\iota)|^{2}- g(x,\iota)^{2})$ for $x\in\set$  and $(\ell,\iota)\in\mathbb{M}\times \mathbb{I}$, and $\langle\gamma,y\rangle\leq\psi_{\varepsilon}(|\gamma|^{2}-g(x,\iota)^{2})+l^{\varepsilon}(y,x,\iota)$ for $x,y\in\R^{d}$ and $(\ell,\iota)\in\mathbb{M}\times\mathbb{I}$. Then, 
	\begin{equation*}%\label{n3d}
		\begin{split}
			&c(X_{\tms},I_{\tms})u^{\varepsilon,\delta_{\hat{n}}} _{\ell_{i}}(X_{\tms},I_{\tms})-\dif_{\ell_{i},I_{\tms}}u^{\varepsilon,\delta_{\hat{n}}}_{\ell_{i}}(X_{\tms},I_{\tms})\\
			&\quad\quad\quad\quad\quad\quad\leq h(X_{\tms},I_{\tms})-\psi_{\varepsilon}(|\deri^{1}u^{\varepsilon,\delta_{\hat{n}}}_{\ell_{i}}(X_{\tms},I_{\tms})|^{2}- g(X_{\tms},I_{\tms})^{2}),\\
			&\langle\deri^{1} u^{\varepsilon,\delta_{\hat{n}}} _{\ell_{i}}(X_{\tms},I_{\tms}),\mathbb{n}_{\tms}\dot{\zeta}_{\tms}\rangle-\psi_{\varepsilon}(|\deri^{1}u^{\varepsilon,\delta_{\hat{n}}}_{\ell_{i}}(X_{\tms},I_{\tms})|^{2}- g(X_{\tms},I_{\tms})^{2})\leq l^{\varepsilon}(\mathbb{n}_{\tms}\dot{\zeta}_{\tms},X_{\tms},I_{\tms}),
		\end{split}
	\end{equation*}
for $\tms\in(\hat{\tau}^{q}_{i},\hat{\tau}^{q}_{i+1}]$. Hence, it implies that 
\begin{multline}\label{s.2_0_3d}
	\E_{ x_{0},\ell_{0},\iota_{0}}[\expo^{-r({\hat{\tau}}^{q}_{i})}u^{\varepsilon,\delta_{\hat{n}}} _{\ell_{i}}(X_{{\hat{\tau}}^{q}_{i}},I_{{\hat{\tau}}^{q}_{i}})\uno_{\{\tau_{i}< \tau\}}]\leq\E_{ x_{0},\ell_{0},\iota_{0}}\bigg[\bigg\{\expo^{-r({\hat{\tau}}^{q}_{i+1})}u^{\varepsilon,\delta_{\hat{n}}} _{\ell_{i}}(X_{{\hat{\tau}}^{q}_{i+1}},I_{{\hat{\tau}}^{q}_{i+1}})\\
	+\int_{\hat{\tau}^{q}_{i}+}^{{\hat{\tau}}^{q}_{i+1}}\expo^{-r(\tms)}[h(X_{\tms},I_{\tms})+l^{\varepsilon}(\mathbb{n}_{\tms}\dot{\zeta}_{\tms},X_{\tms},I_{\tms})]\der\tms\bigg\}\uno_{\{\tau_{i}<\tilde{\tau}_{1}\}}\bigg].
\end{multline}
Noticing that  $\max_{(x,\iota)\in\set\times\mathbb{I}}|u^{\varepsilon,\delta_{\hat{n}}}_{\ell}(x,\iota)-u^{\varepsilon}_{\ell}(x,\iota)|\underset{\delta_{\hat{n}}\rightarrow0}{\longrightarrow}0$ for $\ell\in\mathbb{M}$,  $\hat{\tau}^{q}_{i}\uparrow\tilde{\tau}_{i}$ as $q\rightarrow\infty$, $\Pro_{ x_{0},\ell_{0},\iota_{0}}$-a.s., letting $q\rightarrow\infty$ and $\delta_{\hat{n}}\rightarrow0$ in \eqref{s.2_0_3d},  and using Dominated Convergence Theorem, it follows that
\begin{align*}%\label{e3}
	\E_{ x_{0},\ell_{0},\iota_{0}}[\expo^{-r({{\tau}}_{i})}u^{\varepsilon} _{\ell_{i}}(X_{{\tau}_{i}},I_{{\tau}_{i}})\uno_{\{\tau_{i}< \tau\}}]&\leq\E_{ x_{0},\ell_{0},\iota_{0}}\bigg[\bigg\{\expo^{-r(\tilde{\tau}_{i+1})}u^{\varepsilon} _{\ell_{i}}(X_{\tilde{\tau}_{i+1}},I_{{\tilde{\tau}}_{i+1}})\notag\\
	&+\int_{\tau_{i}+}^{{\tilde{\tau}}_{i+1}}\expo^{-r(\tms)}[h(X_{\tms},I_{\tms})+l^{\varepsilon}(\mathbb{n}_{\tms}\dot{\zeta}_{\tms},X_{\tms},I_{\tms})]\der\tms\bigg\}\uno_{\{\tau_{i}<\tau\}}\bigg].
\end{align*}
Since $u^{\varepsilon}=(u^{\varepsilon}_{1},\dots,u^{\varepsilon}_{m})$ is the unique solution to \eqref{pc1}, observe that $u^{\varepsilon} _{\ell}(x,\iota)-[u^{\varepsilon} _{\ell'}(x,\iota)+\vartheta_{\ell,\ell'}]\leq u^{\varepsilon} _{\ell}(x,\iota)-\mathcal{M}_{\ell}u^{\varepsilon}(x,\iota)\leq 0$ for $x\in\set$  and $(\ell,\iota)\in\mathbb{M}\times \mathbb{I}$. Then
\begin{align}\label{e4.0}
	u^{\varepsilon} _{\ell_{i}}(X_{\tilde{\tau}_{i+1}},I_{\tilde{\tau}_{i+1}})&=f(X_{{\tau}},I_{{\tau}})\uno_{\{\tau\leq \tau_{i+1}\}}+[u^{\varepsilon} _{\ell_{i+1}}(X_{{\tau}_{i+1}},I_{{\tau}_{i+1}})+\vartheta_{\ell_{i},\ell_{i+1}}]\uno_{\{\tau> \tau_{i+1}\}}\notag\\
	&\quad+[u^{\varepsilon} _{\ell_{i}}(X_{{\tau}_{i+1}},I_{{\tau}_{i+1}})-[u^{\varepsilon} _{\ell_{i+1}}(X_{{\tau}_{i+1}},I_{{\tau}_{i+1}})+\vartheta_{\ell_{i},\ell_{i+1}}]]\uno_{\{\tau> \tau_{i+1}\}}\notag\\
	&\leq f(X_{{\tau}},I_{{\tau}})\uno_{\{\tau\leq \tau_{i+1}\}}+[u^{\varepsilon} _{\ell_{i+1}}(X_{{\tau}_{i+1}},I_{{\tau}_{i+1}})+\vartheta_{\ell_{i},\ell_{i+1}}]\uno_{\{\tau> \tau_{i+1}\}}. 
\end{align}
Thus, 
\begin{align}\label{e5}
	\E_{ x_{0},\ell_{0},\iota_{0}}[\expo^{-r({{\tau}}_{i})}u^{\varepsilon} _{\ell_{i}}(X_{{\tau}_{i}},I_{{\tau}_{i}})\uno_{\{\tau_{i}< \tau\}}]&\leq\E_{ x_{0},\ell_{0},\iota_{0}}\bigg[\expo^{-r({\tau})}f(X_{{\tau}},I_{{\tau}})\uno_{\{\tau_{i}<\tau\leq \tau_{i+1}\}}\notag\\
	&\quad+\expo^{-r({\tau}_{i+1})}[u^{\varepsilon} _{\ell_{i+1}}(X_{{\tau}_{i+1}},I_{{\tau}_{i+1}})+\vartheta_{\ell_{i},\ell_{i+1}}]\uno_{\{\tau> \tau_{i+1}\}}\notag\\
	&\quad+\uno_{\{\tau_{i}<\tau\}}\int_{\tau_{i}+}^{{\tilde{\tau}}_{i+1}}\expo^{-r(\tms)}[h(X_{\tms},I_{\tms})+l^{\varepsilon}(\mathbb{n}_{\tms}\dot{\zeta}_{\tms},X_{\tms},I_{\tms})]\der\tms\bigg].
\end{align}
On the other hand, since the control  $\xi$ acts continuously on $X$, we know that $X_{0}=X_{0-}=x_{0}$. From here, using \eqref{e4.0} when $i=0$, and considering recurrently \eqref{e5}, we conclude that    
	\begin{align}\label{n01}
	 u^{\varepsilon} _{\ell_{0}}(x_{0},\iota_{0})&=\E_{ x_{0},\ell_{0},\iota_{0}}[u^{\varepsilon} _{\ell_{0}}(X_{{\hat{\tau}}^{q}_{0}},I_{{\hat{\tau}}^{q}_{0}})\uno_{\{\tau_{0}=\tilde{\tau}_{1}\}}]+\E_{ x_{0},\ell_{0},\iota_{0}}[u^{\varepsilon} _{\ell_{0}}(X_{{\hat{\tau}}^{q}_{0}},I_{{\hat{\tau}}^{q}_{0}})\uno_{\{\tau_{0}<\tilde{\tau}_{1}\}}]\notag\\
	 &\leq \E_{ x_{0},\ell_{0},\iota_{0}}\bigg[f(x_{0},\iota_{0})\uno_{\{\tau_{0}=\tau\}}+\expo^{-r({\tau})}f(X_{{\tau}},I_{{\tau}})\uno_{\{\tau_{0}<\tau\leq \tau_{1}\}}+\expo^{-r({\tau}_{1})}\vartheta_{\ell_{0},\ell_{1}}\uno_{\{\tau> \tau_{1}\geq\tau_{0}\}}\notag\\
	 &\quad+\uno_{\{\tau_{0}<\tau\}}\int_{0}^{{\tilde{\tau}}_{1}}\expo^{-r(\tms)}[h(X_{\tms},I_{\tms})+l^{\varepsilon}(\mathbb{n}_{\tms}\dot{\zeta}_{\tms},X_{\tms},I_{\tms})]\der\tms\bigg]\notag\\
	 &\quad+\E_{x_{0},\ell_{0},\iota_{0}}[\expo^{-r({\tau}_{1})}u^{\varepsilon} _{\ell_{1}}(X_{{\tau}_{1}},I_{{\tau}_{1}})\uno_{\{\tau> \tau_{1}\geq\tau_{0}\}}]\notag\\
	 &\quad\vdots\quad\quad\quad\quad\quad\quad\vdots\quad\quad\quad\quad\quad\quad\vdots\notag\\
	 &\leq \E_{x_{0},\ell_{0},\iota_{0}}\bigg[\expo^{-r({\tau})}f(X_{{\tau}},I_{{\tau}})\uno_{\{ \tau<\infty\}}+\sum_{i\geq0}\expo^{-r(\tau_{i+1})}\vartheta_{\ell_{i},\ell_{i+1}}\uno_{\{\tau_{i+1}<\tau\}}\notag\\
	 &\quad+\int_{0}^{\tau}\expo^{-r(\tms)}[h(X_{\tms},I_{\tms})+l^{\varepsilon}(\mathbb{n}_{\tms}\dot{\zeta}_{\tms},X_{\tms},I_{\tms})]\der\tms\bigg]=\mathcal{V}_{\zeta,\varsigma}(x_{0},
	 \ell_{0},\iota_{0}).
	 \end{align}
 Therefore, it yields $u^{\varepsilon} _{\ell_{0}}(x_{0},\iota_{0})\leq V^{\varepsilon}_{\ell_{0}}(x_{0},\iota_{0})$    
\end{proof}

\subsubsection{ {$\varepsilon$-PACS} optimal control problem}
Before presenting the second part of the verification lemma, let us first  construct the control $(\xi^{\varepsilon,*},\varsigma^{\varepsilon,*})$ which turns out to  be the optimal strategy for the  {$\varepsilon$-PACS}  control problem.  Let us first introduce the switching regions. 

For any $\ell\in\mathbb{M}$, let $\mathcal{S}^{\varepsilon}_{\ell}$ be the set defined by 
$$\mathcal{S}^{\varepsilon}_{\ell}=\{(x,\iota)\in\set\times\mathbb{I}:u^{\varepsilon}_{\ell}(x,\iota)-\mathcal{M}_{\ell}u^{\varepsilon}(x,\iota)=0\}.$$ 
The complement $\mathcal{C}^{\varepsilon}_{\ell}$ of $\mathcal{S}^{\varepsilon}_{\ell}$ in $\set\times\mathbb{I}$,  where is optimal to stay in  {the} regime $\ell$, is the so-called continuation region
$$\mathcal{C}^{\varepsilon}_{\ell}=\{(x,\iota)\in\set\times\mathbb{I}:u^{\varepsilon}_{\ell}(x,\iota)-\mathcal{M}_{\ell}u^{\varepsilon}(x,\iota)<0\}.$$ 
The set $\mathcal{S}^{\varepsilon}_{\ell}$ satisfies the following property.
\begin{lema}\label{optim1}
	Let $\ell$ be in $\mathbb{M}$. Then, $\mathcal{S}^{\varepsilon}_{\ell}=\widetilde{\mathcal{S}}^{\varepsilon}_{\ell}\eqdef\bigcup_{\ell'\in\mathbb{M}\setminus\{\ell \}}\mathcal{S}^{\varepsilon}_{\ell,\ell'}$ where
	$$\mathcal{S}^{\varepsilon}_{\ell,\ell'}\eqdef\{(x,\iota)\in\mathcal{C}^{\varepsilon}_{\ell'}:u^{\varepsilon}_{\ell}(x,\iota)=u^{\varepsilon}_{\ell'}(x,\iota)+\vartheta_{\ell,\ell'} \}.$$
\end{lema}
\begin{proof}
	We obtain trivially that  $\widetilde{\mathcal{S}}^{\varepsilon}_{\ell}\subset\mathcal{S}^{\varepsilon}_{\ell}$ due to $u^{\varepsilon}_{\ell}(x,\iota)-u^{\varepsilon}_{\ell'}(x,\iota)-\vartheta_{\ell,\ell'}\leq u^{\varepsilon}_{\ell}(x,\iota)-\mathcal{M}_{\ell}u^{\varepsilon}(x,\iota)\leq0$ for $(x,\iota)\in\set\times\mathbb{I}$ and $\ell'\in\mathbb{M}\setminus\{\ell\}$. If $(x,\iota)\in\mathcal{S}^{\varepsilon}_{\ell}$, there is an $\ell_{1}\neq\ell$ where $u^{\varepsilon}_{\ell}(x,\iota)=u^{\varepsilon}_{\ell_{1}}(x,\iota)+\vartheta_{\ell,\ell_{1}}$.Notice that $(x,\iota)$ must belong either $\mathcal{C}^{\varepsilon}_{\ell_{1}}$ or $\mathcal{S}^{\varepsilon}_{\ell_{1}}$. If $(x,\iota)\in\mathcal{C}^{\varepsilon}_{\ell_{1}}$, it yields that  $(x,\iota)\in\mathcal{S}^{\varepsilon}_{\ell,\ell_{1}}\subset\widetilde{\mathcal{S}}^{\varepsilon}_{\ell} $. Otherwise, there is an $\ell_{2}\neq\ell_{1}$ such that $u^{\varepsilon}_{\ell_{1}}(x,\iota)=u^{\varepsilon}_{\ell_{2}}(x,\iota)+\vartheta_{\ell_{1},\ell_{2}}$. It implies $u^{\varepsilon}_{\ell}(x,\iota)=u^{\varepsilon}_{\ell_{2}}(x,\iota)+\vartheta_{\ell,\ell_{1}}+\vartheta_{\ell_{1},\ell_{2}}\geq u^{\varepsilon}_{\ell_{2}}(x,\iota)+\vartheta_{\ell,\ell_{2}}$, since \eqref{eq1} holds. Then, $u^{\varepsilon}_{\ell}(x,\iota)= u^{\varepsilon}_{\ell_{2}}(x,\iota)+\vartheta_{\ell,\ell_{2}}$. Again $(x,\iota)$ must belong either $\mathcal{C}^{\varepsilon}_{\ell_{2}}$ or $\mathcal{S}^{\varepsilon}_{\ell_{2}}$. If $(x,\iota)\in\mathcal{C}^{\varepsilon}_{\ell_{2}}$, it yields that  $(x,\iota)\in\mathcal{S}^{\varepsilon}_{\ell,\ell_{2}}\subset\widetilde{\mathcal{S}}^{\varepsilon}_{\ell} $. Otherwise, arguing the same way than before  and since the number of regimes is finite, it must occur that there is some $\ell_{i}\neq\ell$ such that $(x,\iota)\in\mathcal{C}^{\varepsilon}_{\ell_{i}}$ and $u^{\varepsilon}_{\ell}(x,\iota)= u^{\varepsilon}_{\ell_{i}}(x,\iota)+\vartheta_{\ell,\ell_{i}}$. Therefore $(x,\iota)\in\mathcal{S}^{\varepsilon}_{\ell,\ell_{i}}\subset\widetilde{\mathcal{S}}^{\varepsilon}_{\ell} $.
\end{proof}

Now we construct the optimal control $(\xi^{\varepsilon,*},\varsigma^{\varepsilon,*})$  to the problem \eqref{vw1}.   Let $( x_{0},\ell_{0},\iota_{0})\in\overline{\set}\times\mathbb{M}\times\mathbb{I}$. The dynamics of the process $(X^{\varepsilon,*},I^{*})\eqdef\{(X^{\varepsilon,*}_{\tmt},I^{*}_{\tmt}):\tmt\geq0\}$ and $(\xi^{\varepsilon,*},\varsigma^{\varepsilon,*})$  {is} given recursively in the following way:

\begin{enumerate}
	\item[(i)] Define $\tau^{*}_{0}=0$ and $\ell^{*}_{0}=\ell_{0}$. If $( x_{0},\iota_{0})\notin\mathcal{C}^{\varepsilon}_{\ell_{0}}$, take $\tau^{*}_{1}\eqdef0$ and pass to item (ii) due to Lemma \ref{optim1}. Otherwise,  the process  $(X^{\varepsilon,*},I^{*})$ evolves as
	\begin{equation}\label{opt1}
	\begin{split}
	{X}^{\varepsilon,*}_{\tmt\wedge\tilde{\tau}^{*}_{1}}&=\tilde{x}-\int_{0}^{\tmt\wedge\tilde{\tau}^{*}_{1}}[ b ({X}^{\varepsilon,*}_{\tms},I^{*}_{\tms})+\mathbb{n}_{\tms}^{\varepsilon,*}\dot{\zeta}^{\varepsilon,*}_{\tms}]\der \tms+\int_{0}^{\tmt\wedge\tilde{\tau}^{*}_{1}}\sigma({X}^{\varepsilon,*}_{\tms},I^{*}_{\tms})\der  W_{\tms},\\
	I^{*}_{\tmt\wedge\tilde{\tau}^{*}_{1}}&=I^{(\ell^{*}_{0})}_{\tmt\wedge\tilde{\tau}^{*}_{1}},
	\end{split}
	\qquad\text{for}\ t>0,
	\end{equation}
	with $X^{\varepsilon,*}_{0}=x_{0}$, $I^{*}_{0}=\iota_{0}$, $\tau^{*}\eqdef\inf\{\tmt>0:(X^{\varepsilon,*}_{\tmt},I^{*}_{\tmt})\notin\set\}$,
	\begin{equation}\label{opt3}
	\tilde{\tau}^{*}_{1}\eqdef\tau^{*}_{1}\wedge\tau^{*}\quad\text{and}\quad\tau^{*}_{1}\eqdef\inf\big\{\tmt\geq0:(X^{\varepsilon,*}_{\tmt},I^{*}_{\tmt})\in\mathcal{S}^{\varepsilon}_{\ell^{*}_{0}}\big\}.
	\end{equation}
	The control process $\xi^{\varepsilon,*}=(\mathbb{n}^{\varepsilon,*},\zeta^{\varepsilon,*})$ is defined by
	\begin{equation}\label{opt4}
	\mathbb{n}^{\varepsilon,*}_{\tmt}=
	\begin{cases}
	\frac{\deri^{1}u^{\varepsilon}_{\ell^{*}_{0}}(X^{\varepsilon,*}_{\tmt},I^{*}_{\tmt})}{|\deri^{1}u^{\varepsilon}_{\ell^{*}_{0}}( X^{\varepsilon,*}_{\tmt},I^{*}_{\tmt})|}, &\text{if}\ |\deri^{1}u^{\varepsilon}_{\ell^{*}_{0}}( X ^{\varepsilon,*}_{\tmt},I^{*}_{\tmt})|\neq0\ \text{and}\ \tmt\in[0,\tilde{\tau}^{*}_{1}),\\
	\gamma_{0},& \text{if}\ |\deri^{1}u^{\varepsilon}_{\ell^{*}_{0}}(X^{\varepsilon,*}_{\tmt},I^{*}_{\tmt})|=0\  \text{and}\ \tmt\in[0,\tilde{\tau}^{*}_{1})
	\end{cases}
	\end{equation}
	where $\gamma_{0}\in\R^{d}$ is a unit vector fixed, and $\zeta^{\varepsilon,*}_{\tmt}=\int_{0}^{\tmt}\dot{\zeta}^{\varepsilon,*}_{\tms}\der \tms$, with $\tmt\in[0,\tilde{\tau}^{*}_{1})$ and 
	\begin{align}\label{opt5}
	\dot{\zeta}^{\varepsilon,*}_{\tms}= 2\psi'_{\varepsilon}(|\deri^{1} u^{\varepsilon}_{\ell^{*}_{0}}(X^{\varepsilon,*}_{\tms},I^{*}_{\tms})|^{2}- g_{\ell^{*}_{0}}(X^{\varepsilon,*}_{\tms},I^{*}_{\tms})^{2})|\deri^{1} u^{\varepsilon}_{\ell^{*}_{0}}(X^{\varepsilon,*}_{\tms},I^{*}_{\tms})|.
	\end{align} 
	\item[(ii)] Recursively, letting  $i\geq1$ and defining
	\begin{equation}\label{opt6}
	\begin{split}
	&\ell^{*}_{i}\in\displaystyle\argmin_{\ell'\in\mathbb{I}\setminus\{\ell^{*}_{i-1}\}}\big\{u^{\varepsilon}_{\ell'}(X^{\varepsilon,*}_{\tau^{*}_{i}},I^{*}_{\tau^{*}_{i}})+\vartheta_{\ell^{*}_{i-1},\ell'}\big\},\\
	&\tilde{\tau}^{*}_{i+1}=\tau^{*}_{i+1}\wedge\tau^{*},\quad\tau^{*}_{i+1}=\inf\big\{\tmt>\tau^{*}_{i}:(X^{\varepsilon,*}_{\tmt},I^{*}_{\tmt})\in\mathcal{S}^{\varepsilon}_{\ell^{*}_{i}}\big\},
	\end{split}
	\end{equation}  
	if $\tau^{*}_{i}<\tau^{*}$, the process $X^{\varepsilon,*}$  evolves as
	\begin{equation}\label{opt7}
	\begin{split}
	{X}^{\varepsilon,*}_{\tmt\wedge\tilde{\tau}^{*}_{i+1}}&=X^{\varepsilon,*}_{\tau_{i}^{*}}
	-\int_{\tau^{*}_{i}}^{\tmt\wedge\tilde{\tau}^{*}_{i+1}}[ b ({X}^{\varepsilon,*}_{\tms},I^{*}_{\tms})+\mathbb{n}_{\tms}^{\varepsilon,*}\dot{\zeta}^{\varepsilon,*}_{\tms}]\der \tms+\int_{\tau^{*}_{i}}^{\tmt\wedge\tilde{\tau}^{*}_{i+1}}\sigma({X}^{\varepsilon,*}_{\tms},I^{*}_{\tms})\der  W_{\tms},\\ 
	I^{*}_{\tmt\wedge\tilde{\tau}^{*}_{i+1}}&=I^{(\ell^{*}_{i})}_{\tmt\wedge\tilde{\tau}^{*}_{i+1}},
	\end{split}
	\qquad\text{for}\ \tmt\geq\tau^{*}_{i},
	\end{equation}
	where  
	\begin{equation}\label{opt9}
	\mathbb{n}^{\varepsilon,*}_{\tmt}=
	\begin{cases}
	\frac{\deri^{1}u^{\varepsilon}_{\ell^{*}_{i}}(X^{\varepsilon,*}_{\tmt},I^{*}_{\tmt})}{|\deri^{1}u^{\varepsilon}_{\ell^{*}_{i}}( X^{\varepsilon,*}_{\tmt},I^{*}_{\tmt})|}, &\text{if}\ |\deri^{1}u^{\varepsilon}_{\ell^{*}_{i}}( X ^{\varepsilon,*}_{\tmt},I^{*}_{\tmt})|\neq0\ \text{and}\ \tmt\in[\tau^{*}_{i},\tilde{\tau}^{*}_{i+1}),\\
	\gamma_{0},& \text{if}\ |\deri^{1}u^{\varepsilon}_{\ell^{*}_{i}}(X^{\varepsilon,*}_{\tmt},I^{*}_{\tmt})|=0\  \text{and}\ \tmt\in[\tau^{*}_{i},\tilde{\tau}^{*}_{i+1}), 
	\end{cases}
	\end{equation}
	with $\gamma_{0}\in\R^{d}$ is a unit vector fixed, and $\zeta^{\varepsilon,*}_{\tmt}=\int_{\tau^{\varepsilon,*}_{i}}^{\tmt}\dot{\zeta}^{\varepsilon,*}_{\tms}\der \tms$, with $\tmt\in[\tau^{*}_{i},\tilde{\tau}^{*}_{i+1})$ and 
	\begin{align}\label{opt10}
	\dot{\zeta}^{\varepsilon,*}_{\tms}= 2\psi'_{\varepsilon}(|\deri^{1} u^{\varepsilon}_{\ell^{*}_{i}}(X^{\varepsilon,*}_{\tms},I^{*}_{\tms})|^{2}- g_{\ell^{*}_{i}}(X^{\varepsilon,*}_{\tms},I^{*}_{\tms})^{2})|\deri^{1} u^{\varepsilon}_{\ell^{*}_{i}}(X^{\varepsilon,*}_{\tms},I^{*}_{\tms})|.
	\end{align} 
\end{enumerate}

\begin{rem}\label{ins1}
	Suppose that $\tau^{*}_{i}<\tau^{*}$ for some $i>0$. We notice that for $\tmt\in[\tau^{*}_{i},\tau^{*}_{i+1})$,  $\mathbb{n}^{\varepsilon,*}_{\tmt}\dot{\zeta}^{\varepsilon,*}_{\tmt}=2\psi'_{\varepsilon}(|\deri^{1} u^{\varepsilon}_{\ell^{*}_{i}}(X^{\varepsilon,*}_{\tmt},I^{*}_{\tmt})|^{2}- g   (X^{\varepsilon,*}_{\tmt},I^{*}_{\tmt})^{2})\deri^{1}u^{\varepsilon}_{\ell^{*}_{i}}(X^{\varepsilon,*}_{\tmt},I^{*}_{\tmt})$, $\Delta\zeta^{\varepsilon,*}_{\tmt}=0$, $|\mathbb{n}^{\varepsilon,*}_{\tmt}|=1$ and, by \eqref{p12.1} and \eqref{ineqe1}, it yields that  $\dot{\zeta}^{\varepsilon,*}_{\tmt}\leq\frac{2C_{4}}{\varepsilon}$.  Also {we} see that $(X^{\varepsilon,*}_{\tmt},I^{*}_{\tmt})\in\mathcal{C}^{\varepsilon}_{\ell^{*}_{i}}$ if  $\tmt\in[\tau^{*}_{i},\tau^{*}_{i+1})$ due to Lemma \ref{optim1}.
\end{rem}
\begin{rem}
	On the event $\{\tau^{*}=\infty\}$, $\tilde{\tau}^{*}_{i}=\tau_{i}^{*}$ for $i\geq0$.  From here and by \eqref{opt4}--\eqref{opt5} and \eqref{opt9}--\eqref{opt10},  it yields that  the control process $(\xi^{\varepsilon,*},\varsigma^{\varepsilon,*})$ belongs to $\mathcal{U}^{\varepsilon}\times\mathcal{S}$.  On the event $\{\tau^{*}<\infty \}$,   let $\hat{\iota}$ be defined 
	as $\hat{\iota}=\max\{i\in\mathbb{N}:\tau^{*}_{i}\leq\tau^{*}\}$. Then, taking $\tau^{*}_{i}\eqdef\tau^{*}+i$ and $\ell^{*}_{i}=\hat{\ell}$ for $i>\hat{\iota}$, where $\hat{\ell}\in\mathbb{I}$ is fixed,  it follows that $\varsigma^{\varepsilon,*}=(\tau^{*}_{i},\ell^{*}_{i})_{i\geq1}\in\mathcal{S}$. We take $\dot{\zeta}^{\varepsilon,*}_{\tmt}\equiv 0$ and  $\mathbb{n}^{\varepsilon,*}_{\tmt}\eqdef\gamma_{0}$, for $\tmt>\tau^{*}$. In this way, we have that $(\mathbb{n}^{\varepsilon,*},\zeta^{\varepsilon,*})\in\mathcal{U}^{\varepsilon}$.
\end{rem}

\begin{rem}
	Taking $J^{*}_{\tmt}=\ell_{0}\uno_{[0,\tau^{*}_{1})}(\tmt)+\ell^{*}_{1}\uno_{\{\tau^{*}_{1}=\tau^{*}_{0}\}}+\sum_{i\geq1}\ell^{*}_{i}\uno_{[\tau^{*}_{i},\tau^{*}_{i+1})}(\tmt)$, we see that it is a c\`adl\`ag process.
\end{rem}

\begin{lema}[Verification Lemma for  {$\varepsilon$-PACS} control problem. Second part]\label{convexu1.0}
	Let  $\varepsilon\in(0,1)$ be fixed and let  {$(X^{\varepsilon,*},I^{*})$ be the process that is governed by} \eqref{opt1}--\eqref{opt10}. Then, $ {V^{\varepsilon}_{\ell_{0}}(\cdot,\iota_{0})\in\hol^{1}(\overline{\set})\cap\sob^{2,\infty}_{\loc}(
		\set)}$ and  $u^{\varepsilon}_{\ell_{0}}(x_{0},\iota_{0})=\mathcal{V}_{\xi^{\varepsilon,*},\varsigma^{\varepsilon,*}}(x_{0},\ell_{0},\iota_{0})=V^{\varepsilon}_{\ell_{0}}(x_{0},\iota_{0})$ for each $(x_{0},\iota_{0})\in\overline{\set}\times\mathbb{I}$ and $\ell_{0}\in\mathbb{M}$. 
\end{lema}
\begin{proof}
 {Notice that if $u_{\ell_{0}}(x_{0},\iota_{0})=V^{\varepsilon}_{\ell_{0}}(x_{0},\iota_{0})$	 for each $(x_{0},\iota_{0})\in\overline{\set}\times\mathbb{I}$ and $\ell_{0}\in\mathbb{M}$, it follows immediately  $V^{\varepsilon}_{\ell_{0}}(\cdot,\iota_{0})\in\hol^{1}(\overline{\set})\cap\sob^{2,\infty}_{\loc}(
	\set)$. So, let us show the agreement between $u_{\ell_{0}}(x_{0},\iota_{0})$ and $V^{\varepsilon}_{\ell_{0}}(x_{0},\iota_{0})$.}		Taking   $\hat{\tau}_{i}^{*,q}$  as $\hat{\tau}_{i}^{q}$ in Lemma  \ref{Ito1}, with $\tilde{\tau}_{i}=\tilde{\tau}^{*}_{i}$, and considering  $u^{\varepsilon,\delta_{\hat{n}}}$ which is the unique solution of \eqref{NPD.1} when $\delta=\delta_{\hat{n}}$, by Lemma   \ref{Ito1}, we get that 
	\begin{align*}%\label{i1}
		&\E_{ x_{0},\ell_{0},\iota_{0}}[\expo^{-r(\hat{\tau}^{*,q}_{i})}u^{\varepsilon,\delta_{\hat{n}}} _{\ell^{*}_{i}}(X^{\varepsilon,*}_{{\hat{\tau}}^{*,q}_{i}},I^{*}_{{\hat{\tau}}^{*,q}_{i}})\uno_{\{\tau^{*}_{i}<\tau^{*}\}}]\notag\\
		&=\E_{ x_{0},\ell_{0},\iota_{0}}\bigg[\bigg\{\expo^{-r({\hat{\tau}}^{*,q}_{i+1})}u^{\varepsilon,\delta_{\hat{n}}} _{\ell^{*}_{i}}(X^{\varepsilon,*}_{{\hat{\tau}}^{*,q}_{i+1}},I^{*}_{{\hat{\tau}}^{*,q}_{i+1}})\notag\\
		&\quad+\int_{\hat{\tau}^{*,q}_{i}}^{{\hat{\tau}}^{*,q}_{i+1}}\expo^{-r(\tms)}\bigg[h(X^{\varepsilon,*}_{\tms},I^{*}_{\tms})-\sum_{\ell'\in\mathbb{M}\setminus\{\ell^{*}_{i}\}}\psi_{\delta}(u^{\varepsilon,\delta_{\hat{n}}}_{\ell^{*}_{i}}(X^{\varepsilon,*}_{\tms},I^{*}_{\tms})-u^{\varepsilon,\delta_{\hat{n}}}_{\ell'}(X^{\varepsilon,*}_{\tms},I^{*}_{\tms})-\vartheta_{\ell^{*}_{i},\ell'})\notag\\
		&\quad-\psi_{\varepsilon}(|\deri^{1}u^{\varepsilon,\delta_{\hat{n}}} _{\ell^{*}_{i}}(X^{\varepsilon,*}_{\tms},I^{*}_{\tms})|^{2}-g(X^{\varepsilon,*}_{\tms},I^{*}_{\tms})^{2})+\langle\deri^{1} u^{\varepsilon,\delta_{\hat{n}}} _{\ell^{*}_{i}}(X^{\varepsilon,*}_{\tms},I^{*}_{\tms}),\mathbb{n}^{\varepsilon,*}_{\tms}\dot{\zeta}^{\varepsilon,*}_{\tms}\rangle\bigg]\der\tms\bigg\}\uno_{\{\tau^{*}_{i}<\tau^{*}\}}\bigg].
	\end{align*}
Letting $\delta_{\hat{n}}\rightarrow\infty$, by dominated convergence theorem, we get
	\begin{align}\label{i2}
	&\E_{ x_{0},\ell_{0},\iota_{0}}[\expo^{-r(\hat{\tau}^{*,q}_{i})}u^{\varepsilon} _{\ell^{*}_{i}}(X^{\varepsilon,*}_{\hat{\tau}^{*,q}_{i}},I^{*}_{\hat{\tau}^{*,q}_{i}})\uno_{\{\tau^{*}_{i}<\tau^{*}\}}]\notag\\
	&=\E_{ x_{0},\ell_{0},\iota_{0}}\bigg[\bigg\{\expo^{-r({\hat{\tau}}^{*,q}_{i+1})}u^{\varepsilon} _{\ell^{*}_{i}}(X^{\varepsilon,*}_{{\hat{\tau}}^{*,q}_{i+1}},I^{*}_{{\hat{\tau}}^{*,q}_{i+1}})\notag\\
	&\quad+\int_{\hat{\tau}^{*,q}_{i}}^{{\hat{\tau}}^{*,q}_{i+1}}\expo^{-r(\tms)}\bigg[h(X^{\varepsilon,*}_{\tms},I^{*}_{\tms})-\sum_{\ell'\in\mathbb{M}\setminus\{\ell^{*}_{i}\}}\psi_{\delta}(u^{\varepsilon}_{\ell^{*}_{i}}(X^{\varepsilon,*}_{\tms},I^{*}_{\tms})-u^{\varepsilon}_{\ell'}(X^{\varepsilon,*}_{\tms},I^{*}_{\tms})-\vartheta_{\ell^{*}_{i},\ell'})\notag\\
	&\quad-\psi_{\varepsilon}(|\deri^{1}u^{\varepsilon} _{\ell^{*}_{i}}(X^{\varepsilon,*}_{\tms},I^{*}_{\tms})|^{2}-g(X^{\varepsilon,*}_{\tms},I^{*}_{\tms})^{2})+\langle\deri^{1} u^{\varepsilon} _{\ell^{*}_{i}}(X^{\varepsilon,*}_{\tms},I^{*}_{\tms}),\mathbb{n}^{\varepsilon,*}_{\tms}\dot{\zeta}^{\varepsilon,*}_{\tms}\rangle\bigg]\der\tms\bigg\}\uno_{\{\tau^{*}_{i}<\tau^{*}\}}\bigg],
	\end{align}
	because of  $\max_{(x,\iota)\in\mathcal{C}^{\varepsilon}_{\ell}}\{|(u^{\varepsilon,\delta_{\hat{n}}}_{\ell}-u^{\varepsilon}_{\ell})(x,\iota)|,|\deri^{1}(u^{\varepsilon,\delta_{\hat{n}}}_{\ell}-u^{\varepsilon}_{\ell})(x,\iota)|\}\underset{\delta_{\hat{n}}\rightarrow0}{\longrightarrow}0$  for $\ell\in\mathbb{M}$, and continuity  of $\psi_{\cdot}$. Then, considering that $u^{\varepsilon}_{\ell}-u^{\varepsilon}_{\ell'}-\vartheta_{\ell,\ell'}\leq u^{\varepsilon}_{\ell}-\mathcal{M}_{\ell}u^{\varepsilon}<0$ on $\mathcal{C}^{\varepsilon}_{\ell}$, with $\ell'\in\mathbb{M}\setminus\{\ell\}$, and   $l^{\varepsilon}(2\psi'_{\varepsilon}(|\gamma|^{2}- g(x,\iota)^{2})\gamma,x)=2\psi'_{\varepsilon}(|\gamma|^{2}- g(x,\iota)^{2})|\gamma|^{2}-\psi_{\varepsilon}(|\gamma|^{2}- g(x,\iota)^{2})$, and letting $q\rightarrow0$  in \eqref{i2}, it can be checked 
	\begin{align}\label{i3}
	\E_{ x_{0},\ell_{0},\iota_{0}}&[\expo^{-r(\tau^{*}_{i})}u^{\varepsilon} _{\ell^{*}_{i}}(X^{\varepsilon,*}_{\tau^{*}_{i}},I^{*}_{\tau^{*}_{i}})\uno_{\{\tau^{*}_{i}<\tau^{*}\}}]\notag\\
	&=\E_{ x_{0},\ell_{0},\iota_{0}}\bigg[\bigg\{\expo^{-r(\tilde{\tau}^{*}_{i+1})}u^{\varepsilon} _{\ell^{*}_{i}}(X^{\varepsilon,*}_{\tilde{\tau}^{*}_{i+1}},I^{*}_{\tilde{\tau}^{*}_{i+1}})\notag\\
	&\quad+\int_{\tau^{*}_{i}}^{\tilde{\tau}^{*}_{i+1}}\expo^{-r(\tms)}\bigg[h(X^{\varepsilon,*}_{\tms},I^{*}_{\tms})+l^{\varepsilon}(\mathbb{n}^{\varepsilon,*}_{\tms}\dot{\zeta}^{\varepsilon,*},X^{\varepsilon,*}_{\tms},I^{*}_{\tms})\bigg]\der\tms\bigg\}\uno_{\{\tau^{*}_{i}<\tau^{*}\}}\bigg].
	\end{align}
	By \eqref{i3} and noticing that 
	$$
	u^{\varepsilon} _{\ell^{*}_{i}}(X^{\varepsilon,*}_{\tilde{\tau}^{*}_{i+1}},I^{*}_{\tilde{\tau}^{*}_{i+1}})=f(X^{\varepsilon,*}_{\tau^{*}},I^{*}_{\tau^{*}})\uno_{\{\tau^{*}\leq\tau^{*}_{i+1}\}}+[u^{\varepsilon} _{\ell^{*}_{i+1}}(X^{\varepsilon,*}_{\tilde{\tau}^{*}_{i+1}},I^{*}_{\tilde{\tau}^{*}_{i+1}})+\vartheta_{\ell^{*}_{\ell_{i}},\ell^{*}_{\ell_{i+1}}}]\uno_{\{\tau^{*}>\tau^{*}_{i+1}\}},$$ 
	due to \eqref{p13.0.1.0}, \eqref{opt3} and \eqref{opt6}, the reader can verified easily that 
	\begin{align*}
	u^{\varepsilon}_{\ell_{0}}(x_{0},\iota_{0})=\mathcal{V}_{\xi^{\varepsilon,*},\varsigma^{\varepsilon,*}}( x_{0},\ell_{0},\iota_{0})\geq V^{\varepsilon}_{\ell_{0}}(x_{0},\iota_{0}).
	\end{align*}
	From here and Lemma \ref{lv1}, we conclude   $u^{\varepsilon}_{\ell_{0}}(x_{0},\iota_{0})=\mathcal{V}_{\xi^{\varepsilon,*},\varsigma^{\varepsilon,*}}(x_{0},\ell_{0},\iota_{0})=V^{\varepsilon}_{\ell_{0}}(x_{0},\iota_{0})$ for each $(x_{0},\ell_{0},\iota_{0})\in\overline{\set}\times\mathbb{M}\times\mathbb{I}$.
\end{proof}

\subsection{Proof of Theorem \ref{verf2}}

\begin{proof} 

	Let $\{u^{\varepsilon_{\hat{n}}}\}_{\hat{n}\geq1}$, with $u^{\varepsilon_{\hat{n}}}=(u^{\varepsilon_{\hat{n}}}_{1},\dots,u^{\varepsilon_{\hat{n}}}_{m})$, be the sequence of unique strong solutions to the HJB equation \eqref{pc1}, when $\varepsilon=\varepsilon_{\hat{n}}$,  which  satisfy \eqref{econv1}. From Lemma \ref{convexu1.0}, we know that  
$$u^{\varepsilon_{\hat{n}}}_{\ell_{0}}(x_{0},\iota_{0})=\mathcal{V}_{\xi^{\varepsilon_{\hat{n}},*},\varsigma^{\varepsilon_{\hat{n}},*}}(x_{0},\ell_{0},\iota_{0})= V^{\varepsilon_{\hat{n}}}(x_{0},\ell_{0},\iota_{0})\quad\text{for $(x_{0},\ell_{0},\iota_{0})\in\overline{\set}\times\mathbb{M}\times\mathbb{I}$},$$
with $(\xi^{\varepsilon_{\hat{n}},*},\varsigma^{\varepsilon_{\hat{n}},*})$ as in \eqref{opt3}--\eqref{opt6} and \eqref{opt9}--\eqref{opt10}, when $\varepsilon=\varepsilon_{\hat{n}}$. Notice that $l^{\varepsilon_{\hat{n}}}(\beta\gamma,x,\iota)\geq\langle g(x,\iota)\gamma,\beta\gamma\rangle-\psi_{\varepsilon_{\hat{n}}}(| g(x,\iota)\gamma|^{2}- g(x,\iota)^{2})=\beta g(x,\iota)$, with $\beta\in\R$ and $\gamma\in\R^{d}$ a unit vector. Then, from here and considering  {$(X^{\varepsilon_{\hat{n}},*},J^{*},I^{*})$ governed by} \eqref{opt1}--\eqref{opt10}, it follows that 
\begin{align}\label{NPIDD1.0}
	V_{\ell_{0}}(x_{0},\iota_{0})&\leq V_{\xi^{\varepsilon_{\hat{n}},*},\varsigma^{\varepsilon_{\hat{n}},*}}(x_{0},\ell_{0},\iota_{0})\notag\\
	&=\E_{x_{0},\ell_{0},\iota_{0}}\bigg[\int_{0}^{\tau^{*}}\expo^{-r(\tms)}[h(X^{\varepsilon_{\hat{n}},*}_{\tms},I^{*}_{\tms})\der \tms+{\dot{\zeta}^{\varepsilon_{\hat{n}},*}_{\tms}}g(X^{\varepsilon_{\hat{n}},*}_{\tms},I^{*}_{\tms})]\der\tms\notag\\
	&\quad+\expo^{r(\tau^{*})}f(X^{\varepsilon_{\hat{n}},*}_{\tau^{*}},I^{*}_{\tau^{*}})\uno_{\{\tau^{*}<\infty\}}+\sum_{i\geq0}\expo^{-r(\tau^{*}_{i+1})}\vartheta_{\ell^{*}_{i},\ell^{*}_{i+1}}\uno_{\{\tau^{*}_{i+1}<\tau^{*}\}}\bigg]\notag\\
	&\leq \E_{x_{0},\ell_{0},\iota_{0}}\bigg[\int_{0}^{\tau^{*}}\expo^{-r(\tms)}[h(X^{\varepsilon_{\hat{n}},*}_{\tms},I^{*}_{\tms})+l^{\varepsilon_{\hat{n}}}({\dot{\zeta}^{\varepsilon_{\hat{n}},*}_{\tms}\mathbb{n}^{\varepsilon_{\hat{n}},*}_{\tms}},X^{\varepsilon_{\hat{n}},*}_{\tms},I^{*}_{\tms})]\der \tms\notag\\
	&\quad+\expo^{r(\tau^{*})}f(X^{\varepsilon_{\hat{n}},*}_{\tau^{*}},I^{*}_{\tau^{*}})\uno_{\{\tau^{*}<\infty\}}+\sum_{i\geq0}\expo^{-r(\tau^{*}_{i+1})}\vartheta_{\ell^{*}_{i},\ell^{*}_{i+1}}\uno_{\{\tau^{*}_{i+1}<\tau^{*}\}}\bigg]=u^{\varepsilon_{\hat{n}}}_{\ell_{0}}(x_{0},\iota_{0}).
\end{align}
Letting $\varepsilon_{\hat{n}}\rightarrow0$ in \eqref{NPIDD1.0}, it yields {$u_{\ell_{0}}(x_{0},\iota_{0})\geq V_{\ell_{0}}(x_{0},\iota_{0})$ } for each $(x_{0},\ell_{0},\iota_{0})\in\overline{\set}\times\mathbb{M}\times\mathbb{I}$. 	
	
Let us consider $(X,J,I)$ evolving as in \eqref{es1}  with initial state $(x_{0},\ell_{0},\iota_{0})\in\set\times\mathbb{M}\times\mathbb{I}$, and the control process $(\xi,\varsigma)$ belongs to $\mathcal{U}\times\mathcal{S}$.   Taking $\hat{f}=u^{\varepsilon_{\hat{m}},\delta_{\hat{n}}}$, by Lemma \ref{Ito1}, we get that \eqref{m6} holds. Since $u^{\varepsilon_{\hat{m}},\delta_{\hat{m}}}$ is the unique solution to \eqref{NPD.1} when $\varepsilon=\varepsilon_{\hat{m}}$ and $\delta=\delta_{\hat{n}}$, and $\psi_{\cdot}$ is a positive function,   the reader can verify that 
\begin{align}\label{i1}
	&\E_{ x_{0},\ell_{0},\iota_{0}}[\expo^{-r(\hat{\tau}^{q}_{i})}u^{\varepsilon_{\hat{m}},\delta_{\hat{n}}} _{\ell_{i}}(X_{{\hat{\tau}}^{q}_{i}},I_{{\hat{\tau}}^{q}_{i}})\uno_{\{\tau_{i}<\tau\}}]\notag\\
	&\leq\E_{ x_{0},\ell_{0},\iota_{0}}\bigg[\bigg\{\expo^{-r({\hat{\tau}}^{q}_{i+1})}u^{\varepsilon_{\hat{m}},\delta_{\hat{n}}} _{\ell_{i}}(X_{{\hat{\tau}}^{q}_{i+1}},I_{{\hat{\tau}}^{q}_{i+1}})-\sum_{\hat{\tau}^{q}_{i}<\tms\leq\hat{\tau}^{q}_{i+1}}\expo^{-r(\tms)}\mathcal{J}[X^{\xi,\varsigma}_{\tms},I_{\tms}, u^{\varepsilon_{\hat{m}},\delta_{\hat{n}}} _{\ell_{i}}]\notag\\
	&\quad+\int_{\hat{\tau}^{q}_{i}+}^{{\hat{\tau}}^{q}_{i+1}}\expo^{-r(\tms)}[h(X_{\tms},I_{\tms}) +\langle\deri^{1} u^{\varepsilon_{\hat{m}},\delta_{\hat{n}}} _{\ell_{i}}(X_{\tms},I_{\tms}),\mathbb{n}_{\tms}\zeta^{\comp}_{\tms}\rangle]\der\tms\bigg\}\uno_{\{\tau_{i}<\tau\}}\bigg].
\end{align}
Additionally, considering $\Delta\zeta_{\tms}\neq0$ and $X_{\tms-}-\mathbb{n}_{\tms}\Delta \zeta_{\tms} \in\set$ for $\tms\in(\hat{\tau}^{q}_{i},\hat{\tau}^{q}_{i+1}]$, and using mean value theorem, we have that 
\begin{align}\label{n4}
	-\mathcal{J}[X_{\tms},I_{\tms}, u^{\varepsilon_{\hat{m}},\delta_{\hat{n}}}_{\ell_{i}} ]&\leq|u^{\varepsilon_{\hat{m}},\delta_{\hat{n}}}_{\ell_{i}}(X_{\tms-}-\mathbb{n}_{\tms}\Delta \zeta_{\tms},I_{\tms})-u^{\varepsilon_{\hat{m}},\delta_{\hat{n}}}_{\ell_{i}}(X_{\tms-},I_{\tms})|\notag\\
	&\leq\Delta\zeta_{\tms}\int_{0}^{1}|\deri^{1}u^{\varepsilon_{\hat{m}},\delta_{\hat{n}}}_{\ell_{i}}(X_{\tms-}-\lambda\mathbb{n}_{\tms}\Delta \zeta_{\tms},I_{\tms})|\der\lambda.
\end{align}
Recall that $\max_{(x,\iota)\in\set_{q}\times\mathbb{I}}\{|(u^{\varepsilon_{\hat{m}},\delta_{\hat{n}}}_{\ell}-u_{\ell})(x,\iota)|,|\deri^{1}(u^{\varepsilon_{\hat{m}},\delta_{\hat{n}}}_{\ell}-u_{\ell})(x,\iota)|\}\underset{\varepsilon_{\hat{m}},\delta_{\hat{n}}\rightarrow0}{\longrightarrow}0$ for $\ell\in\mathbb{M}$, due to \eqref{conv1} and \eqref{econv1}.  Then, applying \eqref{n4} in \eqref{i1} and letting $\varepsilon_{\hat{m}},\delta_{\hat{n}}\rightarrow0$, by the dominated convergence theorem, it follows that 
\begin{align}\label{i1.0}
	\E_{ x_{0},\ell_{0},\iota_{0}}[\expo^{-r(\hat{\tau}^{q}_{i})}u _{\ell_{i}}(X_{{\hat{\tau}}^{q}_{i}},I_{{\hat{\tau}}^{q}_{i}})\uno_{\{\tau_{i}<\tau\}}]&\leq\E_{ x_{0},\ell_{0},\iota_{0}}\bigg[\bigg\{\expo^{-r({\hat{\tau}}^{q}_{i+1})}u_{\ell_{i}}(X_{{\hat{\tau}}^{q}_{i+1}},I_{{\hat{\tau}}^{q}_{i+1}})\notag\\
	&\quad+\int_{\hat{\tau}^{q}_{i}+}^{{\hat{\tau}}^{q}_{i+1}}\expo^{-r(\tms)}[h(X_{\tms},I_{\tms})\der\tms +g(X_{\tms-},I_{\tms})\circ\der\zeta_{\tms}]\bigg\}\uno_{\{\tau_{i}<\tau\}}\bigg],
\end{align}
due to $|\deri^{1}u(\cdot,\iota)|-g(\cdot,\iota)\leq0$ on $\set$. Letting $q\rightarrow\infty$ in \eqref{i1.0} and taking into account that \eqref{e4.0} holds for this case, by the dominated convergence theorem, it implies that
\begin{align}\label{i1.1}
	\E_{ x_{0},\ell_{0},\iota_{0}}[\expo^{-r({\tau}_{i})}u _{\ell_{i}}(X_{{{\tau}}_{i}},I_{{{\tau}}_{i}})\uno_{\{\tau_{i}<\tau\}}]&\leq\E_{ x_{0},\ell_{0},\iota_{0}}\bigg[\expo^{-r({\tau})}f(X_{{\tau}},I_{{\tau}})\uno_{\{\tau_{i}<\tau\leq \tau_{i+1}\}}\notag\\
	&\quad+\expo^{-r({\tau}_{i+1})}[u _{\ell_{i+1}}(X_{{\tau}_{i+1}},I_{{\tau}_{i+1}})+\vartheta_{\ell_{i},\ell_{i+1}}]\uno_{\{\tau> \tau_{i+1}\}}\notag\\
	&\quad+\uno_{\{\tau_{i}<\tau\}}\int_{\hat{\tau}_{i}+}^{{\hat{\tau}}_{i+1}}\expo^{-r(\tms)}[h(X_{\tms},I_{\tms}) +g(X_{\tms-},I_{\tms})\circ\der\zeta_{\tms}]\der\tms\bigg].
\end{align}
On the other hand, since the control  $\xi$ can act on $X$ by a jump of $\zeta$ at time zero, we have that $X_{0}=x_{0}-\mathbb{n}_{0}\Delta\zeta_{0}$. From here and considering recurrently \eqref{i1.1}, we conclude that    
\begin{align*}%\label{n01}
	u_{\ell_{0}}(x_{0},\iota_{0})&\leq\E_{ x_{0},\ell_{0},\iota_{0}}\bigg[f(x_{0},\iota_{0})\uno_{\{\tau_{0}=\tau\}}+\Delta\xi_{0}\uno_{\{\tau_{0}<\tau\}}\int_{0}^{1}g(x_{0}-\lambda\mathbb{n}_{0}\Delta\xi_{0},I_{0})\der\lambda\notag\\
	&\quad+\vartheta_{\ell_{0},\ell_{1}}\uno_{\{\tau_{0}=\tau_{1}<\tau\}}+u_{\ell_{1}}(X_{0},I_{0})\uno_{\{\tau_{0}=\tau_{1}<\tau\}}+u_{\ell_{0}}(X_{0},I_{0})\uno_{\{\tau_{0}<\tilde{\tau}_{1}\}}\bigg]\notag\\
	&\leq \E_{ x_{0},\ell_{0},\iota_{0}}\bigg[f(x_{0},\iota_{0})\uno_{\{\tau_{0}=\tau\}}+\expo^{-r({\tau})}f(X_{{\tau}},I_{{\tau}})\uno_{\{\tau_{0}<\tau\leq \tau_{1}\}}+\expo^{-r({\tau}_{1})}\vartheta_{\ell_{0},\ell_{1}}\uno_{\{\tau> \tau_{1}\geq\tau_{0}\}}\notag\\
	&\quad+\uno_{\{\tau_{0}<\tau\}}\int_{0}^{{\tilde{\tau}}_{1}}\expo^{-r(\tms)}[h(X_{\tms},I_{\tms})\der\tms+g(X_{\tms-},I_{\tms})\circ\der \zeta_{\tms}]\bigg]\notag\\
	&\quad+\E_{x_{0},\ell_{0},\iota_{0}}[\expo^{-r({\tau}_{1})}u _{\ell_{1}}(X_{{\tau}_{1}},I_{{\tau}_{1}})\uno_{\{\tau> \tau_{1}\geq\tau_{0}\}}]\notag\\
	&\quad\vdots\quad\quad\quad\quad\quad\quad\vdots\quad\quad\quad\quad\quad\quad\vdots\notag\\
	&\leq \E_{x_{0},\ell_{0},\iota_{0}}\bigg[\expo^{-r({\tau})}f(X_{{\tau}},I_{{\tau}})\uno_{\{ \tau<\infty\}}+\sum_{i\geq0}\expo^{-r(\tau_{i+1})}\vartheta_{\ell_{i},\ell_{i+1}}\uno_{\{\tau_{i+1}<\tau\}}\notag\\
	&\quad+\int_{0}^{\tau}\expo^{-r(\tms)}[h(X_{\tms},I_{\tms})\der\tms+g(X_{\tms-},I_{\tms})\circ\der \zeta_{\tms}]\bigg]={V}_{\zeta,\varsigma}(x_{0},
	\ell_{0},\iota_{0}).
\end{align*}
Therefore, by the seen before, it is easily to check that $u_{\ell_{0}}(x_{0},\iota_{0})\leq V_{\ell_{0}}(x_{0},\iota_{0})\leq u_{\ell_{0}}(x_{0},\iota_{0})$ for $(x_{0},\ell_{0},\iota_{0})\in\overline{\set}\times\mathbb{M}\times\mathbb{I}$.
\end{proof}

\appendix

\section{Proofs of some results seen in the article}

\subsection{Proof of Lemma \ref{dm1}}\label{proof1}

The existence of the solution $v$ to \eqref{D1} will be argued using  Schaefer's fixed point theorem (see, i.e., \cite[Thm. 4 p. 539]{evans2}). First, by Theorem 6.14 of \cite{gilb},   notice that for each $\w\in\mathcal{C}^{0,\alpha}_{m, n}$, there exists a unique  $\vs_{\ell,\iota}\in\hol^{2,\alpha}(\overline{\set})$, with $(\ell,\iota)\in\mathbb{M}\times\mathbb{I}$, such that
\begin{equation}\label{D1u}
	\begin{split}
		[c_{\iota}-\widetilde{\dif}_{\iota}]\vs_{\ell,\iota}=   \overline{\Xi}_{\ell,\iota}\w,\ \text{on}\ \set,\quad\text{s.t.}\ \vs_{\ell,\iota}=f_{\iota},\ \text{in}\ \partial{\set},
	\end{split}
\end{equation}
where 
\begin{align}
	\overline{\Xi}_{\ell,\iota}\w&=h_{\iota}-\sum_{\kappa\in\mathbb{I}\setminus\{\iota\}}q_{\ell}(\iota,\kappa)[\w_{\ell,\iota}-\w_{\ell,\kappa}].\notag
\end{align}
Additionally, using \cite[Thm. 4.12, p. 85]{adams} and \cite[Thm. 1.2.19]{garroni}, we get that 
\begin{align}\label{elpd1v}
	\|\vs_{\ell,\iota}\|_{\hol^{1,\alpha}(\overline{\set})}&\leq C_{1}\big[1+\|\w\|_{\mathcal{C}^{0}_{m,n}(\overline{\set})}\big],\quad \text{for}\ (\ell,\iota)\in\mathbb{M}\times\mathbb{I},
\end{align}
for some  $C_{1}=C_{1}(d,\Lambda,\alpha,\theta)$. 	Let us define the mapping
\begin{align*}%\label{w2}
	T:(\mathcal{C}^{0,\alpha}_{m, n},\|\cdot\|_{\mathcal{C}^{0,\alpha}_{m, n}})\longrightarrow(\mathcal{C}^{0,\alpha}_{m, n},\|\cdot\|_{\mathcal{C}^{0,\alpha}_{m, n}})
\end{align*}
as $T[\w]=\vs$ for each $\w\in\mathcal{C}^{0,\alpha}_{m, n}$, where  $\vs\in\mathcal{C}^{2,\alpha}_{m, n}\subset\mathcal{C}^{0,\alpha}_{m, n}$ is the unique solution to the Dirichlet problem \eqref{D1u}. Observe $T$ maps bounded sets in $\mathcal{C}^{0,\alpha}_{m, n}$ into bounded sets in itself that are precompact in $\mathcal{C}^{0,\alpha}_{m, n}$, since \eqref{h0}--\eqref{h2} and\eqref{elpd1v} hold. Then,  by the uniqueness of the solution to  \eqref{D1u}, {it can be checked} that  $T$ is a continuous and compact mapping from $\mathcal{C}^{0,\alpha}_{m, n}$ to $\mathcal{C}^{0,\alpha}_{m, n}$. Then, by the seen previously, to use Schaefer's fixed point theorem, we only need to verify that the set
$$\mathcal{A}_{2}\eqdef\{\w\in\mathcal{C}^{0,\alpha}_{m, n}:\w=\varrho T[\w],\ \text{for some }\ \varrho\in[0,1]\}$$ 
is bounded uniformly with respect to the norm $\|\cdot\|_{\mathcal{C}^{0,\alpha}_{m, n}}$.  Let us show first that  $\mathcal{A}_{2}$ is uniformly bounded with respect to the norm $\|\cdot\|_{\mathcal{C}^{0}_{m,n}}$. {By the arguments seen above (Equation \eqref{eq_1}), observe that if $\varrho=0$,} $\w\equiv\overline{0}\in\mathcal{C}^{0,\alpha}_{m, n}$  where $\overline{0}$ is the null matrix function. 

\begin{lema}\label{A1}
	If $\w\in\mathcal{C}^{0,\alpha}_{m, n}$ is such that $T[\w]=\frac{1}{\varrho}\w=(\frac{1}{\varrho}\w_{\ell,\iota})_{(\ell,\iota)\in\mathbb{M}\times\mathbb{I}}$ for some $\varrho\in(0,1]$,  then \eqref{w1} holds.
	
\end{lema}
\begin{proof}
	Considering $(x_{\circ},\ell_{\circ},\iota_{\circ}),(x^{\circ},\ell^{\circ},\iota^{\circ})\in\overline{\set}\times\mathbb{M}\times\mathbb{I}$ be such that 
	$$\w_{\ell_{\circ},\iota_{\circ}}(x_{\circ})=\min_{(x,\ell,\iota)\in\overline{\set}\times\mathbb{M}\times\mathbb{I}}\w_{\ell,\iota}(x)\quad\text{and}\quad\w_{\ell^{\circ},\iota^{\circ}}(x^{\circ})=\max_{(x,\ell,\iota)\in\overline{\set}\times\mathbb{M}\times\mathbb{I}}\w_{\ell,\iota}(x),$$
	we get 
	\begin{equation*}
		\begin{split}
			&\deri^{1}\w_{\ell_{\circ},\iota_{\circ}}(x_{\circ})=\deri^{1}\w_{\ell^{\circ},\iota^{\circ}}(x^{\circ})=0,\ \tr[a_{\iota^{\circ}}\deri^{2}\w_{\ell^{\circ},\iota^{\circ}}](x^{\circ})\leq 0\leq\tr[a_{\iota_{\circ}}\deri^{2}\w_{\ell_{\circ},\iota_{\circ}}](x_{\circ}),\\
			&\w_{\ell_{\circ},\iota_{\circ}}(x_{\circ})-\w_{\ell_{\circ},\kappa}(x_{\circ} )\leq 0\ \text{for}\ \kappa\in\mathbb{I}\setminus\{\iota_{\circ}\},\ \w_{\ell^{\circ},\iota^{\circ}}(x^{\circ})-\w_{\ell_{\circ},\kappa}(x^{\circ} )\geq 0\ \text{for}\ \kappa\in\mathbb{I}\setminus\{\iota^{\circ}\}.
		\end{split} 
	\end{equation*}
	Therefore, from here, using \eqref{D1u}, and arguing in a similar way as in the proof of Lemma \eqref{B12}, we see that \eqref{w1} is also true for this case.
\end{proof}

\begin{proof}[Proof of Proposition \ref{princ1.0}. Existence] By the seen before and arguing in a similar way than in the proof of Proposition \ref{princ1.0} (existence), it follows immediately that \eqref{D1} has a solution $v$ in $\mathcal{C}^{4,\alpha}_{m,n}$.
\end{proof}	
\begin{proof}[Proof of Proposition \ref{dm1}. Uniqueness] The proof of uniqueness of the solution $v$ to   \eqref{D1} shall be given by contradiction. Assume that there are two solutions $\hat{v},v\in\mathcal{C}^{4,\alpha}_{m,n}$ to  \eqref{D1}. Let $\bar{v}=(\bar{v}_{\ell,\iota})_{(\ell,\iota)\in\mathbb{M}\times\mathbb{I}}\in\mathcal{C}^{4,\alpha}_{m,n}$ such that $\bar{v}_{\ell,\iota}\eqdef \hat{v}_{\ell,\iota}-v_{\ell,\iota}$ for $(\ell,\iota)\in\mathbb{M}\times\mathbb{I}$. Then,
	\begin{equation}\label{D1.1}
		\begin{split}
			[c_{\iota}-\dif_{\ell,\iota}]\bar{v}_{\ell,\iota}=0,\ \text{on}\ \set,\quad\text{s.t.}\ \bar{v}_{\ell,\iota}=0,\ \text{in}\ \partial{\set}.
		\end{split}
	\end{equation}
	Let $(x^{\circ},\ell^{\circ},\iota^{\circ})$ be in  $\overline{\set}\times\mathbb{M}\times\mathbb{I}$ such that $\bar{v}_{\ell^{\circ},\iota^{\circ}}(x^{\circ})=\max_{(x,\ell,\iota)\in\overline{\set}\times\mathbb{M}\times\mathbb{I}}\bar{v}_{\ell,\iota}(x)$.  If $x_{ \circ}\in\partial{\set}$, $\hat{v}_{\ell,\iota}-v_{\ell,\iota}\leq0$ in $\overline{\set}$  for $(\ell,\iota)\in\mathbb{M}\times\mathbb{I}$. If $x^{\circ}\in \set$,
	\begin{equation}\label{sup2.1v}
		\begin{split}
			&\deri^{1}\bar{v}_{\ell^{\circ},\iota^{\circ}}(x^{\circ})=0,\quad\tr[a_{\iota^{\circ}}(x^{\circ})\deri^{2}\bar{v}_{\ell^{\circ},\iota^{\circ}}(x^{\circ})]\leq 0,\\
			&\bar{v}_{\ell^{\circ},\iota^{\circ}}(x^{\circ})-\bar{v}_{\ell_{\circ},\kappa}(x^{\circ} )\geq 0\quad \text{for}\ \kappa\in\mathbb{I}\setminus\{\iota^{0}\}.
		\end{split} 
	\end{equation}
	Then, from \eqref{D1.1} and \eqref{sup2.1v},
	\begin{align}\label{NPDv.1}
		0&\geq\tr[a_{\iota^{\circ}}\deri^{2}\bar{v}_{\ell^{\circ},\iota^{\circ}}]\notag\\
		&=c_{\iota^{\circ}}\bar{v}_{\ell^{\circ},\iota^{\circ}}+\sum_{\kappa\in\mathbb{I}\setminus\{\ell_{\circ}\}}q_{\ell_{\circ}}(\iota_{\circ},\kappa)[\bar{v}_{\ell^{\circ},\iota^{\circ}}-\bar{v}_{\ell_{\circ},\kappa}]\geq c_{\iota^{\circ}}\bar{v}_{\ell^{\circ},\iota^{\circ}}\quad \text{at}\ x^{\circ}.
	\end{align}
	From \eqref{NPDv.1} and since $c_{\ell_{\circ}}>0$, we have that $\hat{v}_{\ell,\iota}(x)-v_{\ell,\iota}(x)\leq \hat{v}_{\ell^{\circ},\iota^{\circ}}(x^{\circ})-v_{\ell^{\circ},\iota^{\circ}}(x^{\circ})\leq0$ for $x\in\overline{\set}$ and $(\ell,\iota)\in\mathbb{M}\times\mathbb{I}$. Taking now $\bar{v}\eqdef v-\hat{v}$ and proceeding the same way than before, it follows immediately that  {$v_{{\ell,\iota}} -\hat{v}_{{\ell,\iota}}\leq0$ on $\overline{\set}$ for $ (\ell,\iota)\in\mathbb{M}\times\mathbb{I}$}.  Therefore $\hat{v} =v $ and from here we conclude that the system of equation \eqref{D1} has a unique solution $v$, whose components belong to $\hol^{4,\alpha}(\overline{\set})$.
\end{proof}

\subsection{Proof of Lemma \ref{Lb1}. Eq. \eqref{ap2}  {and \eqref{ap2.2}}}\label{proof2}

For each $(\ell,\iota)\in\mathbb{M}\times \mathbb{I}$, let us consider the auxiliary function
\begin{equation}\label{der1}
	w_{\ell,\iota}\eqdef\varpi^{2}|\deri^{1}u^{\varepsilon,\delta}_{\ell,\iota}|^{2}-\lambda A_{\varepsilon,\delta} u^{\varepsilon,\delta}_{\ell,\iota},\ \text{on $\overline{\set}$},
\end{equation}  
where   $\lambda\geq1$ is a constant that shall be selected later on and 
\begin{equation}\label{e3}
	A_{\varepsilon,\delta}\eqdef\max_{(x,\ell,\iota)\in\overline{\set}\times\mathbb{M}\times\mathbb{I}}\varpi(x)|\deri^{1}u^{\varepsilon,\delta}_{\ell,\iota}(x)|.
\end{equation}
 We shall show that $w_{\ell,\iota}$ satisfies \eqref{partu4}. In particular,\eqref{partu4} holds when $w_{\ell,\iota}$ is evaluated at its maximum  $x_{\lambda}\in\set$, which helps to see that \eqref{ap2} is true. 
\begin{lema}\label{aux1}
	Let $w_{\ell,\iota}$ be the auxiliary function given by \eqref{D2.1}. 
	Then, there exists a positive constant $C_{7}=C_{7}(d,\Lambda ,1/\theta,K_{2},C_{4})$ such that for $(\ell,\iota)\in\mathbb{M}\times\mathbb{I}$,
	\begin{align}\label{partu4}
		-\tr[a_{\iota}\deri^{2}w_{\ell,\iota}]&\leq C_{7}|\deri^{1}u^{\varepsilon,\delta}_{\ell,\iota}|^{2}+C_{7} [1 +\lambda A_{\varepsilon, \delta}]|\deri^{1}u^{\varepsilon,\delta}_{\ell,\iota}|+\lambda A_{\varepsilon, \delta}C_{7}   \notag\\ &\quad-\psi'_{\varepsilon,\ell,\iota}(\cdot)[2\langle\deri^{1}u^{\varepsilon,\delta}_{\ell,\iota},\deri^{1}w_{\ell,\iota}\rangle+\lambda A_{\varepsilon, \delta}|\deri^{1}u^{\varepsilon,\delta}_{\ell,\iota}|^{2}-C_{7}|\deri^{1}u^{\varepsilon,\delta}_{\ell,\iota}|^{2}A_{\varepsilon,\delta}-C_{7}A_{\varepsilon,\delta}]\notag\\
		&\quad-\sum_{\kappa\in\mathbb{I}\setminus\{\iota\}}q_{\ell}(\iota,\kappa)[w_{\ell,\iota}-w_{\ell,\kappa}]-\sum_{\ell'\in\mathbb{M}\setminus\{\ell\}}\psi'_{\delta,\ell,\ell',\iota}(\cdot)[w_{\ell,\iota}-w_{\ell',\iota}],\ \text{on}\ B_{\beta' r},
	\end{align}
	where $\psi_{\varepsilon,\ell,\iota}(\cdot)$, $\psi_{\delta,\ell,\ell',\iota}(\cdot)$  denote $\psi_{\varepsilon}(|\deri^{1}u^{\varepsilon,\delta}_{\ell,\iota}|^{2}-g^{2}_{\iota})$,  $\psi_{\delta}(u^{\varepsilon,\delta}_{\ell,\iota}-u^{\varepsilon,\delta}_{\ell',\iota} -\vartheta_{\ell,\ell'})$, respectively.
\end{lema}

\begin{proof}[Proof of Lemma \ref{Lb1}. Eq. \eqref{ap2}]
	
	{Without loss of generality, let us assume that $  A_{\varepsilon,\delta}>1$ since   if $  A_{\varepsilon,\delta}\leq1$, we obtain a bound for $  A_{\varepsilon,\delta}$ that is independent of  $\varepsilon,\delta$ and hence, we obtain \eqref{ap2}}. Let $x_{\lambda}
	\in\overline{\set}$ and $(\ell_{\lambda},\iota_{\lambda})\in\mathbb{M}\times\mathbb{I}$ {(depending on $\lambda$)} be such that $$w_{\ell_{\lambda},\iota_{\lambda}}(x_{\lambda})=\max_{(x,\ell,\iota)\in\overline{\set}\times\mathbb{M}\times\mathbb{I}}w_{\ell,\iota}(x).$$
	From here, by \eqref{ap1} and definition of $w_{\ell,\iota}$; see \eqref{der1}, it gives 
	\begin{align*}
		\varpi^{2}(x)|\deri^{1}u_{\ell,\iota}(x)|^{2}%&\leq |\deri^{1}u_{\ell_{\lambda}}(x_{\lambda})|^{2}+\lambda   A_{\varepsilon,\delta}[u_{\ell,\iota}(x)-u_{\ell_{\lambda}}(x_{\lambda})]\notag\\
		&\leq\varpi^{2}(x_{\lambda})|\deri^{1}u_{\ell_{\lambda},\iota_{\lambda}}(x_{\lambda})|^{2}+\lambda A_{\varepsilon,\delta}  C_{1},
	\end{align*}
	for $x\in \overline{\set}$ and $(\ell,\iota)\in\mathbb{M}\times\mathbb{I}$. Then, from here and  using \eqref{e3}, it yields for each $\varrho>0$ small enough, there is $x_{
		\varrho}\in\overline{\set}$ such that $[A_{\varepsilon,\delta}-\varrho]^{2}\leq \varpi^{2}(x_{\varrho})|\deri^{1}u_{\ell,\iota}(x_{\varrho})|^{2}\leq \varpi^{2}(x_{\lambda})|\deri^{1}u_{\ell_{\lambda},\iota_{\lambda}}(x_{\lambda})|^{2}+\lambda A_{\varepsilon,\delta}  C_{1}$. Thus,  letting $\varrho\rightarrow 0$ and since $A_{\varepsilon,\delta}>1$,
	\begin{equation}\label{partu5.0}
		\varpi|\deri^{1}u_{\ell,\iota}|\leq A_{\varepsilon,\delta}\leq \varpi^{2}(x_{\lambda})|\deri^{1}u_{\ell_{\lambda},\iota_{\lambda}}(x_{\lambda})|^{2}+\lambda  C_{1}\quad \text{on}\ \overline{\set}.
	\end{equation}
	From here, we see that to verify \eqref{ap2}, it is enough to check that $\varpi(x_{\lambda})|\deri^{1}u_{\ell_{\lambda},\iota_{\lambda}}(x_{\lambda})|$ is bounded by a non-negative  constant which is independent of $\varepsilon$ and $\delta$. If $x_{\lambda}\in\overline{\set}\setminus B_{\beta' r}$, $\varpi(x_{\lambda})|\deri^{1}u_{\ell_{\lambda},\iota_{\lambda}}(x_{\lambda})|=0$, and therefore   $C_{5}\eqdef\lambda C_{1}$. Let $x_{\lambda}$ be in $B_{\beta'r}$. It is well known that at $x_{\lambda}$,
	\begin{equation*}%\label{e1}
		\begin{split}
			&\deri^{1}w_{\ell_{\lambda},\iota_{\lambda}}=0,\quad \tr[a_{\iota_{\lambda}}\deri^{2}w_{\ell_{\lambda},\iota_{\lambda}}]\leq 0\\
			&[w_{\ell_{\lambda},\iota_{\lambda}}-w_{\ell,\kappa}]\geq 0,\ \text{for}\ \kappa\in\mathbb{I}\setminus\{\iota_{\lambda}\},\\
			&[w_{\ell_{\lambda},\iota_{\lambda}}-w_{\ell',\iota_{\lambda}}]\geq 0,\ \text{for}\ \ell'\in\mathbb{M}\setminus\{\ell_{\lambda}\},
		\end{split}
	\end{equation*}
	Then, from here and  \eqref{partu4},
	\begin{align}\label{e2}
		0&\leq C_{7}|\deri^{1}u_{\ell_{\lambda},\iota_{\lambda}}|^{2}+C_{7} [1 +\lambda A_{\varepsilon,\delta}]|\deri^{1}u_{\ell_{\lambda},\iota_{\lambda}}|+\lambda C_{7}   \notag\\ &\quad-\psi'_{\varepsilon,\ell_{\lambda},\iota_{\lambda}}(\cdot)[\lambda A_{\varepsilon,\delta}|\deri^{1}u_{\ell_{\lambda},\iota_{\lambda}}|^{2}-C_{7}|\deri^{1}u_{\ell_{\lambda},\iota_{\lambda}}|^{2}A_{\varepsilon,\delta}-C_{7}A_{\varepsilon,\delta}]\quad\text{at}\ x_{\lambda}.
	\end{align}
	On the other hand, notice that either $\psi'_{\varepsilon,\ell_{\lambda},\iota_{\lambda}}(\cdot)<\frac{1}{\varepsilon}$ or $\psi'_{\varepsilon,\ell_{\lambda},\iota_{\lambda}}(\cdot)=\frac{1}{\varepsilon}$  at $x_{\lambda}$. If $\psi'_{\varepsilon,\ell_{\lambda},\iota_{\lambda}}(\cdot)<\frac{1}{\varepsilon}$ at $x_{\lambda}$, by definition of $\psi_{\varepsilon}$, given in \eqref{p12.1}, it follows that  $|\deri^{1}u_{\ell_{\lambda},\iota_{\lambda}}|^{2}- g ^{2}_{\iota_{\lambda}}\leq2\varepsilon$ at $x_{\lambda}$. 
	It implies that 
	$\varpi^{2}|\deri^{1}u_{\ell_{\lambda},\iota_{\lambda}}|^{2}\leq \Lambda^{2}+2$ at $x_{\lambda}$.
	Then, from  \eqref{partu5.0} and taking  $C_{5}\eqdef 2+\Lambda^{2}+\lambda C_{1}$, it follows \eqref{ap2}. Now, assume that  $\psi'_{\varepsilon,\ell_{\lambda},\iota_{\lambda}}(\cdot)=\frac{1}{\varepsilon}$ at $x_{\lambda}$.  Then, taking $\lambda>\max\{1,2C_{7}\}$ fixed, and using \eqref{e2},  we get  
	\begin{align}\label{e4}
		0
		&\leq [2C_{7}-\lambda]|\deri^{1}u_{\ell_{\lambda},\iota_{\lambda}}|^{2}+C_{7} [1 +\lambda]|\deri^{1}u_{\ell_{\lambda},\iota_{\lambda}}|+\lambda C_{7}\quad \text{at}\ x_{\lambda}.
	\end{align}
	From here,  it yields that  $|\deri^{1}u_{\ell_{\lambda}}(x_{\lambda})|<K_{3}$, for some $K_{3}=K_{3}(d,\Lambda ,\alpha)$.  Therefore, taking $C_{5}\eqdef K_{3}+\lambda C_{1}$ and using \eqref{partu5.0}, we get \eqref{ap2}.
\end{proof}

 {
	\begin{proof}[Proof of Lemma \ref{Lb1} Eq. \eqref{ap2.2}]
		Let us assume \eqref{h6} holds and take $\varpi\equiv1$ on $\overline{\set}$  in \eqref{der1} and \eqref{e3}. Arguing in the same way as before, it is enough to verify that $|\deri^{1}u_{\ell_{\lambda},\iota_{\lambda}}(x_{\lambda})|^{2}$ is bounded by a positive constant, which is independent of $\varepsilon$ and $\delta$, when $x_{\lambda}
		\in\overline{\set}$ and $(\ell_{\lambda},\iota_{\lambda})\in\mathbb{M}\times\mathbb{I}$ {(depending on $\lambda$)} are such that $$w_{\ell_{\lambda},\iota_{\lambda}}(x_{\lambda})=\max_{(x,\ell,\iota)\in\overline{\set}\times\mathbb{M}\times\mathbb{I}}w_{\ell,\iota}(x)=\max_{(x,\ell,\iota)\in\overline{\set}\times\mathbb{M}\times\mathbb{I}}\{|\deri^{1}u^{\varepsilon,\delta}_{\ell,\iota}(x)|^{2}-\lambda A_{\varepsilon,\delta} u^{\varepsilon,\delta}_{\ell,\iota}(x)\}.$$  
		If $x\in\partial\set$, by  \eqref{ap1.1} and since $u_{\ell,\iota}\geq0$, it follows immediately $|\deri^{1}u_{\ell_{\lambda},\iota_{\lambda}}(x_{\lambda})|^{2}\leq C^{2}_{2}$.
		Proceeding in the same way as before it can be verified that   $|\deri^{1}u_{\ell_{\lambda},\iota_{\lambda}}(x_{\lambda})|^{2}$ is bounded by a positive constant that is independent of $\varepsilon$ and $\delta$, when $x_{\lambda}\in\set$.
	\end{proof}
}

\begin{proof}[Proof of Lemma \ref{aux1}]
	Consider $w_{\ell,\iota}$ as in \eqref{der1} for $(\ell,\iota)\in\mathbb{M}\times\mathbb{I}$. Taking first and second derivatives   in $w_{\ell,\iota}$ on ${B}_{\beta'r}$, it can be checked that
	\begin{align}\label{dercot1.0}
		\partial_{i}w_{\ell,\iota}&=|\deri^{1}u_{\ell,\iota}|^{2}\partial_{i}\varpi^{2}+\varpi^{2}\partial_{i}|\deri^{1}u_{\ell,\iota}|^{2}-\lambda A_{\varepsilon, \delta}\partial_{i}u_{\ell,\iota}\\
		-\tr[a_{\iota}\deri^{2}w_{\ell,\iota}]&=-|\deri^{1}u_{\ell,\iota}|^{2}\tr[a_{\iota}\deri^{2}\varpi^{2}]-2\langle a_{\iota} \deri^{1}\varpi^{2},\deri^{1}|\deri^{1}u_{\ell,\iota}|^{2}\rangle\notag\\
		&\quad-\varpi^{2}\tr[a_{\iota}\deri^{2}|\deri^{1}u_{\ell,\iota}|^{2}]+\lambda A_{\varepsilon, \delta}\tr[a_{\iota}\deri^{2}u_{\ell,\iota}].\notag
	\end{align}	 
	From here and noticing that from \eqref{H2},
	\begin{align*}
		\tr[a_{\iota}\deri^{2}|\deri^{1}u_{\ell,\iota}|^{2}]%&=2\sum_{i}[\langle a_{\iota}\deri^{1}\partial_{i}u_{\ell,\iota},\deri^{1}\partial_{i}u_{\ell,\iota}\rangle+\partial_{i}u_{\ell,\iota}\tr[a_{\iota}\deri^{2}\partial_{i}u_{\ell,\iota}]]\notag\\
		&\geq 2\theta|\deri^{2}u_{\ell,\iota}|^{2}+2\sum_{i}\partial_{i}u_{\ell,\iota}\tr[a_{\iota}\deri^{2}\partial_{i}u_{\ell,\iota}],
	\end{align*}
	it follows that
	\begin{align}\label{dercot1}
		-\tr[a_{\iota}\deri^{2}w_{\ell,\iota}]
		%&\leq-|\deri^{1}u_{\ell,\iota}|^{2}\tr[a_{\iota}\deri^{2}\varpi^{2}]-2\langle a_{\iota} \deri^{1}\varpi^{2},\deri^{1}|\deri^{1}u_{\ell,\iota}|^{2}\rangle\notag\\
		%&\quad-\varpi^{2}\bigg[2\theta|\deri^{2}u_{\ell,\iota}|^{2}+2\sum_{i}\partial_{i}u_{\ell,\iota}\tr[a_{\iota}\deri^{2}\partial_{i}u_{\ell,\iota}]\bigg]+\lambda A_{\varepsilon, \delta}\tr[a_{\iota}\deri^{2}u_{\ell,\iota}]\notag\\
		&\leq-|\deri^{1}u_{\ell,\iota}|^{2}\tr[a_{\iota}\deri^{2}\varpi^{2}]-8\varpi\sum_{i}\partial_{i}u_{\ell,\iota}\langle a_{\iota} \deri^{1}\varpi,\deri^{1}\partial_{i}u_{\ell,\iota}\rangle\notag\\
		&\quad-\varpi^{2}\bigg[2\theta|\deri^{2}u_{\ell,\iota}|^{2}+2\sum_{i}\partial_{i}u_{\ell,\iota}\tr[a_{\iota}\deri^{2}\partial_{i}u_{\ell,\iota}]\bigg]+\lambda A_{\varepsilon, \delta}\tr[a_{\iota}\deri^{2}u_{\ell,\iota}].
	\end{align}	
	
	Meanwhile, from \eqref{eq2} and \eqref{NPD.1},
	\begin{align}
		\lambda A_{\varepsilon, \delta} \tr[a_{\iota}\deri^{2} u_{\ell,\iota}]&=  \lambda A_{\varepsilon, \delta}  \bigg[\widetilde{D}_{1}u_{\ell,\iota}+\psi_{\varepsilon,\ell,\iota}(\cdot) \notag\\
		&\quad+\displaystyle\sum_{\ell'\in\mathbb{M}\setminus\{\ell\}}\psi_{\delta,\ell,\ell',\iota}(\cdot)+\sum_{\kappa\in\mathbb{I}\setminus\{\iota\}}q_{\ell}(\iota,\kappa)[u_{\ell,\iota}-u_{\ell,\kappa}]\bigg]
	\end{align}
	where   $\widetilde{D}_{1}u_{\ell,\iota}\eqdef\langle b_{\iota},\deri^{1}u_{\ell,\iota}\rangle+c_{\iota}u_{\ell,\iota}-h_{\iota}$. Now, differentiating \eqref{NPD.1}, multiplying by $2\partial_{i}u_{\ell,\iota}$ and taking summation over all $i$'s, we see that
	\begin{align}\label{dercot2}
		-2\sum_{i}\tr[a_{\iota}\deri^{2}\partial_{i} u_{\ell,\iota}]\partial_{i} u_{\ell,\iota}
		&=\widetilde{D}_{2} u_{\ell,\iota} -2\psi'_{\varepsilon,\ell,\iota}(\cdot)\langle\deri^{1} u_{\ell,\iota},\deri^{1}[|\deri^{1} u_{\ell,\iota}|^{2}- g_{\iota} ^{2}]\rangle\notag\\
		&\quad-2\sum_{\ell'\in\mathbb{M}\setminus\{\ell\}}\psi'_{\delta,\ell,\ell',\iota}(\cdot)[|\deri^{1}u_{\ell,\iota}|^{2}-\langle\deri^{1}u_{\ell,\iota},\deri^{1}u_{\ell',\iota} \rangle]\notag\\
		&\quad-2\sum_{\kappa\in\mathbb{I}\setminus\{\iota\}}q_{\ell}(\iota,\kappa)[|\deri^{1}u_{\ell,\iota}|^2-\langle\deri^{1}u_{\ell,\iota},\deri^{1}u_{\ell,\kappa}\rangle],
	\end{align}
	where  
	\begin{align}\label{dercot2.1}
		\widetilde{D}_{2}u_{\ell,\iota}&\eqdef 2\sum_{i}\partial_{i}u_{\ell,\iota}\tr[[\partial_{k}a_{\iota}]\deri^{2}u_{\ell,\iota}]-2\langle\deri^{1}u_{\ell,\iota},\deri^{1}[\langle b_{\iota},\deri^{1}u_{\ell,\iota}\rangle+c_{\iota}u_{\ell,\iota}-h_{\iota}]\rangle.
	\end{align}
	Then, from \eqref{dercot1}--\eqref{dercot2}, it can be shown that 
	\begin{align}\label{dercot1.2}
		-\tr[a_{\iota}\deri^{2}w_{\ell,\iota}]%&\leq-2\theta\varpi^{2}|\deri^{2}u_{\ell,\iota}|^{2}-|\deri^{1}u_{\ell,\iota}|^{2}\tr[a_{\iota}\deri^{2}\varpi^{2}]\notag\\
		%&\quad-8\varpi\sum_{i}\partial_{i}u_{\ell,\iota}\langle a_{\iota} \deri^{1}\varpi,\deri^{1}\partial_{i}u_{\ell,\iota}\rangle\notag\\
		%&\quad+\varpi^{2}\bigg[\widetilde{D}_{2} u_{\ell,\iota} -2\psi'_{\varepsilon,\ell,\iota}(\cdot)\langle\deri^{1} u_{\ell,\iota},\deri^{1}[|\deri^{1} u_{\ell,\iota}|^{2}- g_{\iota} ^{2}]\rangle\notag\\
		%&\quad-2\sum_{\ell'\in\mathbb{M}\setminus\{\ell\}}\psi'_{\delta,\ell,\ell',\iota}(\cdot)[|\deri^{1}u_{\ell,\iota}|^{2}-\langle\deri^{1}u_{\ell,\iota},\deri^{1}u_{\ell',\iota} \rangle]\notag\\
		%&\quad-2\sum_{\kappa\in\mathbb{I}\setminus\{\iota\}}q_{\ell}(\iota,\kappa)[|\deri^{1}u_{\ell,\iota}|^2-\langle\deri^{1}u_{\ell,\iota},\deri^{1}u_{\ell,\kappa}\rangle]\bigg]\notag\\
		%&\quad+\lambda A_{\varepsilon, \delta}\bigg[\widetilde{D}_{1}u_{\ell,\iota}+\psi_{\varepsilon,\ell,\iota}(\cdot)+\displaystyle\sum_{\ell'\in\mathbb{M}\setminus\{\ell\}}\psi_{\delta,\ell,\ell',\iota}(\cdot)+\sum_{\kappa\in\mathbb{I}\setminus\{\iota\}}q_{\ell}(\iota,\kappa)[u_{\ell,\iota}-u_{\ell,\kappa}] \bigg]\notag\\
		&\leq-2\theta\varpi^{2}|\deri^{2}u_{\ell,\iota}|^{2}-|\deri^{1}u_{\ell,\iota}|^{2}\tr[a_{\iota}\deri^{2}\varpi^{2}]\notag\\
		&\quad-8\varpi\sum_{i}\partial_{i}u_{\ell,\iota}\langle a_{\iota} \deri^{1}\varpi,\deri^{1}\partial_{i}u_{\ell,\iota}\rangle+\varpi^{2}\widetilde{D}_{2} u_{\ell,\iota}+\lambda A_{\varepsilon, \delta} \widetilde{D}_{1}u_{\ell,\iota}\notag\\
		&\quad -2\varpi^{2}\psi'_{\varepsilon,\ell,\iota}(\cdot)\langle\deri^{1} u_{\ell,\iota},\deri^{1}[|\deri^{1} u_{\ell,\iota}|^{2}- g_{\iota} ^{2}]\rangle+\lambda A_{\varepsilon, \delta} \psi_{\varepsilon,\ell,\iota}(\cdot)\notag\\
		&\quad-\sum_{\ell'\in\mathbb{M}\setminus\{\ell\}}\{2\varpi^{2}\psi'_{\delta,\ell,\ell',\iota}(\cdot)[|\deri^{1}u_{\ell,\iota}|^{2}-\langle\deri^{1}u_{\ell,\iota},\deri^{1}u_{\ell',\iota} \rangle]-\lambda A_{\varepsilon, \delta}\psi_{\delta,\ell,\ell',\iota}(\cdot)\}\notag\\
		&\quad-\sum_{\kappa\in\mathbb{I}\setminus\{\iota\}}q_{\ell}(\iota,\kappa)\{2\varpi^{2}[|\deri^{1}u_{\ell,\iota}|^2-\langle\deri^{1}u_{\ell,\iota},\deri^{1}u_{\ell,\kappa}\rangle]-\lambda A_{\varepsilon,\delta}[u_{\ell,\iota}-u_{\ell,\kappa}]\}.\notag\\
	\end{align}	
	By \eqref{a4}, \eqref{h2} and \eqref{ap1}, notice that
	\begin{align}\label{dercot1.1}
		&-2\theta\varpi^{2}|\deri^{2}u_{\ell,\iota}|^{2}-|\deri^{1}u_{\ell,\iota}|^{2}\tr[a_{\iota}\deri^{2}\varpi^{2}]\notag\\
		&-8\varpi\sum_{i}\partial_{i}u_{\ell,\iota}\langle a_{\iota} \deri^{1}\varpi,\deri^{1}\partial_{i}u_{\ell,\iota}\rangle+\varpi\widetilde{D}_{2} u_{\ell,\iota}+\lambda A_{\varepsilon, \delta} \widetilde{D}_{1}u_{\ell,\iota}\notag\\
		%&\leq -2\theta\varpi^{2}|\deri^{2}u_{\ell,\iota}|^{2}+2\Lambda K_{1}d^{2}|\deri^{1}u_{\ell,\iota}|^{2}(1+\varpi)+8\Lambda K_{1}d^{3}\varpi |\deri^{1}u_{\ell,\iota}|\,|\deri^{2}u_{\ell,\iota}|\notag\\
		%&\quad+\varpi[2\Lambda |\deri^{1}u_{\ell,\iota}|^{2}+2\Lambda[1+d^{3}]|\deri^{1}u_{\ell,\iota}|\,|\deri^{2}u_{\ell,\iota}|+2\Lambda[1+C_{1}]|\deri^{1}u_{\ell,\iota}|]\notag\\
		%&\quad+\lambda A_{\varepsilon, \delta}\Lambda[|\deri^{1}u_{\ell,\iota}|+C_{1}]\notag\\
		%&= -2\theta\varpi^{2}|\deri^{2}u_{\ell,\iota}|^{2}+2[\Lambda K_{1}d^{2}+\Lambda d]\varpi|\deri^{1}u_{\ell,\iota}|^{2}+2\Lambda[4 K_{1}d^{3}+1+d^{3}]\varpi |\deri^{1}u_{\ell,\iota}|\,|\deri^{2}u_{\ell,\iota}|\notag\\
		%&\quad+2\Lambda[1+C_{1}]\varpi|\deri^{1}u_{\ell,\iota}|+\lambda A_{\varepsilon, \delta}\Lambda[|\deri^{1}u_{\ell,\iota}|+C_{1}]+2\Lambda K_{1}d^{2}|\deri^{1}u_{\ell,\iota}|^{2}\notag\\
		%&\leq 2[\Lambda K_{1}d^{2}+\Lambda d]\varpi|\deri^{1}u_{\ell,\iota}|^{2}+\dfrac{1}{2\theta}\Lambda^{2}[4 K_{1}d^{3}+1+d^{3}]^{2}|\deri^{1}u_{\ell,\iota}|^{2}\notag\\
		%&\quad+2\Lambda[1+C_{1}]\varpi|\deri^{1}u_{\ell,\iota}|+\lambda A_{\varepsilon, \delta}\Lambda[|\deri^{1}u_{\ell,\iota}|+C_{1}]+2\Lambda K_{1}d^{2}|\deri^{1}u_{\ell,\iota}|^{2}\notag\\
		&\leq 2\bigg[2\Lambda K_{2}d^{2}+\Lambda d +\dfrac{1}{4\theta}\Lambda^{2}[4 K_{2}d^{3}+1+d^{3}]^{2}\bigg]|\deri^{1}u_{\ell,\iota}|^{2}\notag\\
		&\quad+2\Lambda\bigg[1+C_{1}+\dfrac{\lambda A_{\varepsilon, \delta}}{2}\bigg]|\deri^{1}u_{\ell,\iota}|+\lambda A_{\varepsilon, \delta}\Lambda C_{4}\notag\\
		&\leq K_{4}|\deri^{1}u_{\ell,\iota}|^{2}+K_{4}[1+\lambda A_{\varepsilon, \delta}]|\deri^{1}u_{\ell,\iota}|+\lambda A_{\varepsilon, \delta} K_4,
	\end{align}
	for some $K_{4}=K_{4}(d,\Lambda ,1/\theta,K_{2},C_{4})$. On the other hand  by \eqref{dercot1.0}, it can be checked that 
	\begin{align}\label{eq3.0}
		-\varpi^{2}\langle\deri^{1} u_{\ell,\iota}&,\deri^{1}[|\deri^{1} u_{\ell,\iota}|^{2}-g^{2}_{\iota}]\rangle\notag\\
		&=-\langle\deri^{1}u_{\ell,\iota},\deri^{1}w_{\ell,\iota}\rangle-\lambda A_{\varepsilon, \delta}|\deri^{1}u_{\ell,\iota}|^{2}\notag\\
		&\quad+|\deri^{1}u_{\ell,\iota}|^{2}\langle\deri^{1}u_{\ell,\iota},\deri^{1}\varpi^{2}\rangle+\varpi^{2}\langle\deri^{1} u_{\ell,\iota},\deri^{1}[g^{2}_{\iota}]\rangle\notag\\
		%&\leq-\langle\deri^{1}u_{\ell,\iota},\deri^{1}w_{\ell,\iota}\rangle-\lambda A_{\varepsilon, \delta}|\deri^{1}u_{\ell,\iota}|^{2}\notag\\
		%&\quad+2[K_{1}]^{2}d|\deri^{1}u_{\ell,\iota}|^{2}\varpi|\deri^{1}u_{\ell,\iota}|+\varpi^{2}|\deri^{1} u_{\ell,\iota}|\,|\deri^{1}[g^{2}_{\iota}]|\notag\\
		&\leq-\langle\deri^{1}u_{\ell,\iota},\deri^{1}w_{\ell,\iota}\rangle-\lambda A_{\varepsilon, \delta}|\deri^{1}u_{\ell,\iota}|^{2}+2dK_{2}|\deri^{1}u_{\ell,\iota}|^{2}A_{\varepsilon,\delta}+2d\Lambda A_{\varepsilon, \delta}.
	\end{align}
	Using \eqref{der1} and since $|y_{1}|^{2}-|y_{2}|^{2}=2[|y_{1}|^{2}-\langle y_{1},y_{2}\rangle]-|y_{1}-y_{2}|^{2}\leq2[|y_{1}|^{2}-\langle y_{1},y_{2}\rangle]$, for $y_{1},y_{2}\in\R^{d}$, it yields that 
	\begin{multline}\label{eq4.0}
		-2\varpi^{2}[|\deri^{1}u_{\ell,\iota}|^{2}-\langle\deri^{1}u_{\ell,\iota},\deri^{1}u_{\ell',\kappa} \rangle]\\
		\leq-[w_{\ell,\iota}-w_{\ell',\kappa}]-\lambda A_{\varepsilon, \delta}[u_{\ell,\iota}-u_{\ell',\kappa}]\quad\text{for}\ (\ell',\kappa)\in\mathbb{M}\times\mathbb{I}.
	\end{multline}
	Then, from \eqref{eq3.0}--\eqref{eq4.0},
	\begin{align}\label{eq5}
		&-2\varpi^{2}\psi'_{\varepsilon,\ell,\iota}(\cdot)\langle\deri^{1} u_{\ell,\iota},\deri^{1}[|\deri^{1} u_{\ell,\iota}|^{2}- g_{\iota} ^{2}]\rangle+\lambda A_{\varepsilon, \delta} \psi_{\varepsilon,\ell,\iota}(\cdot)\notag\\
		&-\sum_{\ell'\in\mathbb{M}\setminus\{\ell\}}\{2\varpi^{2}\psi'_{\delta,\ell,\ell',\iota}(\cdot)[|\deri^{1}u_{\ell,\iota}|^{2}-\langle\deri^{1}u_{\ell,\iota},\deri^{1}u_{\ell',\iota} \rangle]-\lambda A_{\varepsilon, \delta}\psi_{\delta,\ell,\ell',\iota}(\cdot)\}\notag\\
		&-\sum_{\kappa\in\mathbb{I}\setminus\{\iota\}}q_{\ell}(\iota,\kappa)\{2\varpi^{2}[|\deri^{1}u_{\ell,\iota}|^2-\langle\deri^{1}u_{\ell,\iota},\deri^{1}u_{\ell,\kappa}\rangle]-\lambda A_{\varepsilon,\delta}[u_{\ell,\iota}-u_{\ell,\kappa}]\}\notag\\
		&\leq -2\psi'_{\varepsilon,\ell,\iota}(\cdot)[\langle\deri^{1}u_{\ell,\iota},\deri^{1}w_{\ell,\iota}\rangle+\lambda A_{\varepsilon, \delta}|\deri^{1}u_{\ell,\iota}|^{2}\notag\\
		&\quad-2dK_{2}|\deri^{1}u_{\ell,\iota}|^{2}A_{\varepsilon,\delta}-2d\Lambda^{2}A_{\varepsilon,\delta}]+\lambda A_{\varepsilon, \delta} \psi_{\varepsilon,\ell,\iota}(\cdot)-\sum_{\kappa\in\mathbb{I}\setminus\{\iota\}}q_{\ell}(\iota,\kappa)[w_{\ell,\iota}-w_{\ell,\kappa}]\notag\\
		&\quad-\sum_{\ell'\in\mathbb{M}\setminus\{\ell\}}\{\psi'_{\delta,\ell,\ell',\iota}(\cdot)[[w_{\ell,\iota}-w_{\ell',\iota}]+\lambda A_{\varepsilon, \delta}[u_{\ell,\iota}-u_{\ell',\iota}]]-\lambda A_{\varepsilon, \delta}\psi_{\delta,\ell,\ell',\iota}(\cdot)\}\notag\\
		&\leq -\psi'_{\varepsilon,\ell,\iota}(\cdot)[2\langle\deri^{1}u_{\ell,\iota},\deri^{1}w_{\ell,\iota}\rangle+\lambda A_{\varepsilon, \delta}|\deri^{1}u_{\ell,\iota}|^{2}-4dK_{2}|\deri^{1}u_{\ell,\iota}|^{2}A_{\varepsilon,\delta}-4d\Lambda^{2}A_{\varepsilon,\delta}]\notag\\
		&\quad-\sum_{\kappa\in\mathbb{I}\setminus\{\iota\}}q_{\ell}(\iota,\kappa)[w_{\ell,\iota}-w_{\ell,\kappa}]-\sum_{\ell'\in\mathbb{M}\setminus\{\ell\}}\psi'_{\delta,\ell,\ell',\iota}(\cdot)[w_{\ell,\iota}-w_{\ell',\iota}],
	\end{align}
	due to $\psi_{\cdot}(r)\leq \psi'_{\cdot}(r)r$, for all $r\in\R$, $g^{2}_{\iota}\geq0$ and $\vartheta_{\ell,\ell'}\geq0$. Therefore, applying \eqref{dercot1.1} and \eqref{eq5} in \eqref{dercot1.2},  we get that \eqref{partu4} is true for some $C_{7}=C_{7}(d,\Lambda ,1/\theta,K_{2},C_{4})$.
\end{proof}

\subsection{Proof of Lemma \ref{Lb1}. Eq. \eqref{ap3}}\label{proof3}
Let us define the auxiliary function $\phi_{\ell,\iota}$  as
\begin{equation}\label{D2.1}
	\phi_{\ell,\iota}\eqdef\varpi^{2}|\deri^{2}u^{\varepsilon,\delta}_{\ell,\iota}|^{2}+\lambda A^{1}_{\varepsilon,\delta}\varpi\tr[\alpha_{\iota_{0}}\deri^{2}u^{\varepsilon,\delta}_{\ell,\iota}]+\mu|\deri^{1}u^{\varepsilon,\delta}_{\ell,\iota}|^{2}\ \text{ on}\ \set,
\end{equation}
with $A^{1}_{\varepsilon,\delta}\eqdef\max_{(x,\ell,\iota)\in\overline{\set}\times\mathbb{M}\times\mathbb{I}}\varpi(x)|\deri^{2}u^{\varepsilon,\delta}_{\ell,\iota}(x)|$, $\lambda\geq\max\{1,2/\theta\}$, $\mu\geq1$ fixed, and $\alpha_{\iota_{0}}=(\alpha_{\iota_{0}\,ij})_{d\times d}$ be such that $\alpha_{\iota_{0}\,ij}\eqdef  a_{\iota_{0}\,ij}(x_{0})$, where $x_{0}\in\overline{\set}$,  $(\ell_{0},\iota_{0})\in\mathbb{M}\times\mathbb{I}$ are fixed.  We shall show that $\phi_{\ell,\iota}$ satisfies \eqref{d7}. In particular, \eqref{d7} holds when $\phi_{\ell,\iota}$ is evaluated at its maximum  $x_{\mu}\in\set$, which helps to see that \eqref{ap3} is true. 
\begin{lema}\label{D2.0.0}
	Let $\phi_{\ell,\iota}$ be the auxiliary function given by \eqref{D2.1}. 
	Then, there exists a positive constant $C_{8}=C_{8}(d,\Lambda ,\alpha,K_{1})$ such that on $(x,\ell)\in\overline{B}_{\beta'r} \times\mathbb{I}$,
	\begin{align}\label{d7}
		\varpi^{2}\tr[a_{\iota}\deri^{2}\phi_{\ell,\iota}]
		&\geq 2\theta[\varpi^{4}|\deri^{3}u^{\varepsilon,\delta}_{\ell,\iota}|^{2}+\mu\varpi^{2}|\deri^{2}u^{\varepsilon,\delta}_{\ell,\iota}|^{2}]-2\lambda C_{8}A^{1}_{\varepsilon,\delta}\varpi^{2}|\deri^{3}u^{\varepsilon,\delta}_{\ell,\iota}|\notag\\
		&\quad-\lambda C_{8}[A^{1}_{\varepsilon,\delta}]^{2}-C_{8}(\lambda+\mu)A^{1}_{\varepsilon,\delta}-C_{8}\mu\notag\\
		&\quad+\varpi^{2}\bigg\{\sum_{\ell'\in\mathbb{M}\setminus\{\ell\}}\psi'_{\delta,\ell,\ell',\iota}(\cdot)[\phi_{\ell,\iota}-\phi_{\ell',\iota} ]+\sum_{\kappa\in\mathbb{I}\setminus\{\iota\}}q_{\ell}(\iota,\kappa)[\phi_{\ell,\iota}-\phi_{\ell,\kappa} ]\bigg\}\notag\\
		&\quad+\varpi^{2}\psi'_{\varepsilon,\ell,\iota}(\cdot)\bigg\{2\varpi A^{1}_{\varepsilon,\delta}[\lambda\theta-2]|\deri^{2}u^{\varepsilon,\delta}_{\ell,\iota}|^{2}-2\lambda A^{1}_{\varepsilon,\delta} C_{8}|\deri^{2}u^{\varepsilon,\delta}_{\ell,\iota}|\notag\\
		&\quad-(\lambda+\mu)C_{8}A^{1}_{\varepsilon,\delta}+2A^{1}_{\varepsilon,\delta}\langle\deri^{1}u^{\varepsilon,\delta}_{\ell,\iota},\deri^{1}\phi_{\ell,\iota}\rangle\bigg\}.
	\end{align}
\end{lema}
\begin{proof}[Proof of Lemma \ref{Lb1}. Eq. \eqref{ap3}]
	Let $\phi_{\ell,\iota}$ be as in \eqref{D2.1}, where $\lambda\geq\max\{1,2/\theta\}$ is fixed and  $\mu\geq1$ will be determined later on, and $(x_{0}, \ell_{0},\iota_{0})\in\overline{\set}\times\mathbb{M}\times\mathbb{I}$ satisfies 
	\begin{equation}\label{D2.0}
		\varpi(x_{0})|\deri^{2}u_{\ell_{0},\iota_{0}}(x_{0})|=A^{1}_{\varepsilon,\delta}=\max_{(x,\ell,\iota)\in\overline{\set}\times\mathbb{M}\times\mathbb{I}}\varpi(x)|\deri^{2}u_{\ell,\iota}(x)|.
	\end{equation}
	Notice that 
	if $x_{0}\in\overline{\set}\setminus B_{\beta' r}$, by  Remark \ref{R1} and \eqref{D2.0}, we obtain $\varpi(x)|\deri^{2}u_{\ell,\iota}(x)|\equiv0$, for each $(x,\ell,\iota)\times\overline{\set}\times\mathbb{M}\times\mathbb{I}$. From here, \eqref{ap3} is trivially true. So, assume that $x_{0}$ is in $B_{\beta'r}$. Without loss of generality we also assume that $A^{1}_{\varepsilon,\delta}>1$, since if $A^{1}_{\varepsilon,\delta}\leq1$, we get that  $\varpi(x) |\deri^{2}u_{\ell,\iota}(x)|\leq A^{1}_{\varepsilon,\delta}\leq1$ for $(x,\ell,\iota)\in\overline{\set}\times\mathbb{M}\times\mathbb{I}$.  Taking $C_{6}=1$, we obtain the result in {\eqref{ap3}}.   Let $(x_{\mu},\ell_{\mu},\iota_{\mu})\in\overline{\set}\times\mathbb{M}\times\mathbb{I}$ be such that $\phi_{\ell_{\mu},\iota_{\mu}}(x_{\mu})=\max_{(x,\ell,\iota)\in\overline{\set}\times\mathbb{M}\times\mathbb{I}}\phi_{\ell,\iota}(x)$. If $x_{\mu}\in\overline{\set}\setminus B_{\beta'r}$,  from \eqref{ap2} and \eqref{D2.1}, it follows that
	\begin{equation}\label{D2.2}
		\varpi^{2}|\deri^{2}u_{\ell,\iota}|^{2}\leq-\lambda A^{1}_{\varepsilon,\delta}\varpi\tr[\alpha_{\iota_{0}}\deri^{2}u_{\ell,\iota}]+\mu C_{5}^{2},\ \text{for}\ (x,\ell,\iota)\in\overline{\set}\times\mathbb{M}\times\mathbb{I}.
	\end{equation}
	Evaluating $(x_{0},l_{0},\iota_{0})$ in \eqref{D2.2} and by \eqref{eq2}, \eqref{a4}, \eqref{p12.1}, \eqref{NPD.1} and \eqref{ap2}, it can be verified that  $[A_{\varepsilon,\delta}^{1}]^{2}\leq\lambda\Lambda[1+C_{2}] A_{\varepsilon,\delta}^{1}+\mu C_{2}^{2}$. From here  and due to $A^{1}_{\varepsilon,\delta}>1$, we conclude that $\varpi(x)|\deri^{2}u_{\ell,\iota}(x)|\leq A_{\varepsilon,\delta}^{1}\leq \lambda\Lambda[1+C_{2}]+\mu C_{2}^{2}=:C_{3},\quad \text{for $(x,\ell,\iota)\in\overline{\set}\times\mathbb{M}\times\mathbb{I}$.}$   From now, assume that $x_{\mu}\in B_{\beta'r}$. Then, 
	\begin{equation}\label{D2.3}
		\begin{split}
			&\deri^{1}\phi_{\ell_{\mu},\iota_{\mu}}(x_{\mu})=0, \quad \tr[a_{\iota_{\mu}}(x_{\mu})\deri^{2}\phi_{\ell_{\mu},\iota_{\mu}}(x_{\mu})]\leq0,\\
			&\phi_{\ell_{\mu},\iota_{\mu}}(x_{\mu})-\phi_{\ell',\kappa}(x_{\mu})\geq0 \quad\text{for}\ (\ell',\kappa)\in\mathbb{M}\times\mathbb{I}.
		\end{split}
	\end{equation}
	Noting that $2\theta\varpi^{4}|\deri^{3}u_{\ell,\iota}|^{2}-2\lambda C_{7}A^{1}_{\varepsilon,\delta}\varpi^{2}|\deri^{3}u_{\ell,\iota}|\geq-\frac{\lambda^{2}C^{2}_{7}}{\theta}[A^{1}_{\varepsilon,\delta}]^{2}$, with $C_{7}>0$ {as} in Lemma \ref{D2.0.0},  and using \eqref{d7} and \eqref{D2.3}, it yields that
	\begin{align*}%\label{d7.0}
		0\geq&2\theta\mu\varpi^{2}|\deri^{2}u_{\ell_{\mu},\iota_{\mu}}|^{2}-\lambda^{2} C_{7}\bigg[1+\frac{C_{7}}{\theta}\bigg][A^{1}_{\varepsilon,\delta}]^{2}-C_{7}(\lambda+\mu)A^{1}_{\varepsilon,\delta}-C_{7}\mu\notag\\
		&+A^{1}_{\varepsilon,\delta}\varpi^{2}\psi'_{\varepsilon,\ell}(\cdot)\{2\varpi[\lambda\theta-2]|\deri^{2}u_{\ell_{\mu},\iota_{\mu}}|^{2}-2\lambda C_{7}|\deri^{2}u_{\ell_{\mu},\iota_{\mu}}|-(\lambda+\mu)C_{7}\},\quad \text{at}\ x_{\mu}.
	\end{align*}
	From here, we have that at least one of the next two inequalities is true:
	\begin{align}
		2\theta\mu\varpi^{2}|\deri^{2}u_{\ell_{\mu},\iota_{\mu}}|^{2}-\lambda^{2} C_{7}\bigg[1+\frac{C_{7}}{\theta}\bigg][A^{1}_{\varepsilon,\delta}]^{2}-C_{7}(\lambda+\mu)A^{1}_{\varepsilon,\delta}-C_{7}\mu&\leq0,\quad\text{at}\ x_{\mu},\label{d7.1}\\
		A^{1}_{\varepsilon,\delta}\varpi^{2}\psi'_{\varepsilon,\ell,\iota}(\cdot)\{2\varpi[\lambda\theta-2]|\deri^{2}u_{\ell_{\mu},\iota_{\mu}}|^{2}-2\lambda C_{7}|\deri^{2}u_{\ell_{\mu},\iota_{\mu}}|-(\lambda+\mu)C_{7}\}&\leq0,\quad\text{at}\ x_{\mu}.\label{d7.2}
	\end{align}
	Suppose that \eqref{d7.1} holds. Then, evaluating $(x_{\mu},\ell_{\mu})$ in \eqref{D2.1}, it follows
	\begin{align}\label{d.8}
		\phi_{\ell_{\mu},\iota_{\mu}}%&\leq\varpi^{2}|\deri^{2}u_{\ell_{\mu},\iota_{\mu}}|^{2}+\lambda\Lambda A^{1}_{\varepsilon,\delta}\varpi|\deri^{2}u_{\ell_{\mu},\iota_{\mu}}|+\mu [C_{2}]^{2}\notag\\
		&\leq \frac{\lambda^{2} C_{7}}{2\theta\mu}\bigg[1+\frac{C_{7}}{\theta}\bigg][A^{1}_{\varepsilon,\delta}]^{2}+\frac{C_{7}(\lambda+\mu)}{2\theta\mu}A^{1}_{\varepsilon,\delta}+\frac{C_{7}}{2\theta}+\mu [C_{2}]^{2}\notag\\
		&\quad+\lambda\Lambda A^{1}_{\varepsilon,\delta}\bigg\{\frac{\lambda^{2} C_{7}}{2\theta\mu}\bigg[1+\frac{C_{7}}{\theta}\bigg][A^{1}_{\varepsilon,\delta}]^{2}+\frac{C_{7}(\lambda+\mu)}{2\theta\mu}A^{1}_{\varepsilon,\delta}+\frac{C_{7}}{2\theta}\bigg\}^{1/2},\quad\text{at}\ x_{\mu}.
	\end{align} 
	Meanwhile, evaluating $(x_{0},\ell_{0},\iota_{0})$ in \eqref{D2.1} and  using \eqref{eq2} and \eqref{NPD.1}, we get 
	\begin{align}\label{d.9}
		\phi_{\ell_{0},\iota_{0}}%=\varpi^{2}|\deri^{2}u_{\ell_{0},\iota_{0}}|^{2}+\lambda A^{1}_{\varepsilon,\delta}\varpi\tr[\alpha_{\iota_{0}}\deri^{2}u_{\ell_{0},\iota_{0}}]+\mu|\deri^{1}u_{\ell_{0},\iota_{0}}|^{2}
		\geq[A^{1}_{\varepsilon,\delta}]^{2}-\lambda \Lambda A^{1}_{\varepsilon,\delta}[C_{2}+1],\quad\text{at}\ x_{0}.
	\end{align}
	Then, taking $\mu$ large enough such that $$\frac{K_{7}^{(\lambda)}}{\mu}\leq\Big[\frac{1}{\lambda\Lambda}\Big[1-\frac{K_{7}^{(\lambda)}}{\mu}\Big]\Big]^{2},$$ with $K_{7}^{(\lambda)}\eqdef\frac{\lambda^{2}C_{7}}{2\theta}\big[1+\frac{C_{7}}{\theta}\big]$, using \eqref{d.8}--\eqref{d.9} and since $\phi_{\ell_{0},\iota_{0}}(x_{0})\leq\phi_{\ell_{\mu},\iota_{\mu}}(x_{\mu})$ and $\lambda, A^{1}_{\varepsilon,\delta}>1$, we have that 
	\begin{align*}
		\frac{1}{\lambda \Lambda}\bigg[1-\frac{K_{7}^{(\lambda)}}{\mu}\bigg]A^{1}_{\varepsilon,\delta}-K_{8}^{(\mu)}
		\leq\bigg\{\frac{ K_{7}^{(\lambda)}}{\mu}[A^{1}_{\varepsilon,\delta}]^{2}+\frac{C_{7}(\lambda+\mu)}{2\theta\mu}A^{1}_{\varepsilon,\delta}+\frac{C_{7}}{2\theta}\bigg\}^{1/2}.
	\end{align*}
	with $K_{8}^{(\mu)}\eqdef\frac{C_{7}}{2\theta\Lambda}\Big[\frac{1}{\mu}+1\Big]+\frac{C_{7}}{2\theta}+\mu C_{2}$. Then,
	%\begin{align*}
	%\frac{1}{\lambda^{2} %\Lambda^{2}}\bigg[1-\frac{K_{7}^{(\lambda)}}{\mu}\bigg]^{2}[A^{1}_{\varepsilon,\delta}]^{2}-&\frac{2 K_{8}^{(\mu)}}{\lambda %\Lambda}\bigg[1-\frac{K_{7}^{(\lambda)}}{\mu}\bigg]A^{1}_{\varepsilon,\delta}+[K_{8}^{(\mu)}]^{2}\notag\\
	%&\leq\frac{ %K_{7}^{(\lambda)}}{\mu}[A^{1}_{\varepsilon,\delta}]^{2}+\frac{C_{3}(\lambda+\mu)}{2\theta\mu}A^{1}_{\varepsilon,\delta}+\frac{C_{3}}{2\theta}.
	%\end{align*}
	%It implies that 
	\begin{equation*}
		\bigg\{\frac{1}{\lambda^{2} \Lambda^{2}}\bigg[1-\frac{K_{7}^{(\lambda)}}{\mu}\bigg]^{2}-
		\frac{ K_{7}^{(\lambda)}}{\mu}\bigg\}[A^{1}_{\varepsilon,\delta}]^{2} 
		\leq\bigg\{\frac{2K_{8}^{(\mu)}}{\lambda \Lambda}\bigg[1-\frac{K_{7}^{(\lambda)}}{\mu}\bigg]+\frac{C_{7}(\lambda+\mu)}{2\theta\mu}\bigg\}A^{1}_{\varepsilon,\delta}+\frac{C_{7}}{2\theta}.
	\end{equation*}
	From here, we conclude there exists a constant $C_{3}=C_{3}(d,\Lambda,\alpha,K_{3})$ such that
	\begin{equation*}
		\varpi(x)|\deri^{2}u_{\ell,\iota}(x)|\leq A^{1}_{\varepsilon,\delta}\leq C_{3}\quad \text{for}\  (x,\ell)\in \overline{\set}\times\mathbb{I}. 
	\end{equation*}
	Now, assume that \eqref{d7.2} holds. Then,  $2\varpi^{2}[\lambda\theta-2]|\deri^{2}u_{\ell_{\mu},\iota_{\mu}}|^{2}\leq2\lambda C_{7}\varpi|\deri^{2}u_{\ell_{\mu},\iota_{\mu}}|+(\lambda+\mu)C_{7}$ at $x_{\mu}$ due to $\psi'_{\varepsilon}\geq0$ and $\varpi\leq1$. From here, we have that $\varpi|\deri^{2}u_{\ell_{\mu},\iota_{\mu}}|\leq K_{9}^{(\lambda,\mu)}$ at $x_{\mu}$, where $K_{9}^{(\lambda,\mu)}$ is a positive constant independent of $A_{\varepsilon,\delta}^{1}$. Therefore, $[A^{1}_{\varepsilon,\delta}]^{2}-\lambda\Lambda A^{1}_{\varepsilon,\delta}[C_{2}+1]\leq\phi_{\ell_{0},\iota_{0}}(x_{0})\leq\phi_{\ell_{\mu},\iota_{\mu}}(x_{\mu})\leq [K_{9}^{(\lambda,\mu)}]^{2}+\lambda\Lambda A^{1}_{\varepsilon,\delta} K_{9}^{(\lambda,\mu)}+\mu [C_{2}]^{2}$. From here, we conclude there exists a constant $C_{3}=C_{3}(d,\Lambda,\alpha,K_{1})$ such that $\varpi|\deri^{2}u_{\ell,\iota}|\leq A^{1}_{\varepsilon,\delta}\leq C_{3}$ for all $(x,\ell)\in \overline{\set}\times\mathbb{I}$.
\end{proof}

\begin{proof}[Proof of Lemma \ref{D2.0.0}]  
	Taking first and second derivatives of $\phi_{\ell,\iota}$ on $\overline{B}_{\beta'r}$, it can be verified that 
	\begin{align*}
	\tr[a_{\iota}\deri^{2}\phi_{\ell,\iota}] &=|\deri^{2}u_{\ell,\iota}|^{2}\tr[a_{\iota}\deri^{2}\varpi^{2}]+2\langle a_{\iota}\deri^{1}\varpi^{2},\deri^{1}|\deri^{2}u_{\ell,\iota}|^{2}\rangle+\varpi^{2}\tr[a_{\iota}\deri^{2}|\deri^{2}u_{\ell,\iota}|^{2}]\notag\\
	&\quad+\lambda A^{1}_{\varepsilon,\delta}\tr[\alpha_{\iota_{0}}\deri^{2}u_{\ell,\iota}]\tr[a_{\iota}\deri^{2}\varpi]+2\lambda A^{1}_{\varepsilon,\delta}\langle a_{\iota}\deri^{1}\varpi,\deri^{1}\tr[\alpha_{\iota_{0}}\deri^{2}u_{\ell,\iota}]\rangle\notag\\
	&\quad+\lambda A^{1}_{\varepsilon,\delta}\varpi\sum_{ji}\alpha_{\iota_{0}\,ji}\tr[a_{\iota}\deri^{2}\partial_{ji}u_{\ell,\iota}]+\mu\tr[a_{\iota}\deri^{2}|\deri^{1}u_{\ell,\iota}|^{2}].
	\end{align*}	
	From here and noticing that from \eqref{H2},
	\begin{align*}
	\tr[a_{\iota}\deri^{2}|\deri^{1}u_{\ell,\iota}|^{2}]%&=2\sum_{i}[\langle a_{\iota}\deri^{1}\partial_{i}u_{\ell,\iota},\deri^{1}\partial_{i}u_{\ell,\iota}\rangle+\partial_{i}u_{\ell,\iota}\tr[a_{\iota}\deri^{2}\partial_{i}u_{\ell,\iota}]]\notag\\
	&\geq 2\theta|\deri^{2}u_{\ell,\iota}|^{2}+2\sum_{i}\partial_{i}u_{\ell,\iota}\tr[a_{\iota}\deri^{2}\partial_{i}u_{\ell,\iota}],\\
	\tr[a_{\iota}\deri^{2}|\deri^{2}u_{\ell,\iota}|^{2}]%&=2\sum_{ji}[\langle a_{\iota}\deri^{1}\partial_{ji}u_{\ell,\iota},\deri^{1}\partial_{ji}u_{\ell,\iota}\rangle+\partial_{ji}u_{\ell,\iota}\tr[a_{\iota}\deri^{2}\partial_{ji}u_{\ell,\iota}]]\notag\\
	&\geq 2\theta|\deri^{3}u_{\ell,\iota}|^{2}+2\sum_{ji}\partial_{ji}u_{\ell,\iota}\tr[a_{\iota}\deri^{2}\partial_{ji}u_{\ell,\iota}],%\label{de2}
	\end{align*}
	it follows that
	\begin{align}\label{d1.2}
	\tr[a_{\iota}\deri^{2}\phi_{\ell,\iota}]&\geq 2\theta[\varpi^{2}|\deri^{3}u_{\ell,\iota}|^{2}+\mu|\deri^{2}u_{\ell,\iota}|^{2}]+|\deri^{2}u_{\ell,\iota}|^{2}\tr[a_{\iota}\deri^{2}\varpi^{2}]\notag\\
	&\quad+2\langle a_{\iota}\deri^{1}\varpi^{2},\deri^{1}|\deri^{2}u_{\ell,\iota}|^{2}\rangle+\lambda A^{1}_{\varepsilon,\delta}\tr[\alpha_{\iota_{0}}\deri^{2}u_{\ell,\iota}]\tr[a_{\iota}\deri^{2}\varpi]\notag\\
	&\quad+2\lambda A^{1}_{\varepsilon,\delta}\langle a_{\iota}\deri^{1}\varpi,\deri^{1}\tr[\alpha_{\iota_{0}}\deri^{2}u_{\ell,\iota}]\rangle+2\mu\sum_{i}\tr[a_{\iota}\deri^{2}\partial_{i}u_{\ell,\iota}]\partial_{i}u_{\ell,\iota}\notag\\
	&\quad+\sum_{ji}[2\varpi^{2}\partial_{ji}u_{\ell,\iota}+\lambda A^{1}_{\varepsilon,\delta}\varpi\alpha_{\iota_{0}\,ji}]\tr[a_{\iota}\deri^{2}\partial_{ji}u_{\ell,\iota}].
	\end{align}	
	Meanwhile, differentiating twice in \eqref{NPD.1}, we see that 
	\begin{align}\label{d2}
	\tr[&a_{\iota}\deri^{2}\partial_{ji}u_{\ell,\iota}]\notag\\
	&=\psi''_{\varepsilon,\ell,\iota}(\cdot)\bar{\eta}^{(i)}_{\ell,\iota}\bar{\eta}^{(j)}_{\ell,\iota}+\psi'_{\varepsilon,\ell,\iota}(\cdot)\partial_{ji}[|\deri^{1}u_{\ell,\iota}|^{2}-g_{\iota}^{2}]+\sum_{\ell'\in\mathbb{M}\setminus\{\ell\}}\psi''_{\delta,\ell,\ell',\iota}(\cdot)\bar{\eta}^{(i)}_{\ell,\ell',\iota}\bar{\eta}^{(j)}_{\ell,\ell',\iota}\notag\\
	&\quad+\sum_{\ell'\in\mathbb{I}\setminus\{\ell\}}\psi'_{\delta,\ell,\ell',\iota}(\cdot)\partial_{ji}[u_{\ell,\iota}-u_{\ell',\iota} ]{-\sum_{\kappa\in\mathbb{I}\setminus\{\iota\}}q_{\ell}(\iota,\kappa)\partial_{ji}[u_{\ell,\kappa}-u_{\ell,\iota}]}-\tr[[\partial_{j}a_{\iota}]\deri^{2}\partial_{i}u_{\ell,\iota}]\notag\\
	&\quad-\tr[[\partial_{ji}a_{\iota}]\deri^{2}u_{\ell,\iota}]-\tr[[\partial_{i}a_{\iota}]\deri^{2}\partial_{j}u_{\ell,\iota}]-\partial_{ji}[h_{\iota}-\langle b_{\iota},\deri^{1}u_{\ell,\iota}\rangle-c_{\iota}u_{\ell,\iota}] 
	\end{align}  
	where $\bar{\eta}_{\ell,\iota}=(\bar{\eta}_{\ell,\iota}^{(1)},\dots,\bar{\eta}_{\ell,\iota}^{(d)})$ and $\bar{\eta}_{\ell,\ell',\iota}=(\bar{\eta}_{\ell,\ell',\iota}^{(1)},\dots,\bar{\eta}_{\ell,\ell',\iota}^{(d)})$ with  $\bar{\eta}_{\ell,\iota}^{(i)}\eqdef\partial_{i}[|\deri^{1}u_{\ell,\iota}|^{2}-g_{\iota}^{2}]$ {and} $\bar{\eta}_{\ell,\ell',\iota}^{(i)}\eqdef\partial_{i}[u_{\ell,\iota}-u_{\ell',\iota} ].$  From \eqref{dercot2} and \eqref{d1.2}--\eqref{d2}, it follows that 
	\begin{align}\label{d3}
	\varpi^{2}&\tr[a_{\iota}\deri^{2}\phi_{\ell,\iota}]\notag\\
	&\geq 2\theta[\varpi^{4}|\deri^{3}u_{\ell,\iota}|^{2}+\mu\varpi^{2}|\deri^{2}u_{\ell,\iota}|^{2}]+\widetilde{D}_{3}+\widetilde{D}_{4}\notag\\
	&\quad+\varpi^{2}\bigg\{\psi''_{\varepsilon,\ell,\iota}(\cdot)\langle[2\varpi^{2}\deri^{2}u_{\ell,\iota}+\lambda A^{1}_{\varepsilon,\delta}\varpi\alpha_{\iota_{0}}]\bar{\eta}_{\ell,\iota},\bar{\eta}_{\ell,\iota}\rangle\notag\\
	&\quad+\sum_{\ell'\in\mathbb{M}\setminus\{\ell\}}\psi''_{\delta,\ell,\ell',\iota}(\cdot)\langle[2\varpi^{2}\deri^{2}u_{\ell,\iota}+\lambda A^{1}_{\varepsilon,\delta}\varpi\alpha_{\iota_{0}}]\bar{\eta}_{\ell,\ell',\iota},\bar{\eta}_{\ell,\ell',\iota}\rangle\bigg\}\notag\\
	&\quad+\varpi^{2}\psi'_{\varepsilon,\ell,\iota}(\cdot)\widetilde{D}_{\ell,\iota}+\varpi^{2}\sum_{\ell'\in\mathbb{M}\setminus\{\ell\}}\psi'_{\delta,\ell,\ell',\iota}(\cdot)\widetilde{D}_{\ell,\ell'}^{\iota,\iota}{+\varpi^{2}\sum_{\kappa\in\mathbb{I}\setminus\{\iota\}}q_{\ell}(\iota,\kappa)D_{\ell,\ell}^{\iota,\kappa}}
	\end{align}
	where
	
	\begin{align*}
	\widetilde{D}_{3}&\eqdef2\varpi^{2}\langle a_{\iota}\deri^{1}\varpi^{2},\deri^{1}|\deri^{2}u_{\ell,\iota}|^{2}\rangle+2\lambda A^{1}_{\varepsilon,\delta}\varpi^{2}\langle a_{\iota}\deri^{1}\varpi,\deri^{1}\tr[\alpha_{\iota_{0}}\deri^{2}u_{\ell,\iota}]\rangle\notag\\
	&\quad-\sum_{ij}[2\varpi^{4}\partial_{ij}u_{\ell,\iota}+\lambda A^{1}_{\varepsilon,\delta}\varpi^{3}\alpha_{\iota_{0}\,ij}][2\tr[\partial_{j}a_{\iota}\deri^{2}\partial_{i}u_{\ell,\iota}]-\partial_{ij}\langle b_{\iota},\deri^{1}u_{\ell,\iota}\rangle],\\
	\widetilde{D}_{4}&\eqdef\varpi^{2}|\deri^{2}u_{\ell,\iota}|^{2}\tr[a_{\iota}\deri^{2}\varpi^{2}]\notag-\mu\varpi^{2}\widetilde{D}_{2}u_{\ell,\iota}+\lambda A^{1}_{\varepsilon,\delta}\varpi^{2}\tr[\alpha_{\iota_{0}}\deri^{2}u_{\ell,\iota}]\tr[a_{\iota}\deri^{2}\varpi]\notag\\
	&\quad-\sum_{ji}[2\varpi^{4}\partial_{ji}u_{\ell,\iota}+\lambda A^{1}_{\varepsilon,\delta}\varpi^{3}\alpha_{\iota_{0}\,ji}]\left\{\tr[[\partial_{ji}a_{\iota}]\deri^{2}u_{\ell,\iota}]+\partial_{ji}[h_{\iota}-c_{\iota}u_{\ell,\iota}]\right\},\\
	\widetilde{D}_{\ell,\iota}&\eqdef 2\mu\langle\deri^{1}u_{\ell,\iota},\bar{\eta}_{\ell,\iota}\rangle+\tr[[2\varpi^{2}\deri^{2}u_{\ell,\iota}+\lambda A^{1}_{\varepsilon,\delta}\varpi\alpha_{\iota_{0}}]\deri^{2}[|\deri^{1}u_{\ell,\iota}|^{2}-g_{\iota}^{2}]],\\
	%\widetilde{D}_{6}^{(\ell,\ell',\iota)}&\eqdef 2\mu\langle\deri^{1}u_{\ell,\iota},\bar{\eta}_{\ell,\ell',\iota}\rangle+\tr[[2\varpi^{2}\deri^{2}u_{\ell,\iota}+\lambda A^{1}_{\varepsilon,\delta}\varpi\alpha_{\iota_{0}}]\deri^{2}[u_{\ell,\iota}-u_{\ell',\iota} ]],\\
	\widetilde{D}_{\ell,\ell'}^{\iota,\kappa}&\eqdef 2\mu\langle\deri^{1}u_{\ell,\iota},\deri^{1}[u_{\ell,\iota}-u_{\ell',\kappa}]\rangle+\tr[[2\varpi^{2}\deri^{2}u_{\ell,\iota}+\lambda A^{1}_{\varepsilon,\delta}\varpi\alpha_{\iota_{0}}]\deri^{2}[u_{\ell,\iota}-u_{\ell',\kappa} ]].
	\end{align*}
	Recall that $\widetilde{D}_{2}u_{\ell,\iota}$ is given in \eqref{dercot2.1}. To obtain the next inequalities, we shall recurrently use \eqref{a4}, \eqref{h2}, Remark \ref{R1},  \eqref{ap1}, \eqref{ap2} and $\lambda,\mu\geq1$. Then,
	\begin{align}
	\widetilde{D}_{3}%&=8\varpi^{3}\sum_{ijkm}a_{\ell\,ij}\partial_{j}\varpi\partial_{km}u_{\ell,\iota}\partial_{kmi}u_{\ell,\iota}+2\lambda A^{1}_{\varepsilon,\delta}\varpi^{2}\sum_{ijkm} a_{\ell\,im}\partial_{i}\varpi\alpha_{\ell_{0}\,km}\partial_{kmj}u_{\ell,\iota}\notag\\
	%&\quad-2\sum_{ij}[2\varpi^{4}\partial_{ij}u_{\ell,\iota}+\lambda A^{1}_{\varepsilon,\delta}\varpi^{3}\alpha_{\iota_{0}\,ij}]\left\{\sum_{km}\partial_{j}a_{\ell\,km}\partial_{kmi}u_{\ell,\iota}\right.\notag\\
	%&\quad\left.+\sum_{k}[\partial_{k}u_{\ell,\iota}\partial_{ij}b_{\ell\,k}+\partial_{j}b_{\ell\,k}\partial_{ki}u_{\ell,\iota}+\partial_{i}b_{\ell\,k}\partial_{kj}u_{\ell,\iota}+b_{\ell\,k}\partial_{ijk}u_{\ell,\iota}]\right\}\notag\\
	%&\geq-8\Lambda K_{1}d^{4}A^{1}_{\varepsilon,\delta}\varpi^{2}|\deri^{3}u_{\ell,\iota}|-2\lambda\Lambda^{2} d^{4}A^{1}_{\varepsilon,\delta}\varpi^{2}|\deri^{3}u_{\ell,\iota}|-2d^{3}[2+\lambda\Lambda][d+\Lambda]A^{1}_{\varepsilon,\delta}\varpi^{2}|\deri^{3}u_{\ell,\iota}|\notag\\
	%&\quad-2d^{2}A^{1}_{\varepsilon,\delta}[2+\lambda\Lambda]dC_{2}\Lambda-4d^{3}\Lambda[A^{1}_{\varepsilon,\delta}]^{2}[2+\lambda\Lambda]\notag\\
	%&\geq-2\left\{4\Lambda K_{1}d^{4}+\lambda\Lambda^{2} d^{4}+d^{3}[2+\lambda\Lambda][d+\Lambda]\right\}A^{1}_{\varepsilon,\delta}\varpi^{2}|\deri^{3}u_{\ell,\iota}|\notag\\
	%&\quad-2d^{2}A^{1}_{\varepsilon,\delta}[2+\lambda\Lambda]dC_{2}\Lambda-4d^{3}\Lambda[A^{1}_{\varepsilon,\delta}]^{2}[2+\lambda\Lambda]\notag\\
	&\geq-2\left\{4\Lambda K_{1}d^{4}+\Lambda^{2} d^{4}+d^{3}[2+\Lambda][d+\Lambda]\right\}\lambda A^{1}_{\varepsilon,\delta}\varpi^{2}|\deri^{3}u_{\ell,\iota}|\notag\\
	&\quad-2d^{2}\lambda A^{1}_{\varepsilon,\delta}[2+\Lambda]dC_{2}\Lambda-4d^{3}\Lambda\lambda[A^{1}_{\varepsilon,\delta}]^{2}[2+\Lambda],
	\intertext{and by \eqref{H2},}
	\widetilde{D}_{4}%&\geq-2[A^{1}_{\varepsilon,\delta}]^{2}\Lambda d^{2}K_{1}-2\mu\{[2C_{2}\Lambda d^{3}+2d^{1/2}\Lambda C_{2}]A^{1}_{\varepsilon,\delta}+2C_{2}\Lambda d^{2}+2C_{1}C_{2}d^{1/2}\Lambda-2C_{2}\Lambda d^{1/2}\}\notag\\
	%&\quad-\{\lambda d^{4}\Lambda^{2}K_{1}+d^{2}\Lambda[2+\lambda\Lambda][d^{2}+1]\}[A^{1}_{\varepsilon,\delta}]^{2}-d^{2}\Lambda[2+\lambda\Lambda][2C_{2}+C_{1}]A^{1}_{\varepsilon,\delta}\notag\\
	%&\geq-\{2\Lambda d^{2}K_{1}+\lambda d^{4}\Lambda^{2}K_{1}+d^{2}\Lambda[2+\lambda\Lambda][d^{2}+1]\}[A^{1}_{\varepsilon,\delta}]^{2}\notag\\
	%&\quad-\{2\mu[2C_{2}\Lambda d^{3}+2d^{1/2}\Lambda C_{2}]+d^{2}\Lambda[2+\lambda\Lambda][2C_{2}+C_{1}]\}A^{1}_{\varepsilon,\delta}\notag\\
	%&\quad-2\mu\{2C_{2}\Lambda d^{2}+2C_{1}C_{2}d^{1/2}\Lambda-2C_{2}\Lambda d^{1/2}\}\notag\\
	&\geq-\{2\Lambda d^{2}K_{1}+ d^{4}\Lambda^{2}K_{1}+d^{2}\Lambda[2+\Lambda][d^{2}+1]\}\lambda[A^{1}_{\varepsilon,\delta}]^{2}\notag\\
	&\quad-\{2\mu[2C_{2}\Lambda d^{3}+2d^{1/2}\Lambda C_{2}]+d^{2}\lambda\Lambda[2+\Lambda][2C_{2}+C_{1}]\}A^{1}_{\varepsilon,\delta}\notag\\
	&\quad-2\mu\{2C_{2}\Lambda d^{2}+2C_{1}C_{2}d^{1/2}\Lambda-2C_{2}\Lambda d^{1/2}\}.
	\end{align}	 
	On the other hand, since $\lambda\geq\frac{2}{\theta}$ and using \eqref{H2}, we have that
	\begin{align}\label{d4}
	\langle[2\varpi^{2}\deri^{2}u_{\ell,\iota}+\lambda A^{1}_{\varepsilon,\delta}\varpi\alpha_{\iota_{0}}]\gamma,\gamma\rangle&\geq\varpi[\lambda A^{1}_{\varepsilon,\delta}\theta-2\varpi|\deri^{2}u_{\ell,\iota}|]\,|\gamma|^{2}\notag\\
	& \geq\varpi A^{1}_{\varepsilon,\delta}[\lambda\theta-2]\,|\gamma|^{2}\geq0,
	\end{align}
	for $\gamma\in\R^{d}$. From here and since $\psi''_{\varepsilon,\ell,\iota}(\cdot)\geq0$ and $\psi''_{\delta,\ell,\ell',\iota}(\cdot)\geq0$, it follows that 
	\begin{multline}\label{d5}
	\psi''_{\varepsilon,\ell,\iota}(\cdot)\langle[2\varpi^{2}\deri^{2}u_{\ell,\iota}+\lambda A^{1}_{\varepsilon,\delta}\varpi\alpha_{\iota_{0}}]\bar{\eta}_{\ell,\iota},\bar{\eta}_{\ell,\iota}\rangle\\
	+\sum_{\ell'\in\mathbb{M}\setminus\{\ell\}}\psi''_{\delta,\ell,\ell',\iota}(\cdot)\langle[2\varpi^{2}\deri^{2}u_{\ell,\iota}+\lambda A^{1}_{\varepsilon,\delta}\varpi\alpha_{\iota_{0}}]\bar{\eta}_{\ell,\ell',\iota},\bar{\eta}_{\ell,\ell',\iota}\rangle\geq0.
	\end{multline}
	It is easy to verify that
	\begin{multline}\label{d8}
	\varpi^{2}\langle\deri^{1}u_{\ell,\iota},\deri^{1}|\deri^{2}u_{\ell,\iota}|^{2}\rangle+\lambda A^{1}_{\varepsilon,\delta}\varpi\langle\deri^{1}u_{\ell,\iota},\deri^{1}\tr[\alpha_{\iota_{0}}\deri^{2}u_{\ell,\iota}]\rangle+\mu\langle\deri^{1}u_{\ell,\iota},\deri^{1}|\deri^{1}u_{\ell,\iota}|^{2}\rangle\\
	=\langle\deri^{1}u_{\ell,\iota},\deri^{1}\phi_{\ell,\iota}\rangle-\langle\deri^{1}u_{\ell,\iota},\deri^{1}\varpi^{2}\rangle|\deri^{2}u_{\ell,\iota}|^{2}-\lambda A^{1}_{\varepsilon,\delta}\tr[\alpha_{\iota_{0}}\deri^{2}u_{\ell,\iota}]\langle\deri^{1}u_{\ell,\iota},\deri^{1}\varpi\rangle
	\end{multline}
	due to 
	\begin{align*}
	\partial_{i}\phi_{\ell,\iota}&=|\deri^{2}u_{\ell,\iota}|^{2}\partial_{i}\varpi^{2}+\varpi^{2}\partial_{i}|\deri^{2}u_{\ell,\iota}|^{2}\\
	&\quad+\lambda A^{1}_{\varepsilon,\delta}\tr[\alpha_{\iota_{0}}\deri^{2}u_{\ell,\iota}]\partial_{i}\varpi+\lambda A^{1}_{\varepsilon,\delta}\varpi\tr[\alpha_{\iota_{0}}\deri^{2}\partial_{i}u_{\ell,\iota}]+\mu\partial_{i}|\deri^{1}u_{\ell,\iota}|^{2}\quad\text{on}\ B_{\beta' r}.
	\end{align*}
	Then,  by \eqref{d4} and \eqref{d8},
	\begin{align}%\label{d6}
	\widetilde{D}_{\ell,\iota}
	&\geq2\varpi A^{1}_{\varepsilon,\delta}[\lambda\theta-2]|\deri^{2}u_{\ell,\iota}|^{2}+2\langle\deri^{1}u_{\ell,\iota},\deri^{1}\phi_{\ell,\iota}\rangle\notag\\
	&\qquad-4\lambda d^{1/2}C_{2}K_{1}A^{1}_{\varepsilon,\delta}|\deri^{2}u_{\ell,\iota}|-2\lambda\Lambda d^{5/2}C_{2}K_{1}A^{1}_{\varepsilon,\delta}|\deri^{2}u_{\ell,\iota}|\notag\\
	&\qquad-4\mu  \Lambda^{2}d^{1/2}A^{1}_{\varepsilon,\delta}C_{2}-2\lambda d\Lambda^{2}A^{1}_{\varepsilon,\delta}-\lambda\Lambda^{3}d^{2}A^{1}_{\varepsilon,\delta}\notag\\
	&\qquad-4d^{2}\lambda\Lambda^{2}A^{1}_{\varepsilon,\delta}-2\Lambda^{3}d^{2}\lambda A^{1}_{\varepsilon,\delta}.
	\end{align}
	Using the following properties $|A|^{2}-2\tr[AB]+|B|^{2}=\sum_{ij}(A_{ij}-B_{ij})^{2}\geq0$ and $|y_{1}|^{2}-2\langle y_{1},y_{2}\rangle+|y|^{2}=\sum_{i}(y_{1,i}-y_{2,i})^{2}\geq0$ 
	where $A=(A_{ij})_{d\times d},B=(B_{ij})_{d\times d}$ and $y_{1}=(y_{1,1},\dots,y_{1,d}),y_{2}=(y_{2,1},\dots,y_{2,d})$ belong $\mathcal{S}(d)$ and $\R^{d}$, respectively, and by definition of $\phi_{\ell,\iota}$,  it is easy to corroborate the following identity
	\begin{equation}\label{d6}
	\widetilde{D}_{\ell,\ell'}^{\iota,\kappa}\geq \phi_{\ell,\iota}-\phi_{\ell',\kappa} ,\ \text{for}\ (\ell',\kappa)\in\mathbb{M}\times\mathbb{I}.
	\end{equation}
	Applying \eqref{d5}--\eqref{d6} in \eqref{d3} and considering that all  constants that appear in those inequalities (i.e. \eqref{d5}--\eqref{d6}) are bounded by an universal constant $C_{8}=C_{8}(d,\Lambda,\alpha,K_{2})$, we obtain the desired result in the lemma above. With this remark, the proof is concluded.
\end{proof}

\subsection{Proof of Proposition \ref{princ1.1}}\label{proof4}

\begin{proof}[Proof of Proposition \ref{princ1.1}. Existence]  Taking $\ell'\in\mathbb{M}\setminus\{\ell\}$, and using \eqref{eq2}, \eqref{NPD.1} and Lemma \ref{Lb1}, we have that $\psi_{\delta}(u^{\varepsilon,\delta}_{\ell,\iota}-u^{\varepsilon,\delta}_{\ell',\iota} -\vartheta_{\ell,\ell'})$ is locally bounded, uniformly in $\delta$. From here and \eqref{conv1}, it yields that  $u^{\varepsilon}_{\ell,\iota}-u^{\varepsilon}_{\ell',\iota} -\vartheta_{\ell,\ell'}\leq0$ in $\set$. Then, 
	\begin{equation}\label{ineq2}
		u^{\varepsilon}_{\ell,\iota}-\mathcal{M}_{\ell,\iota}u^{\varepsilon}\leq0,\quad \text{in}\ \set. 
	\end{equation}
	Note that the previous inequality is true on the boundary set $\partial\set$, since $u^{\varepsilon,\delta}_{\ell,\iota}=u^{\varepsilon,\delta}_{\ell',\iota}=f_{\iota}$ on $\partial\set$ and $\vartheta_{\ell,\ell'}\geq0$. Recall that the operator $\mathcal{M}_{\ell,\iota}$ is defined in \eqref{p6.0}. On the other hand, since $u^{\varepsilon,\delta_{\hat{n}}}_{\ell,\iota}$ is the unique solution to \eqref{NPD.1},  when $\delta=\delta_{\hat{n}}$, it follows that 
	\begin{equation}\label{ineq1}
		\int_{B_{r}}\Big\{[c_{\iota}-\dif_{\ell,\iota}]  u_{\ell,\iota}^{\varepsilon,\delta_{\hat{n}}}+ \psi_{\varepsilon}(|\deri^{1} u_{\ell,\iota}^{\varepsilon,\delta_{\hat{n}}}|^{2}- g_{\iota}^{2})\Big\}\varpi\der x\leq  \int_{B_{r}}h_{\iota}\varpi\der x  ,\quad\text{for}\ \varpi\in\mathcal{B}(B_{r}),  
	\end{equation}
	where
	\begin{equation}\label{c1}
		\mathcal{B}(A)\eqdef\{\varpi\in \hol^{\infty}_{\comp}(A): \varpi\geq0\ \text{and}\ \sop[\varpi]\subset A\subset\set \}.
	\end{equation}
	By \eqref{conv1} and  letting $\delta_{\hat{n}}\rightarrow0$ in \eqref{ineq1}, we obtain that 
	\begin{equation}\label{ineq3}
		[c_{\iota}-\dif_{\ell,\iota}]  u_{\ell,\iota}^{\varepsilon}+ \psi_{\varepsilon}(|\deri^{1} u_{\ell,\iota}^{\varepsilon}|^{2}- g_{\iota}^{2})\leq h_{\iota}\quad  \text{a.e. in $\set$}. 
	\end{equation}
	From \eqref{ineq2} and \eqref{ineq3},  $\max\left\{[c_{\iota}-\dif_{\ell,\iota}]  u_{\ell,\iota}^{\varepsilon}+\psi_{\varepsilon}(|\deri^{1} u_{\ell,\iota}^{\varepsilon}|^{2}- g_{\iota}^{2})-h_{\iota},u^{\varepsilon}_{\ell,\iota}-\mathcal{M}_{\ell,\iota}u^{\varepsilon}\right\}\leq 0$  a.e. in $\set$. We shall prove that if 
	\begin{equation}\label{ineq5}
		u^{\varepsilon}_{\ell,\iota}(x^{*})-\mathcal{M}_{\ell,\iota}u^{\varepsilon}(x^{*})<0,\quad \text{for some}\ x^{*}\in\set,
	\end{equation}
	then, there exists a  neighborhood $\mathcal{N}_{x^{*}}\subset\set$  of  $x^{*}$ such that
	\begin{equation}\label{ineq6}
		[c_{\iota}-\dif_{\ell,\iota}]  u_{\ell,\iota}^{\varepsilon}+\psi_{\varepsilon}(|\deri^{1} u_{\ell,\iota}^{\varepsilon}|^{2}- g_{\iota}^{2})= h_{\iota},\quad  \text{a.e. in $\mathcal{N}_{x^{*}}$.}
	\end{equation}
	{Assume} \eqref{ineq5} holds. Then, taking  $\ell'\in\mathbb{M}\setminus\{\ell\}$, we see that $u^{\varepsilon}_{\ell,\iota}-u^{\varepsilon}_{\ell',\iota} -\vartheta_{\ell,\ell'}\leq u^{\varepsilon}_{\ell,\iota}-\mathcal{M}_{\ell,\iota}u^{\varepsilon}<0$ at $x^{*}$. Since  $u^{\varepsilon}_{\ell,\iota}-u^{\varepsilon}_{\ell',\iota} $ is a continuous function, there exists a ball $B_{r_{\ell'}}\subset\set$ such that  $x^{*}\in B_{r_{\ell'}}$ and $u^{\varepsilon}_{\ell,\iota}-u^{\varepsilon}_{\ell',\iota} -\vartheta_{\ell,\ell'}<0$ in $B_{r_{\ell'}}$. From here and defining $\mathcal{N}_{x^{*}}$ as $\bigcap_{\ell'\in\mathbb{M}\setminus\{\ell\}}B_{r_{\ell'}}$, %and $\hat{\vartheta}=\max_{\ell'\in\mathbb{M}\setminus\{\ell\}}\{\vartheta_{\ell,\ell'}\}$, 
	we have that $\mathcal{N}_{x^{*}}\subset\set$ is a neighborhood of $x^{*}$ and 
	\begin{equation}\label{ineq7}
		%u^{\varepsilon}_{\ell,\iota}-u^{\varepsilon}_{\ell',\iota}-\hat{\vartheta}\leq 
		u^{\varepsilon}_{\ell,\iota}-u^{\varepsilon}_{\ell',\iota} -\vartheta_{\ell,\ell'}<0,\quad \text{ in}\ \mathcal{N}_{x^{*}},\ \text{for}\ \ell'\in\mathbb{M}\setminus\{\ell\}.
	\end{equation}
	Meanwhile, observe that
	\begin{equation}\label{ineq8}
		||u^{\varepsilon,\delta_{\hat{n}}}_{\ell,\iota}-u^{\varepsilon,\delta_{\hat{n}}}_{\ell',\iota} -(u^{\varepsilon}_{\ell,\iota}-u^{\varepsilon}_{\ell',\iota} )||_{\hol(\set)}\underset{\delta_{\hat{n}}\rightarrow0}{\longrightarrow}0,\quad\text{for}\ \ell'\in\mathbb{M}\setminus\{\ell\},
	\end{equation}
	since \eqref{conv1} holds. Then, by \eqref{ineq7}--\eqref{ineq8}, it yields that for each $\ell'\in\mathbb{M}\setminus\{\ell\}$, there exists a $\delta^{(\ell')}\in(0,1)$ such that if $\delta_{\hat{n}}\leq\delta^{(\ell')}$, 
	$u^{\varepsilon,\delta_{\hat{n}}}_{\ell,\iota}-u^{\varepsilon,\delta_{\hat{n}}}_{\ell',\iota} -\vartheta_{\ell,\ell'}<0$ in $\mathcal{N}_{x^{*}}$. Taking $\delta'\eqdef\min_{\ell'\in\mathbb{M}\setminus\{\ell\}}\{\delta^{(\ell')}\}$, it follows that $u^{\varepsilon,\delta_{\hat{n}}}_{\ell,\iota}-u^{\varepsilon,\delta_{\hat{n}}}_{\ell',\iota} -\vartheta_{\ell,\ell'}<0$ in $\mathcal{N}_{x^{*}}$, for all $\delta_{\hat{n}}\leq\delta'$ and $\ell'\in\mathbb{M}\setminus\{\ell\}$. From here and since for each $\delta_{\hat{n}}\leq\delta'$, $u^{\varepsilon,\delta_{\hat{n}}}_{\ell,\iota}$ is the unique solution  to \eqref{NPD.1}, when $\delta=\delta_{\hat{n}}$, it implies that 
	\begin{equation*}
		\int_{\mathcal{N}_{x^{*}}}\Big\{[c_{\iota}-\dif_{\ell,\iota}]  u_{\ell,\iota}^{\varepsilon,\delta_{\hat{n}}}+ \psi_{\varepsilon}(|\deri^{1} u_{\ell,\iota}^{\varepsilon,\delta_{\hat{n}}}|^{2}- g_{\iota}^{2})\Big\}\varpi\der x=  \int_{\mathcal{N}_{x^{*}}}h_{\iota}\varpi\der x,   \quad\text{for}\ \varpi\in\mathcal{B}(\mathcal{N}_{x^{*}}).
	\end{equation*}
	Therefore,  \eqref{ineq6} holds. Hence, we get that for each $\varepsilon\in(0,1)$, $u^{\varepsilon}=(u^{\varepsilon}_{\ell,\iota})_{(\ell,\iota)\in\mathbb{M}\times\mathbb{I}}$ is a solution to the HJB equation \eqref{p13.0.1.0}.
\end{proof}

\begin{proof}[Proof of Proposition \ref{princ1.1}. Uniqueness] Let $\varepsilon\in(0,1)$ be fixed. Suppose that  $u^{\varepsilon}=(u^{\varepsilon}_{\ell,\iota})_{(\ell,\iota)\in\mathbb{M}\times\mathbb{I}}$ and $v^{\varepsilon}=(v^{\varepsilon}_{\ell,\iota})_{(\ell,\iota)\in\mathbb{M}\times\mathbb{I}}$ are two solutions to the HJB equation \eqref{pc1} whose components belong to  $\hol^{0}(\overline{\set})\cap\sob^{2,\infty}_{\loc}(\set)$. Take $(x_{0},\ell_{0},\iota_{0})\in\overline{\set}\times\mathbb{M}\times\mathbb{I}$ such that
	\begin{equation}\label{ineq9}
		u^{\varepsilon}_{\ell_{0},\iota_{0}}(x_{0})-v^{\varepsilon}_{\ell_{0},\iota_{0}}(x_{0})=\max_{(x,\ell,\kappa)\in\overline{\set}\times\mathbb{M}\times\mathbb{I}}\{u^{\varepsilon}_{\ell,\kappa}(x)-v_{\ell,\kappa}^{\varepsilon}(x)\}.
	\end{equation}
	Notice that by \eqref{ineq9}, we only need to verify that 
	\begin{equation}\label{ineq10}
		u^{\varepsilon}_{\ell_{0},\iota_{0}}(x_{0})-v^{\varepsilon}_{\ell_{0},\iota_{0}}(x_{0})\leq0,
	\end{equation}
	which is trivially true,  if $x_{0}\in\partial\set$, since $u^{\varepsilon}_{\ell_{0},\iota_{0}}-v^{\varepsilon}_{\ell_{0},\iota_{0}}=0$ on $\partial\set$. Let us assume $x_{0}\in\set$. We shall verify \eqref{ineq10} by contradiction. Suppose that $u^{\varepsilon}_{\ell_{0},\iota_{0}}-v^{\varepsilon}_{\ell_{0},\iota_{0}}>0$ at $x_{0}$. Then, by continuity of $u^{\varepsilon}_{\ell_{0},\iota_{0}}-v^{\varepsilon}_{\ell_{0},\iota_{0}}$, there exists a  ball $B_{r_{1}}(x_{0})\subset\set$ such that 
	\begin{equation}\label{ineq10.1}
		c_{\iota_{0}}[u^{\varepsilon}_{\ell_{0},\iota_{0}}-v^{\varepsilon}_{\ell_{0},\iota_{0}}]\geq\min_{x\in B_{r_{1}}(x_{0})}\{c_{\iota_{0}}(x)[u^{\varepsilon}_{\ell_{0},\iota_{0}}(x)-v^{\varepsilon}_{\ell_{0},\iota_{0}}(x)]\}>0,\quad  \text{in}\ B_{r_{1}}(x_{0}).
	\end{equation}
	The last inequality is true because of $c_{\iota_{0}}>0$ in $\set$. {Additionally, again by \eqref{ineq9} and by the continuity of $u^{\varepsilon}_{\ell',\kappa}-v^{\varepsilon}_{\ell',\kappa}$ on $\overline{\set}$, we get that there is a ball $B_{r_{2}}(x_{0})\subset\set$ such that  
		\begin{equation}\label{w3}	
			\sum_{\kappa\in\mathbb{I}\setminus\{\iota_{0}\}}q_{\ell_{0}}(\iota_{0},\kappa)\{u^{\varepsilon}_{\ell_{0},\iota_{0}}-v^{\varepsilon}_{\ell_{0},\iota_{0}}-[u^{\varepsilon}_{\ell_{0},\kappa}-v^{\varepsilon}_{\ell_{0},\kappa}]\}\geq0\quad \text{in}\  B_{r_{2}}(x_{0}).
	\end{equation} }
	Meanwhile, taking $\ell_{1}\in\mathbb{I}$ such that 
	\begin{equation}\label{ineq10.0}
		\mathcal{M}_{\ell_{0},\iota_{0}}v^{\varepsilon}(x_{0})=v^{\varepsilon}_{\ell_{1},\iota_{0}}(x_{0} )+\vartheta_{\ell_{0},\ell_{1}},
	\end{equation}
	by \eqref{pc1} and \eqref{ineq9}, we get that
	$v^{\varepsilon}_{\ell_{0},\iota_{0}}-(v^{\varepsilon}_{\ell_{1},\iota_{0}}+\vartheta_{\ell_{0},\ell_{1}})
	=v^{\varepsilon}_{\ell_{0},\iota_{0}}-\mathcal{M}_{\ell_{0},\iota_{0}}v^{\varepsilon}\leq u^{\varepsilon}_{\ell_{0},\iota_{0}}-\mathcal{M}_{\ell_{0},\iota_{0}}u^{\varepsilon}\leq 0$ at $x_{0}$.
	If $v^{\varepsilon}_{\ell_{0},\iota_{0}}(x_{0})-\mathcal{M}_{\ell_{0},\iota_{0}}v^{\varepsilon}(x_{0})<0$, there exists a ball $B_{r_{3}}(x_{0})\subset\set$ such that $v^{\varepsilon}_{\ell_{0},\iota_{0}}-\mathcal{M}_{\ell_{0},\iota_{0}}v^{\varepsilon}<0$ in $B_{r_{3}}(x_{0})$. Moreover, from \eqref{pc1},
	\begin{equation}\label{ineq11}
		\begin{split}
			[c_{\iota_{0}}-\dif_{\ell_{0},\iota_{0}}]  v_{\ell_{0},\iota_{0}}^{\varepsilon}+\psi_{\varepsilon}(|\deri^{1} v_{\ell_{0},\iota_{0}}^{\varepsilon}|^{2}- g_{\iota_{0}}^{2})-h_{\iota_{0}}&=0,\\
			[c_{\iota_{0}}-\dif_{\ell_{0},\iota_{0}}]  u_{\ell_{0},\iota_{0}}^{\varepsilon}+\psi_{\varepsilon}(|\deri^{1} u_{\ell_{0},\iota_{0}}^{\varepsilon}|^{2}- g_{\iota_{0}}^{2})-h_{\iota_{0}}&\leq0,
		\end{split}
		\quad\text{in}\ B_{r_{3}}(x_{0}).
	\end{equation}
	Notice that $\psi_{\varepsilon}(|\deri^{1} u_{\ell_{0},\iota_{0}}^{\varepsilon}|^{2}- g_{\iota_{0}}^{2})-\psi_{\varepsilon}(|\deri^{1} v_{\ell_{0},\iota_{0}}^{\varepsilon}|^{2}- g_{\iota_{0}}^{2})$ is a continuous function in $\set$  due to $\partial_{i}u^{\varepsilon}_{\ell_{0},\iota_{0}},\partial_{i}v^{\varepsilon}_{\ell_{0},\iota_{0}}\in\hol^{0}(\set)$, which satisfies 
	$\psi_{\varepsilon}(|\deri^{1} u_{\ell_{0},\iota_{0}}^{\varepsilon}|^{2}- g_{\iota_{0}}^{2})-\psi_{\varepsilon}(|\deri^{1} v_{\ell_{0},\iota_{0}}^{\varepsilon}|^{2}- g_{\iota_{0}}^{2})=0$ at $x_{0}$, since $x_{0}$ is the point where $u^{\varepsilon}_{\ell_{0},\iota_{0}}-v^{\varepsilon}_{\ell_{0},\iota_{0}}$ attains its maximum. Meanwhile, by Bony's maximum principle (see \cite{lions}), it is known that for every $r\leq r_{4}$, with $r_{4}>0$ small enough,
	\begin{equation}\label{ineq11.0}
		\tr[a_{\iota_{0}}\deri^{2}[u^{\varepsilon}_{\ell_{0},\iota_{0}}-v^{\varepsilon}_{\ell_{0},\iota_{0}}]]\leq0, \quad\text{a.e. in}\ B_{r}(x_{0}). 
	\end{equation}
	So, from \eqref{ineq10.1}, \eqref{w3}, \eqref{ineq11} and \eqref{ineq11.0}, it yields that for every $r\leq \hat{r}\eqdef\min\{r_{1},r_{2},r_{3},r_{4}\}$,
	\begin{align*}
		0&\geq \tr[a_{\iota_{0}}\deri^{2}[u^{\varepsilon}_{\ell_{0},\iota_{0}}-v^{\varepsilon}_{\ell_{0},\iota_{0}}]]\notag\\
		&\geq  c_{\iota_{0}}[u^{\varepsilon}_{\ell_{0},\iota_{0}}-v^{\varepsilon}_{\ell_{0},\iota_{0}}]+\langle b_{\iota_{0}},\deri^{1}[u^{\varepsilon}_{\ell_{0},\iota_{0}}-v^{\varepsilon}_{\ell_{0},\iota_{0}}]\rangle\notag\\
		&\quad+\psi_{\varepsilon}(|\deri^{1} u_{\ell_{0},\iota_{0}}^{\varepsilon}|^{2}- g_{\iota_{0}}^{2})-\psi_{\varepsilon}(|\deri^{1} v_{\ell_{0},\iota_{0}}^{\varepsilon}|^{2}- g_{\iota_{0}}^{2})\\
		&\quad{+\sum_{\kappa\in\mathbb{I}\setminus\{\iota_{0}\}}q_{\ell_{0}}(\iota_{0},\kappa)\{u^{\varepsilon}_{\ell_{0},\iota_{0}}-v^{\varepsilon}_{\ell_{0},\iota_{0}}-[u^{\varepsilon}_{\ell_{0},\kappa}-v^{\varepsilon}_{\ell_{0},\kappa}]\}}\\
		&\geq \min_{x\in B_{r_{1}}(x_{0})}\{c_{\iota_{0}}(x)[u^{\varepsilon}_{\ell_{0},\iota_{0}}(x)-v^{\varepsilon}_{\ell_{0},\iota_{0}}(x)]\} +\langle b_{\iota_{0}},\deri^{1}[u^{\varepsilon}_{\ell_{0},\iota_{0}}-v^{\varepsilon}_{\ell_{0},\iota_{0}}]\rangle\notag\\
		&\quad+\psi_{\varepsilon}(|\deri^{1} u_{\ell_{0},\iota_{0}}^{\varepsilon}|^{2}- g_{\iota_{0}}^{2})-\psi_{\varepsilon}(|\deri^{1} v_{\ell_{0},\iota_{0}}^{\varepsilon}|^{2}- g_{\iota_{0}}^{2}),\quad\text{a.e. in $B_{r}(x_{0})$.}
	\end{align*}
	Then, 
	\begin{multline}
		\lim_{r\rightarrow0}\bigg\{\infess_{B_{r}(x_{0})}[\psi_{\varepsilon}(|\deri^{1} u_{\ell_{0},\iota_{0}}^{\varepsilon}|^{2}- g_{\iota_{0}}^{2})-\psi_{\varepsilon}(|\deri^{1} v_{\ell_{0},\iota_{0}}^{\varepsilon}|^{2}- g_{\iota_{0}}^{2})]\bigg\}\\
		<-\min_{x\in B_{r_{1}}(x_{0})}\{c_{\iota_{0}}(x)[u^{\varepsilon}_{\ell_{0},\iota_{0}}(x)-v^{\varepsilon}_{\ell_{0},\iota_{0}}(x)]\}<0.
	\end{multline}
	That means $\psi_{\varepsilon}(|\deri^{1} u_{\ell_{0},\iota_{0}}^{\varepsilon}|^{2}- g_{\iota_{0}}^{2})-\psi_{\varepsilon}(|\deri^{1} v_{\ell_{0},\iota_{0}}^{\varepsilon}|^{2}- g_{\iota_{0}}^{2})$ is not continuous at $x_{0}$ which is a contradiction. Thus, 
	\begin{equation}\label{ineq12.01}
		0=v^{\varepsilon}_{\ell_{0},\iota_{0}}-(v^{\varepsilon}_{\ell_{1},\iota_{0}}+\vartheta_{\ell_{0},\ell_{1}})
		=v^{\varepsilon}_{\ell_{0},\iota_{0}}-\mathcal{M}_{\ell_{0},\iota_{0}}v^{\varepsilon}\leq u^{\varepsilon}_{\ell_{0},\iota_{0}}-\mathcal{M}_{\ell_{0},\iota_{0}}u^{\varepsilon}\leq 0\quad\text{at}\ x_{0}.
	\end{equation}
	It implies that 
	\begin{align}
		&u^{\varepsilon}_{\ell_{1},\iota_{0}}({x_{0} })-v^{\varepsilon}_{\ell_{1},\iota_{0}}({x_{0} })\geq u^{\varepsilon}_{\ell_{0},\iota_{0}}(x_{0})-v^{\varepsilon}_{\ell_{0},\iota_{0}}(x_{0})>0,\label{ineq13}\\
		&v^{\varepsilon}_{\ell_{0},\iota_{0}}(x_{0})=v^{\varepsilon}_{\ell_{1},\iota_{0}}({x_{0} })+\vartheta_{\ell_{0},\ell_{1}}.\notag%\label{ineq14}
	\end{align}
	By \eqref{ineq9} and \eqref{ineq13}, we have that $u^{\varepsilon}_{\ell_{1},\iota_{0}}-v^{\varepsilon}_{\ell_{1},\iota_{0}}$ attains its maximum  at $ x_{0} \in\set$, whose value agrees with $u^{\varepsilon}_{\ell_{0},\iota_{0}}(x_{0})-v^{\varepsilon}_{\ell_{0},\iota_{0}}(x_{0})$.  Then, {replacing} $u^{\varepsilon}_{\ell_{0},\iota_{0}}-v^{\varepsilon}_{\ell_{0},\iota_{0}}$ by $u^{\varepsilon}_{\ell_{1},\iota_{0}}-v^{\varepsilon}_{\ell_{1},\iota_{0}}$ above and   repeating the same arguments seen in  \eqref{ineq10.0}--\eqref{ineq12.01}, we get that there is a regime $\ell_{2}\in\mathbb{I}$ such that 
	\begin{align*}
		&u^{\varepsilon}_{\ell_{2},\iota_{0}}(x_{0})-v^{\varepsilon}_{\ell_{2},\iota_{0}}(x_{0})=u^{\varepsilon}_{\ell_{1},\iota_{0}}(x_{0})-v^{\varepsilon}_{\ell_{1},\iota_{0}}(x_{0})=u^{\varepsilon}_{\ell_{0},\iota_{0}}(x_{0})-v^{\varepsilon}_{\ell_{0},\iota_{0}}(x_{0})>0,\notag\\
		&v^{\varepsilon}_{\ell_{1},\iota_{0}}(x_{0})=v^{\varepsilon}_{\ell_{2},\iota_{0}}(x_{0})+\vartheta_{\ell_{1},\ell_{2}}.%\label{ineq15}
	\end{align*}
	Recursively, we obtain a sequence of regimes $\{\ell_{i}\}_{i\geq0}$ such that
	\begin{align}
		&u^{\varepsilon}_{\ell_{i},\iota_{0}}(x_{0})-v^{\varepsilon}_{\ell_{i},\iota_{0}}(x_{0})=u^{\varepsilon}_{\ell_{i-1},\iota_{0}}(x_{0})-v^{\varepsilon}_{\ell_{i-1},\iota_{0}}(x_{0})=\cdots=u^{\varepsilon}_{\ell_{0},\iota_{0}}(x_{0})-v^{\varepsilon}_{\ell_{0},\iota_{0}}(x_{0})>0,\notag\\
		&v^{\varepsilon}_{\ell_{i},\iota_{0}}(x_{0})=v^{\varepsilon}_{\ell_{i+1},\iota_{0}}(x_{0})+\vartheta_{\ell_{i},\ell_{i+1}}.\label{ineq16}
	\end{align}
	Since  $\mathbb{M}$ is finite,  {there is} a regime $\ell'$  that will appear infinitely often in  $\{\ell_{i}\}_{i\geq0}$. Let  $\ell_{\tilde{n}}=\ell'$, for some $\tilde{n}>1$. After $\hat{n}$ steps, the regime $\ell'$ reappears, i.e. $\ell_{\tilde{n}+\hat{n}}=\ell'$. Then, by \eqref{ineq16}, we get
	\begin{equation}\label{ineq17}
		v^{\varepsilon}_{\ell',\iota_{0}}(x_{0})=v^{\varepsilon}_{\ell',\iota_{0}}(x_{0})+\vartheta_{\ell',\ell_{\tilde{n}+1}}+\vartheta_{\ell_{\tilde{n}+1},\ell_{\tilde{n}+2}}+\cdots+\vartheta_{\ell_{\tilde{n}+\hat{n}-1},\ell'}.
	\end{equation} 
	Notice that \eqref{ineq17} contradicts the assumption that there is no loop of zero cost (see Eq. \eqref{l1}). From here we conclude that \eqref{ineq10} must occur. Taking $v^{\varepsilon}-u^{\varepsilon}$ and proceeding in the same way as before, it follows that for each $(\ell,\iota)\in\mathbb{M}\times\mathbb{I}$, $v^{\varepsilon}_{\ell,\iota}-u^{\varepsilon}_{\ell,\iota}\leq 0$  in $ \set $, and hence  we conclude that the solution $u^{\varepsilon}$ to the HJB equation \eqref{pc1} is unique.
\end{proof}	  

\subsection{Proof of Proposition \ref{M1}}\label{proof5}

\begin{proof}[Proof of Proposition \ref{M1}. Existence]

	Now, let $(\ell,\iota)\in\mathbb{M}\times\mathbb{I}$ be fixed. Since $u^{\varepsilon_{n}}_{\ell,\iota}$ is the unique  {strong} solution to the HJB equation \eqref{pc1} when $\varepsilon=\varepsilon_{n}$,  {which belongs to $\hol^{0}(\overline{\set})$},  it follows that  for each $\ell'\in\mathbb{M}\setminus\{\ell\}$,  $u^{\varepsilon_{n}}_{\ell,\iota}-(u^{\varepsilon_{n}}_{\ell',\iota} +\vartheta_{\ell,\ell'})\leq u^{\varepsilon_{n}}_{\ell,\iota}-\mathcal{M}_{\ell,\iota}u^{\varepsilon_{n}}\leq0$ in $\set$. From here and \eqref{econv1}, it yields that  $u_{\ell,\iota}-u_{\ell',\iota} -\vartheta_{\ell,\ell'}\leq0$ in $\set$. Then, $u_{\ell,\iota}-\mathcal{M}_{\ell,\iota}u\leq0$, in $\set$. Also, we know that $[c_{\iota}-\dif_{\ell,\iota}]  u_{\ell,\iota}^{\varepsilon_{n}}+ \psi_{\varepsilon_{n}}(|\deri^{1} u_{\ell,\iota}^{\varepsilon_{n}}|^{2}- g_{\iota}^{2})\leq h_{\iota}$ a.e. in $\set$. Then,
	\begin{align}
		&0\leq\psi_{\varepsilon_{n}}(|\deri^{1} u_{\ell,\iota}^{\varepsilon_{n}}|^{2}- g_{\iota}^{2})\leq h_{\iota}-[c_{\iota}-\dif_{\ell,\iota}]  u_{\ell,\iota}^{\varepsilon_{n}},\quad \text{a.e.}\ \text{in}\  \set.\label{eneq2.1}%\\
		%&[c_{\iota}-\dif_{\ell,\iota}]  u_{\ell,\iota}^{\varepsilon_{n}}-h_{\iota}\leq 0,\quad \text{a.e.}\ \text{in}\  \set.\label{eneq2.2} 
	\end{align}
	Consequently, by \eqref{a4}, \eqref{ineqe1} and \eqref{eneq2.1}, there exists a positive constant $C_{6}=C_{6}(d,\Lambda,\alpha)$ such that $0\leq\int_{B_{r}}\psi_{\varepsilon_{n}}(|\deri^{1} u_{\ell,\iota}^{\varepsilon_{n}}|^{2}- g_{\iota}^{2})\varpi\der x\leq\int_{B_{r}} \{h_{\iota}-[c_{\iota}-\dif_{\ell,\iota}]  u_{\ell,\iota}^{\varepsilon_{n}}\}\varpi\der x\leq C_{6}$ for each  $\varpi\in\mathcal{B}(B_{r})$, with $\mathcal{B}(\cdot)$ as in \eqref{c1}. Thus, using definition of $\psi_{\varepsilon}$ (see  \eqref{p12.1})  {and since $|\deri^{1}u_{\ell,\iota}^{\varepsilon_{n}}|^{2}-g^{2}_{\iota}$ is continuous in $\set$}, we have that for each $B_{r}\subset \set$, there exists $\varepsilon'\in(0,1)$ small enough, such that for all $\varepsilon_{n}\leq \varepsilon'$,  $|\deri^{1}u^{\varepsilon_{n}}_{\ell,\iota}|-g_{\iota}\leq0$ in $B_{r}$. Then, since \eqref{econv1} holds, it follows that $|\deri^{1} u_{\ell,\iota}|\leq g_{\iota}$ in $\set$. From \eqref{eneq2.1}, we get $\int_{B_{r}}\big\{[c_{\iota}-\dif_{\ell,\iota}]  u_{\ell,\iota}^{\varepsilon_{n}}-h_{\iota}\big\}\varpi\der x\leq 0$, for each $\varpi\in \mathcal{B}(B_{r})$. From here and \eqref{econv1}, we obtain that $[c_{\iota}-\dif_{\ell,\iota}]  u_{\ell,\iota}-h_{\iota}\leq 0$ a.e. in $\set$. Therefore, by the seen previously, 
	\begin{equation}\label{eineq4}
		\max\left\{[c_{\iota}-\dif_{\ell,\iota}]  u_{\ell,\iota}-h_{\iota},|\deri^{1} u_{\ell,\iota}|- g_{\iota},u_{\ell,\iota}-\mathcal{M}_{\ell,\iota}u\right\}\leq 0 ,\quad \text{a.e. in}\ \set.
	\end{equation}
	Without loss of generality we assume that $
	u_{\ell,\iota}(x^{*})-\mathcal{M}_{\ell,\iota}u(x^{*})<0$, for some $x^{*}\in\set$. Otherwise, the equality is satisfied in \eqref{eineq4}. Then, for each $\ell'\in\mathbb{M}\setminus\{\ell\}$,  $u_{\ell,\iota}-(u_{\ell',\iota} +\vartheta_{\ell,\ell'})\leq u_{\ell,\iota}-\mathcal{M}_{\ell,\iota}u<0$ at $x^{*}$.  There exists a ball $B_{r_{1}}(x^{*})\subset\set$ such that 
	\begin{equation}\label{eineq5}
		u_{\ell,\iota}-(u_{\ell',\iota} +\vartheta_{\ell,\ell'})\leq u_{\ell,\iota}-\mathcal{M}_{\ell,\iota}u<0,\quad\text{in}\ B_{r_{1}}(x^{*})
	\end{equation}
	due to the continuity of $u_{\ell,\iota}-u_{\ell',\iota}$ in $\overline{\set}$. Now, consider that $|\deri^{1}u_{\ell,\iota}|-g_{\iota}<0$ for some $x^{*}_{1}\in B_{r_{1}}(x^{*})$. Otherwise, the equality is also satisfied in \eqref{eineq4}. By continuity of $|\deri^{1}u_{\ell,\iota}|-g_{\iota}$, it yields that for some $B_{r_{2}}(x^{*}_{1})\subset\set$, $|\deri^{1}u_{\ell,\iota}|-g_{\iota}<0$ in $B_{r_{2}}(x^{*}_{1})$. From here, using \eqref{econv1}, \eqref{eineq5} and  taking $\mathcal{N}\eqdef B_{r_{1}}(x^{*})\cap B_{r_{2}}(x^{*}_{1})$,  it can be verified that there exists an $\varepsilon'\in(0,1)$ small enough, such that for each $\varepsilon_{n}\leq\varepsilon'$, $|\deri^{1}u^{\varepsilon_{n}}_{\ell,\iota}|-g_{\iota}<0$ and $ u^{\varepsilon_{n}}_{\ell,\iota}-\mathcal{M}_{\ell,\iota}u^{\varepsilon_{n}}<0$ in $\mathcal{N}$. Thus, $[c_{\iota}-\dif_{\ell,\iota}]  u_{\ell,\iota}^{\varepsilon_{n}}= h_{\iota}$ a.e. in $\mathcal{N}$, since $u^{\varepsilon_{n}}$ is the unique solution to the HJB equation \eqref{pc1}, when $\varepsilon=\varepsilon_{n}$. Then, $\int_{\mathcal{N}}\big\{[c_{\iota}-\dif_{\ell,\iota}]  u_{\ell,\iota}^{\varepsilon_{n}}-h_{\iota}\big\}\varpi\der x= 0$, for each $\varpi\in\mathcal{B}(\mathcal{N})$.  Hence, letting $\varepsilon_{n}\rightarrow0$ and using again \eqref{econv1}, we get that  $u=(u_{1},\dots,u_{m})$ is a solution to the HJB equation  \eqref{esd5}.
\end{proof}

\begin{proof}[Proof of Theorem \ref{M1}. Uniqueness] Suppose that  
	$$u=(u_{\ell,\iota})_{(\ell,\iota)\in\mathbb{M}\times\mathbb{I}}\quad \text{and}\quad v=(v_{\ell,\iota})_{(\ell,\iota)\in\mathbb{M}\times\mathbb{I}}$$ 
	are two solutions to the HJB equation  \eqref{esd5} whose components belong to  $\hol^{0,1}(\overline{\set})\cap\sob^{2,\infty}_{\loc}(\set)$. Take $(x_{0},\ell_{0})\in\overline{\set}\times\mathbb{I}$ such that
	\begin{equation}\label{eineq9}
		u_{\ell_{0},\iota_{0}}(x_{0})-v_{\ell_{0},\iota_{0}}(x_{0})=\sup_{(x,\ell,\iota)\in\overline{\set}\times\mathbb{M}\times\mathbb{I}}\{u_{\ell,\iota}(x)-v_{\ell,\iota}(x)\}.
	\end{equation}
	As before (see Subsection \ref{prop1}), we only need to verify that 
	\begin{equation}\label{eineq10}
		u_{\ell_{0},\iota_{0}}-v_{\ell_{0},\iota_{0}}\leq0,\quad \text{at}\  x_{0}\in\set.
	\end{equation}
	Assume that $u_{\ell_{0},\iota_{0}}-v_{\ell_{0},\iota_{0}}>0$ at $x_{0}$. Then, there exists a  ball $B_{r_{1}}(x_{0})\subset\set$ such that $c_{\iota_{0}}[u_{\ell_{0},\iota_{0}}-v_{\ell_{0},\iota_{0}}]\geq\min_{x\in B_{r_{1}}(x_{0})}\{c_{\iota_{0}}(x)[u_{\ell_{0},\iota_{0}}(x)-v_{\ell_{0},\iota_{0}}(x)]\}>0$ in $B_{r_{1}}(x_{0})$ due to the continuity of  $u_{\ell_{0},\iota_{0}}-v_{\ell_{0},\iota_{0}}$ in $\overline{\set}$ and that $c_{\iota_{0}}>0$ in $\set$. Meanwhile, from \eqref{eineq9}, $v_{\ell_{0},\iota_{0}}-\mathcal{M}_{\ell_{0},\iota_{0}}v\leq u_{\ell_{0},\iota_{0}}-\mathcal{M}_{\ell_{0},\iota_{0}}u\leq 0$ at $x_{0}$. If $v_{\ell_{0},\iota_{0}}-\mathcal{M}_{\ell_{0},\iota_{0}}v<0$ at $x_{0}$, there exists a ball $B_{r_{2}}(x_{0})\subset\set$ such that $v_{\ell_{0},\iota_{0}}-\mathcal{M}_{\ell_{0},\iota_{0}}v<0$ in $B_{r_{2}}(x_{0})$. Now, consider the auxiliary function $f_{\varrho}\eqdef u_{\ell_{0},\iota_{0}}-v_{\ell_{0},\iota_{0}}-\varrho u_{\ell_{0},\iota_{0}}$, with $\varrho\in(0,1)$. Notice that $f_{\varrho}={-\varrho f_{\iota}<0}$ on $\partial\set$, for $\varrho\in(0,1)$, and  
	\begin{equation}\label{eineq11}
		f_{\varrho}\uparrow  u_{\ell_{0},\iota_{0}}-v_{\ell_{0},\iota_{0}}\ \text{uniformly in}\ \set,\ \text{when}\  \varrho\downarrow 0. 
	\end{equation}
	Besides, there is a $\varrho'\in(0,1)$ small enough such that $\sup_{x\in  B_{r_{2}}(x_{0})}\{f_{\varrho}(x)\}>0$ for all $\varrho\in(0,\varrho')$ because of $u_{\ell_{0},\iota_{0}}-v_{\ell_{0},\iota_{0}}>0$ at $x_{0}$. By \eqref{eineq9} and \eqref{eineq11}, there exists $\hat{\varrho}\in(0,\varrho')$ small enough such that $f_{\hat{\varrho}}$ has a local maximum  at  $x_{\hat{\varrho}}\in B_{r_{1}}(x_{0})\cap B_{r_{2}}(x_{0})$. It follows that $|\deri^{1}v_{\ell_{0},\iota_{0}}(x_{\hat{\varrho}})|=[1-{\hat{\varrho}}]|\deri^{1}u_{\ell_{0},\iota_{0}}(x_{\hat{\varrho}})|<|\deri^{1}u_{\ell_{0},\iota_{0}}(x_{\hat{\varrho}})|\leq g_{\iota}(x_{\hat{\varrho}}).$ Thus, there exists a ball $B_{r_{3}}(x_{\hat{\varrho}})\subset B_{r_{1}}(x_{0})\cap B_{r_{2}}(x_{0})$ such that $[c_{\iota_{0}}-\dif_{\ell_{0},\iota_{0}}]  v_{\ell_{0},\iota_{0}}-h_{\iota_{0}}=0$ and  $[c_{\iota_{0}}-\dif_{\ell_{0},\iota_{0}}]  u_{\ell_{0},\iota_{0}}-h_{\iota_{0}}\leq0$ in $B_{r_{3}}(x_{\hat{{\varrho}}})$.  Then, by Bony's maximum principle, we have that $0\geq\lim_{r\rightarrow0}\big\{\infess_{B_{r}(x_{\hat{\varrho}})} \tr[a_{\iota_{0}}\deri^{2}f_{\hat{\varrho}}]\big\}\geq c_{\iota_{0}}f_{\hat{\varrho}}+\hat{\varrho} h_{\iota_{0}}$ at $x_{\hat{\varrho}}$,  which is a contradiction because of $\hat{\varrho} h_{\iota_{0}}\geq0$, $f_{\hat{\varrho}}>0$ and $c_{\iota_{0}}>0$ at $x_{\hat{\varrho}}$. We conclude that, $0=v_{\ell_{0},\iota_{0}}-\mathcal{M}_{\ell_{0},\iota_{0}}v\leq u_{\ell_{0},\iota_{0}}-\mathcal{M}_{\ell_{0},\iota_{0}}u\leq 0$ at $x_{0}$. Using the same arguments seen in  the proof of uniqueness of the solution to the  HJB equation \eqref{esd5}  (see Subsection \ref{prop1}), it can be verified that there is a contradiction with the assumption that there is no loop of  zero cost  (see Eq. \eqref{l1}). From here we conclude that \eqref{eineq10} must occur. Taking $v-u$ and proceeding in the same way as  before, we see $u$ is the unique solution to the HJB equation \eqref{esd5}.
\end{proof}

\end{document}